\documentclass[10pt,english, reqnor]{amsart}

\usepackage[T1]{fontenc}

\usepackage{datetime}
\usepackage[T1]{fontenc}

\usepackage{chngcntr}

\usepackage{url,xspace,adjustbox}

\usepackage{footmisc}

\usepackage{amssymb}
\usepackage[alphabetic]{amsrefs}

\usepackage{mathrsfs,stmaryrd}
\usepackage{yfonts, bbm}
\usepackage{enumitem}
\usepackage{float}

\usepackage{hyperref}
\hypersetup{
  colorlinks   = true, 
  urlcolor     = blue, 
  linkcolor    = blue, 
  citecolor   = green 
}

\usepackage{tikz}
\usetikzlibrary{shapes,arrows,calc,matrix}
\usepackage{tikz-cd}
\usepackage{mathdots}

\def\rtwidth{0.6}
\def\rtheight{0.4}

\usepackage{todonotes}
\usepackage{color}

\usepackage{amsmath}
\usepackage{lipsum}
\usepackage{setspace}

\theoremstyle{plain}
      \newtheorem{theorem}{Theorem}[section]
      \newtheorem{proposition}[subsubsection]{Proposition}
      \newtheorem{lemma}[subsubsection]{Lemma}
      \newtheorem{corollary}[subsubsection]{Corollary}

      \theoremstyle{definition}

      \newtheorem{remark}[theorem]{Remark}

      \newtheorem*{conjecture*}{Conjecture}
      
\counterwithin{table}{section}

\newcommand{\ZZ}{{\mathbb{Z}}}

\newcommand{\CC}{{\mathbb{C}}}

\DeclareMathOperator{\GL}{GL}

\DeclareMathOperator{\SL}{SL}

\DeclareMathOperator{\SO}{SO}

\DeclareMathOperator{\codim}{codim}

\newcommand{\Lgroup}[1]{{\hskip-2 pt \,^L\hskip-1pt{#1}}}
\newcommand{\dualgroup}[1]{{\widehat{#1}}}

\newcommand{\Frob}{{\operatorname{\mathfrak{f}_F}}}

\DeclareMathOperator{\id}{id}

\DeclareMathOperator{\trace}{trace}
\DeclareMathOperator{\rank}{rank}

\DeclareMathOperator{\image}{im}

\DeclareMathOperator{\Ad}{Ad}

\DeclareMathOperator{\Sym}{Sym}
\DeclareMathOperator{\sing}{sing}

\newcommand{\abs}[1]{{\vert #1 \vert}}
\newcommand{\ceq}{{\, :=\, }}
\newcommand{\tq}{{\ \vert\ }}
\newcommand{\iso}{{\ \cong\ }}

\DeclareMathOperator{\diag}{diag}

\newcommand{\Perv}{\mathsf{Per}}

\newcommand{\Loc}{\mathsf{Loc}}
\newcommand{\Rep}{\mathsf{Rep}}

\newcommand{\Ev}{\operatorname{\mathsf{E}\hskip-1pt\mathsf{v}}}
\newcommand{\Evs}{\operatorname{\mathsf{E}\hskip-1pt\mathsf{v}\hskip-1pt\mathsf{s}}}

\newcommand{\IC}{{\mathcal{I\hskip-1pt C}}}

\newcommand{\RPhi}{{\mathsf{R}\hskip-0.5pt\Phi}}

\makeatletter
\newcommand{\labitem}[2]{
\def\@itemlabel{\textbf{#1}}
\item
\def\@currentlabel{#1}\label{#2}}
\makeatother

\newcommand{\1}{{\mathbbm{1}}}

\newcommand{\Ind}{\text{Ind}}

\newcommand{\A}{\mathbb{A}}

\newcommand{\Aut}{\text{Aut}}

\newcommand{\Int}{\text{Int}}

\newcommand{\Hom}{\text{Hom}}

\newcommand{\KPair}[2]{( \, #1\, \vert\, #2\, )}

\newcommand{\Lie}{\operatorname{Lie}}

\newcommand\KS{{\operatorname{KS}}}
\newcommand{\St}{\text{St}}
\newcommand{\gl}{\mathfrak{gl}}
\newcommand{\Mat}{\operatorname{Mat}}

\newcommand{\ABV}{{\mbox{\raisebox{1pt}{\scalebox{0.5}{$\mathrm{ABV}$}}}}}

\tikzset{
  symbol/.style={
    draw=none,
    every to/.append style={
      edge node={node [sloped, allow upside down, auto=false]{$#1$}}}
  }
}

\def\code#1{\texttt{#1}}



\date{\today}                                           

\begin{document}

\title[Appearance of the Kashiwara-Saito singularity]{Appearance of the Kashiwara-Saito singularity in the representation theory of $p$-adic $\GL_{16}$}

\author[C. Cunningham]{Clifton Cunningham}
\address{Department of Mathematics and Statistics, University of Calgary, 
2500 University Drive NW, 
Calgary, Alberta, 
T2N 1N4, 
Canada}
\email{clifton@automorphic.ca}
\thanks{Clifton Cunningham gratefully acknowledges the support of NSERC Discovery Grant RGPIN-2015-06103 and additional support from the Pacific Institute for the Mathematical Sciences (PIMS)}

\author[A. Fiori]{Andrew Fiori}
\address{Department of Mathematics and Statistics, University of Lethbridge,
4401 University Drive,
Lethbridge, Alberta,
T1K 3M4,
Canada}
\email{andrew.fiori@uleth.ca}
\thanks{Andrew Fiori thanks and acknowledges the University of Lethbridge for their financial support as well as the support of NSERC Discovery Grant RGPIN-2020-05316.}

\author[N. Kitt]{Nicole Kitt}
\address{University of Waterloo, Department of Pure Mathematics,
200 University Avenue West
Waterloo, Ontario, N2L 3G1, Canada
}
\email{nkitt@uwaterloo.ca}
\thanks{Nicole Kitt would like to thank and acknowledge support from the NSERC Undergraduate Summer Research Award (USRA) program and from the University of Calgary PURE program.}

\begin{abstract}
In 1993 David Vogan proposed a basis for the vector space of stable distributions on $p$-adic groups using the microlocal geometry of moduli spaces of Langlands parameters. In the case of general linear groups, distribution characters of irreducible admissible representations, taken up to equivalence, form a basis for the vector space of stable distributions. In this paper we show that these two bases, one putative, cannot be equal. Specifically, we use the Kashiwara-Saito singularity to find a non-Arthur type irreducible admissible representation of $p$-adic $\mathop{GL}_{16}$ whose ABV-packet, as defined in \cite{CFMMX}, contains exactly one other representation. Consequently, for general linear groups, while all A-packets are singletons, some ABV-packets are not.
In the course of the proof of this result we strengthen the main result concerning the Kashiwara-Saito singularity in \cite{KS}. 
\end{abstract}
\maketitle


\section{Introduction}

\subsection{Background}\label{ssec:background}

The local Langlands correspondence for a $p$-adic group $G(F)$ is generally stated, or more precisely, conjectured, as a  bijection between sets: on the one hand, the set $\Pi(G(F))$ of equivalence classes  of irreducible admissible representations of $G(F)$; and on the other hand, the set $\Xi(\Lgroup{G})$ of equivalence classes of pairs $\xi = (\phi,\epsilon)$, called enhanced Langlands parameters, where $\phi : W'_F \to \Lgroup{G}$ is a Langlands parameter and $\epsilon$ is an irreducible representation of a finite group attached to $\phi$. This bijection should satisfy a list of properties, largely derived from compatibility with class field theory and the principle of functoriality together with certain normalizing choices.
For general linear groups over $p$-adic fields, the proof of the local Langlands correspondence was concluded by \cite{Harris-Taylor}, building on earlier work including \cite{Z2}, and for quasisplit classical groups of $p$-adic fields the correspondence is known by work of Arthur \cite{Arthur:book}, again relying on the work of a large community of mathematicians.

Since $\Pi(G(F))$ is the set of equivalence classes of simple objects in a category, namely, the category of smooth representations of $G(F)$, it is natural to ask if we can identify $\Phi(G(F))$ with the set of equivalence classes of simple objects in a category too and then seek a relation between these categories that would recover the bijection above on simple objects. In 1993 David Vogan observed that simple objects in the category $\Perv_{\dualgroup{G}}(X_\lambda)$ of equivariant perverse sheaves on the moduli space $X_\lambda$ of Langlands parameters with a common ``infinitesimal parameter'' $\lambda$, are naturally identified with pairs $(\phi,\epsilon)$ as above, when the finite group attached to $\phi$ is properly interpreted and when the corresponding representations $\pi$ are enlarged to include $G(F)$ and its pure inner forms.
When viewed from this perspective, the local Langlands correspondence takes the form of a bijection between the set $\Xi_\lambda(\Lgroup{G})$ of equivalence classes of objects in $\Perv_{\dualgroup{G}}(X_\lambda)$ and the set of simple objects $\Pi_\chi(G/F)$ in the category $\Rep_\chi(G/F)$ of smooth representations of $G(F)$, and its pure inner forms, with matching ``infinitesimal character'' $\chi$. There is considerable evidence that the categories $\Rep_\chi(G/F)$ and $\Perv_{\dualgroup{G}}(X_\lambda)$ are more closely related than simply by a bijection between their simple objects  -- at the very least we expect that their derived categories should be equivalent -- and this evidence motivates a closer study of $X_\lambda$.

The Langlands correspondence determines a partition of $\Pi(G(F))$ into finite sets $\Pi_\phi(G(F))$, called {\it L}-packets, consisting of those irreducible admissible representations $\pi$ that correspond to pairs $(\phi,\epsilon)\in \Xi(\Lgroup{G})$ for this fixed $\phi$. 
In cases where they have been defined, local Arthur packets $\Pi_\psi(G(F))$, also known as A-packets, are also finite sets of irreducible representations, where $\psi$ is an Arthur parameter, but these packets are not disjoint.
The representations that appear in Arthur packets are said to be of Arthur type and form a class that is conjectured to be the same as unitary admissible irreducible representations.
Although Arthur packets were introduced in the late 1980s, it was not until \cite{Arthur:book} that a purely local characterization of local Arthur packets was available, for quasi-split classical groups and in partial form for inner twists of classical groups.
Since then this work has been extended to unitary groups.
Every A-packet $\Pi_\psi(G(F))$ contains a distinguished {\it L}-packet $\Pi_{\phi_\psi}(G(F))$ through a simple injection $\psi \mapsto \phi_\phi$ from Arthur parameters to Langlands parameters of Arthur type.
We refer to $\Pi_{\psi}(G(F))\setminus \Pi_{\phi_\psi}(G(F))$ as the {\it corona} of $\Pi_{\phi_\psi}(G(F))$.
For $G=\GL_n$ the equality $\Pi_{\psi}(G(F)) =  \Pi_{\phi_\psi}(G(F))$ is a consequence of \cite{Arthur:unipotent-motivation}, for all $A$-parameters $\psi$.

Vogan's perspective on the local Langlands correspondence suggests a beautiful geometric approach to A-packets and in fact defines a corona for every {\it L}-packet, not just those of Arthur type \cite{Vogan:Langlands}.
In particular, this perspective attaches, to every irreducible admissible representation $\pi$ of $G(F)$, a simple equivariant perverse sheaf $\mathcal{P}(\pi)$ on the moduli space of Langlands parameters with the same infinitesimal parameter as $\pi$.
Following an approach similar to \cite{ABV}, Vogan then defines the ABV-packet $\Pi_\phi^\ABV(G(F))$ for any Langlands parameter $\phi$, as the set of irreducible representations of $G(F)$ for which the characteristic cycles of $\mathcal{P}(\pi)$ contains the conormal bundle of orbit of $\phi$ in this moduli space.
This notion is revisited in \cite{CFMMX} where it is cast in terms of vanishing cycles and a functor denoted there by $\Evs$; 
specifically, we define
 \[  \Pi_\phi^\ABV(G(F)) = \{ \pi \tq \Evs_{C_\phi}( \mathcal{P}(\pi) ) \neq 0 \}, \]
where $C_\phi$ is the orbit of $\phi$ in this moduli space.
We wish to emphasize that this definition of $\Pi_\phi^\ABV(G(F))$ assumes the Langlands correspondence and is therefore unconditional for general linear groups and quasi-split classical groups. Moreover, since the Langlands correspondence is known for unipotent representations of  simple, adjoint $p$-adic groups by \cite{Lusztig:classification-unipotent} and for all connected reductive $p$-adic groups by \cite{Solleveld:LLC-unipotent}, the definition of $\Pi_\phi^\ABV(G(F))$ is also unconditional in these cases.

Apocryphally, ABV-packets are rather difficult to calculate; see \cite{BST}*{Introduction}, for example. 
 In \cite{CFZ:unipotent} and \cite{CFZ:cubics} we found the ABV-packets for all unipotent representations of the exceptional group $G_2$, building on techniques developed in \cite{CFMMX}, where we also calculated over 50 examples of ABV-packets for split classical groups and their pure inner forms. In this paper we continue our exploration by finding a curious ABV-packet for $\GL_{16}(F)$, using a method that lends itself to algorithmic implementation and symbolic computation.


\subsection{Main result}

Since all A-packets for general linear groups are L-packets \cite{Arthur:unipotent-motivation}, and since all L-packets for general linear groups are singletons, and since ABV-packets are expected to generalize A-packets \cite{CFMMX}, it is reasonable to ask:
{\it Are all ABV-packets singletons for $p$-adic general linear groups?} 
In this paper we show that the answer is no by exhibiting an unramified Langlands parameter $\phi_\KS$ for $\GL_{16}(F)$ for which the ABV-packet has exactly two representations. 
More precisely, we find two unipotent, non-tempered, irreducible representations, $\pi_\KS$ and $\pi_\psi$, such that
\begin{equation}\label{eqn:intro-main}
\Pi^\ABV_{\phi_\KS}(\GL_{16}(F)) = \{ \pi_\KS, \pi_\psi \},
\end{equation}
where $\phi_\KS$ is the Langlands parameter for $\pi_\KS$. 
We show that $\pi_\KS$ is not of Arthur type while its coronal representation, $\pi_\psi$, is of Arthur type.
The theory of ABV-packets matches the representations $\pi_\KS$ and $\pi_\psi$ with constant and quadratic characters, respectively, of a group $A^\ABV_{\phi_\KS}$, studied in Section~\ref{ssec:generic}.

After extensive searching in lower rank general linear groups, we expect that $\pi_\KS$ is the simplest example of an admissible representation of a $p$-adic general linear group with a corona.
Using the results of this paper, it follows that there are coronal representations, or equivalently, non-singleton ABV-packets, of $\GL_n(F)$ for all $n\geq16$; this can be done by manufacturing an infinitesimal parameter $\lambda$ for $\GL_n(F)$ for which the moduli space $V_\lambda$ is precisely the one studied in this paper for $\GL_{16}$. 
Many examples of coronal representations for orthogonal groups appear in \cite{CFMMX}, starting with $\SO_5(F)$, and for the exceptional group $G_2(F)$ in \cite{CFZ:cubics} and \cite{CFZ:unipotent}.

\subsection{Geometry}

The main result, above, is a consequence of a geometric calculation at the heart of this paper.
Before describing that calculation, we begin by recalling a fact from geometry.
Let $V$ be a prehomogenous vector space, that is, let $V$ be a finite-dimensional vector space over a field, equipped with the action of an affine algebraic group $H$ for which there is a dense Zariski-open $H$-orbit. Now let $\Lambda$ be the conormal variety for $V$ with this action, which is to say, set
\[
\Lambda \ceq \{ (x,y)\in V\times V^* \tq [x,y] =0\},
\]
where $[\ ,\ ] : V\times V^*\to \Lie H$ is the moment map for $V$ with $H$-action.
For every $x\in V$, the fibre $\Lambda_x$ of the projection $\Lambda \to V$ is also a finite-dimensional vector space, equipped with the natural action of the affine algebraic group $Z_{H}(x)$, so it is natural to ask: Does $\Lambda_x$ have a dense Zariski-open $Z_{H}(x)$-orbit, or in other words, {\it is $\Lambda$ a bundle of prehomogeneous vector spaces?} 
In this paper we show that even when $V$ is the representation variety of a quiver of type $A_5$, the answer is no.
This paper is largely an exploration of the consequence of this fact for $p$-adic representation theory, as it relates directly to Arthur parameters and A-packets.

The geometric calculation at the heart of this paper concerns the moduli space $X$ of Langlands parameters with a common infinitesimal parameter, $\lambda$, in the sense of Section~\ref{ssec:V}. 
We find a prehomogeneous vector space $V$ with $H\ceq Z_{\dualgroup{G}}(\lambda)$-action such that 
\[
X = \dualgroup{G}\times_{H} V
\]
and we find the elements $x_\psi, x_\KS\in V$ that correspond to the Langlands parameters $\phi_\psi$ and $\phi_\KS$, respectively; again see Section~\ref{ssec:V}. 
In this paper we find that $\Lambda$, the conormal bundle to $V$, is not a bundle of prehomogenous vector spaces by showing that $\Lambda_{x_\KS}$ is not a prehomogeneous vector space. 

We make a fairly detailed study of the $H$-orbits $C_\psi$ of $x_\psi$ and $C_\KS$ of $x_\KS$. 
We find that $C_\psi$ has co-dimension $8$ in $V$ and that $C_\KS$ has co-dimension $8$ in the closure $\overline{C}_\psi$ of $C_\psi$ in $V$.
We show that the conormal bundle
\[
\Lambda_{C_\KS} = \{ (x,y)\in V\times V^* \tq x\in C_\KS,\ [x,y]=0 \}
\]
has no open $H$-orbit.
We use this fact to conclude that $\phi_\KS$ is not of Arthur type. 

In fact, we find a connected, dense open $H$-stable sub-bundle $\Lambda^\text{gen}_{C_\KS} \subset \Lambda_{C_\KS}$ whose equivariant fundamental group, $A^\ABV_{\phi_\KS}$, is non-trivial. 
This is another surprise, since the groups $A^\ABV_{\phi}$ are trivial for all Langlands parameters $\phi$ of Arthur type for general linear groups. For classical groups, it is easy to find examples when the group $A^\ABV_{\phi}$ is non-trivial; many appear in \cite{CFMMX}*{Part II}.

The main geometric result of this paper refers to the functor 
\[
\Evs_{C_\KS} : \Perv_H(V) \to \Loc_{H}(\Lambda^\text{gen}_{C_\KS})
\]
which is a special case of the functor defined in \cite{CFMMX}*{Section 7.9}. 
Theorem~\ref{thm:geometric}, implies that the rank of the local system $\Evs_{C_\KS} \IC(\1_{C})$ is $1$ for $C= C_\KS$ and $C= C_\psi$ and otherwise the local system $\Evs_{C_\KS} \IC(\1_{C})$ is $0$.
In fact, Theorem~\ref{thm:geometric} gives more information, since we find the local systems  $\Evs_{C_\KS} \IC(\1_{C_\KS})$ and $ \Evs_{C_\KS} \IC(\1_{C_\psi})$: the former is constant, the latter is quadratic.
Theorem~\ref{thm:geometric} implies Corollary~\ref{cor:main}.

This paper is almost entirely an exploration of the consequences for representation theory of geometric facts: not only are the characteristic cycles for conormal bundles of quiver representation varieties of type $A_n$ sometimes {\it reducible}, contrary to a ``hope'' of Lusztig appearing in \cite{Lusztig:Quivers}*{13.7}, but the microlocal equivariant fundamental groups of these bundles are sometimes {\it non-trivial} and, moreover, the microlocal vanishing cycles of equivariant perverse sheaves on these quiver representation varieties are sometimes non-constant, corresponding to {\it non-trivial representations} of these fundamental groups. 

\subsection{Consequences}

The singularity we are studying was introduced in \cite[Section 7]{KS}.
In \cite[Theorem 7.2.2]{KS}, it is shown that structure of the singularity is fundamentally the
same as the singularity which other sources, such as  \cite{BST}*{Section 3.4} and \cite{W}, refer to as the Kashiwara-Saito singularity.
Specifically, a transverse slice through $\overline{C}_\psi$ at $x_\KS$ reveals the Kashiwara-Saito singularity in the moduli space of Langlands parameters with the same infinitesimal parameter as $\psi$.

It follows from Theorem~\ref{thm:geometric} that the characteristic cycles of the equivariant perverse sheaf $\IC(\1_{C_\psi})$ contains ${\Lambda_{C_\KS}}$ and ${\Lambda_{C_\psi}}$.
That result is proved in \cite{KS}*{Theorem 7.2.1} where it is used to provide a counter example to a conjecture of Lusztig \cite{Lusztig:Quivers}*{13.7}.
Our proof of Theorem~\ref{thm:geometric} gives an independent proof of \cite{KS}*{Theorem 7.2.1} using techniques developed here and in \cite{CFMMX}.
Moreover, Theorem~\ref{thm:geometric} is stronger than \cite{KS}*{Theorem 7.2.1}, since the latter does not establish that the multiplicities of ${\Lambda_{C_\KS}}$ and ${\Lambda_{C_\psi}}$ in the characteristic cycles of $\IC(\1_{C_\psi})$ are both $1$, which also follows from Theorem~\ref{thm:geometric}.
See Remark~\ref{rem:KS} for more on the relation between Theorem~\ref{thm:geometric} and \cite{KS}*{Theorem 7.2.1}.

The main result of this paper also has the following interpretation related to another conjecture by Vogan. 
To state this conjecture, we must build distributions on $G(F)$ from ABV-packets.
These distributions are defined using a function  $\Pi^\ABV_\phi(G(F))\to \Rep(A^\ABV_\phi)$ that attaches irreducible representations of a finite group $A^\ABV_\phi$ attached to the Langlands parameter $\phi$.
The group $A^\ABV_\phi$ is the equivariant fundamental group of an open, dense variety in the conormal bundle above $\phi$ in the moduli space of Langlands parameters.
Using this, in \cite{CFMMX} we define {\it packet coefficients} $\langle a, \pi\rangle$, for $a\in A^\ABV_\phi$ and $\pi\in \Pi^\ABV_\phi(G(F))$.
Consider the distribution
\begin{equation}\label{eqn:intro-Theta}
\Theta_{\phi} \ceq e(G) \sum_{\pi\in \Pi^\ABV_\phi(G(F))} (-1)^{\dim(\phi)-\dim(\phi_\pi)} \trace \langle 1 , \pi\rangle\ \Theta_\pi,
\end{equation}
where $e(G)$ is the Kottwitz invariant of the group $G$ and $\dim(\phi_\pi)-\dim(\phi)$ is the relative dimension of the orbit of $\phi$ in the closure of the orbit of the Langlands parameter $\phi_\pi$ for $\pi$ in the moduli space of Langlands parameters; see \cite{CFMMX} and \cite{CFZ:cubics} and \cite{CFZ:unipotent} for the theory and examples of these distributions.
Vogan's \cite{Vogan:Langlands}*{Conjecture 8.15'}, revisited in \cite{CFMMX}*{Conjecture 2}, posits that the distributions $\Theta_{\phi}$ are stable and, moreover, form a basis for the space of stable distributions on $G(F)$, letting $\phi$ run over conjugacy classes of Langlands parameters.
It follows from the main result of this paper, Corollary~\ref{cor:main}, that
\begin{equation}\label{eqn:intro-Theta-2}
\Theta_{\phi_\KS} = \Theta_{\pi_\KS} + \Theta_{\pi_\psi}.
\end{equation}

Since $\Theta_{\pi_\KS}$ and $\Theta_{\pi_\psi} $ are stable distributions on $\GL_{16}(F)$, this means that {\it Vogan's putative basis for the vector space of stable distributions is not the same as the basis given by characters of irreducible admissible representations, for general linear groups}. This does not in any way contradict Vogan's conjecture. We further remark that, since $\Theta_{\phi_\psi} = \Theta_{\pi_\psi} $, the transition matrix comparing Vogan's basis for this part of the vector space of stable distributions is unipotent.

\subsection*{Relation to the Archimedean case}
For the Archimedean fields $F$, the theory of ABV-packets was developed in \cite{ABV} and \cite{Vogan:Langlands}.
We are grateful to the authors of \cite{BST} for pointing out to us that $\GL_8(F)$ admits an ABV-packet with two representations and, consequently, $\GL_n(F)$ admits non-singleton ABV-packets for all $n\geq8$.
As in the present work, the genesis of the ABV-packet with two representations for $\GL_8(F)$ is the Kashiwara-Saito singularity in a variety that appears in \cite{ABV} and \cite{Vogan:Langlands};
this variety is not a moduli space of Langlands parameters. 
Even when we pass to the case the non-singleton packets for $\GL_{16}(F)$, it is not clear to the authors of this paper how, or if, the representations of these Archimedean groups are related to the representations of the $p$-adic $\GL_{16}$ appearing in this paper.

\subsection*{Acknowledgements}
We thank Reginald Lybbert for establishing the foundation for some of the computations in this paper during an undergraduate summer research project (NSERC USRA) with the first author. 
We thank Sarah Dijols, Mishty Ray, Elijah Thompson, Geoff Vooys, Bin Xu and Qing Zhang for helpful conversations and all the other members of the \href{http://www.automorphic.ca}{Voganish Project} for listening to numerous seminars while we were working out this material. 

\section{The Kashiwara-Saito representation of $\GL_{16}(F)$}\label{sec:statement-main}

In this section we introduce the Kashiwara-Saito representation $\pi_\KS$ and its coronal representation $\pi_\psi$. We give the standard modules for these two representations, as well as their {\it L}-parameters. While both $\pi_\KS$ and $\pi_\psi$ are unipotent representations, neither are tempered and only $\pi_\psi$ is of Arthur type, as we will see in Section~\ref{sec:KSisnotArthur}.

\subsection{A self-dual representation of $\GL_{16}(F)$ of Arthur type }\label{ssec:psi}

Let $F$ be a $p$-adic field and consider the Arthur parameter for $\GL_8(F)$ defined by
\[
\begin{array}{rcl}
\psi_0 : W_F \times \SL_{2}(\CC) \times \SL_{2}(\CC) &\mathop{\longrightarrow}\limits^{\psi_0}& \GL_{8}(\CC)\\
(w,x,y) &\mapsto& \Sym^3(x)\otimes \Sym^1(y),
\end{array}
\]
where $\Sym^n$ is the irreducible representation of $\SL_{2}(\CC)$ of dimension $n+1$.
Let $\phi_0$ be the Langlands parameter defined by $\phi_{0}(w,x)\ceq \psi_0(w,x,\diag(\abs{w}^{1/2},\abs{w}^{-1/2}))$, so
\[
\phi_0(w,x) = \abs{w}^{1/2}\Sym^3(x) \oplus  \abs{w}^{-1/2}\Sym^3(x).
\]
Now consider the dual Arthur parameter ${\hat \psi}_0$ for $\GL_8(F)$ defined by
\[
{\hat \psi}_0(w,x,y) = \Sym^1(x)\otimes \Sym^3(y).
\]
The Langlands parameter for ${\hat \psi}_0$ is 
\[
\begin{array}{rcl}
{\hat \phi}_0(w,x) &=& \abs{w}^{3/2} \Sym^1(x) \oplus \abs{w}^{1/2} \Sym^1(x)\\
& & \oplus\ \abs{w}^{-1/2} \Sym^1(x) \oplus \abs{w}^{-3/2} \Sym^1(x).
\end{array}
\]

Now let $\psi$ be the Arthur parameter for $\GL_{16}(F)$ defined by
$
\psi = r \circ (\psi_0\boxtimes{\hat\psi}_0),
$
where $r : \GL_{8}\times\GL_{8} \to \GL_{16}$ is a natural diagonal block representation; we can also write
$
\psi = \psi_0\oplus {\hat\psi}_0,
$
so
\[
\psi(w,x,y) = \Sym^3(x)\otimes\Sym^1(y) \oplus \Sym^1(x)\otimes\Sym^3(y).
\]
Then the Langlands parameter $\phi_\psi$ for $\psi$ is 
$
\phi_\psi = \phi_0 \oplus {\hat \phi}_0.
$
so
\[
\begin{array}{rcl}
\phi_\psi(w,x) &=& \abs{w}^{3/2}\Sym^1(x) \oplus \abs{w}^{1/2}\Sym^3(x) \oplus \abs{w}^{1/2}\Sym^1(x)\\
&&\oplus\ \abs{w}^{-1/2}\Sym^1(x)\oplus \abs{w}^{-1/2}\Sym^3(x) \oplus \abs{w}^{-3/2}\Sym^1(x).
\end{array}
\]
Let $\pi_\psi$ be the irreducible admissible representation of $\GL_{16}(F)$ with Langlands parameter $\phi_\psi$.

To describe the standard module for $\pi_\psi$, set 
\[
M_\psi \ceq\GL_{2}\times\GL_{4}\times\GL_{2}\times\GL_{2}\times\GL_{4}\times\GL_{2}
\]
and set
\[
\begin{array}{rcl}
\sigma_\psi&\ceq&
\left(\abs{\det}^{3/2} \St_{\GL_2(F)}\right)\oplus
\left(\abs{\det}^{1/2} \St_{\GL_4(F)}\right)\\
&& \oplus\
\left(\abs{\det}^{1/2} \St_{\GL_2(F)}\right)\oplus
\left(\abs{\det}^{-1/2} \St_{\GL_2(F)}\right)\\
&& \oplus\
\left(\abs{\det}^{-1/2} \St_{\GL_4(F)}\right)\oplus
\left(\abs{\det}^{-3/2} \St_{\GL_2(F)}\right).
\end{array}
\]
Let $P_\psi$ be the standard parabolic subgroup of $\GL_{16}$ with Levi subgroup $M_\psi$. 
Then $\pi_\psi$ is the unique irreducible quotient of  
\[
\Ind_{P_\psi(F)}^{\GL_{16}(F)} (\sigma_\psi).
\]
Note that, by construction, the irreducible admissible representation $\pi_\psi$ is of Arthur type and equivalent to its Zelevinsky dual.

\subsection{The Kashiwara-Saito representation of $\GL_{16}(F)$}\label{ssec:KS}

Consider the Langlands parameter $\phi_1$ for $\GL_8(F)$ defined by
\[
\phi_1(w,x) \ceq \Sym^2(x) \oplus \abs{w}^{3/2}\Sym^1(x) \oplus \abs{w}^{-3/2}\Sym^1(x) \oplus \Sym^0(x).
\]
We apologize for writing $\Sym^0(x)$ for the trivial representation.
%
Set 
\[
\phi_\KS \ceq r\circ (\phi_1\boxtimes\phi_1) = \phi_1 \oplus \phi_1,
\]
for $ : \GL_8\times \GL_8\to \GL_{16}$.
Thus,
\[
\begin{array}{rcl}
\phi_\KS(w,x) 
&=& \oplus \abs{w}^{3/2}\left(\Sym^1(x)\oplus\Sym^1(x)\right)\\
&&\oplus\ 
\left(\Sym^2(x)\oplus \Sym^2(x)\right) \oplus (\Sym^0(x)\oplus\Sym^0(x))\\
&& \oplus\
\abs{w}^{-3/2}(\Sym^1(x)\oplus\Sym^1(x)).
\end{array}
\]
Let $\pi_\KS$ be the irreducible admissible representation of $\GL_{16}(F)$ with this Langlands parameter.
We refer to $\pi_\KS$ as the {\it Kashiwara-Saito representation} of $\GL_{16}(F)$; this representation is self-dual.

We find the standard module for $\pi_\KS$.
Consider the representation $\sigma_\KS$ of 
\[
M_\KS = (\GL_2\times\GL_2)\times(\GL_3\times\GL_3)\times(\GL_1\times\GL_1)\times(\GL_2\times\GL_2)
\]
defined by
\[
\begin{array}{rcl}
\sigma_\KS &\ceq&
\abs{\det}^{3/2}\left(\St_{\GL_2(F)}\oplus\St_{\GL_2(F)}\right)\\
&& \oplus\
\left(\St_{\GL_3(F)}\oplus\St_{\GL_3(F)}\right)\oplus\left(\St_{\GL_1(F)}\boxtimes\St_{\GL_1(F)}\right)\\
&& \oplus\
\abs{\det}^{-3/2} \left(\St_{\GL_2(F)}\oplus\St_{\GL_2(F)}\right).
\end{array}
\]
The standard module for $\pi_\KS$ is 
\[
\Ind_{P_\KS(F)}^{\GL_{16}(F)} (\sigma_\KS),
\]
where $P_\KS$ is the standard parabolic with Levi subgroup $M_\KS$.

\subsection{Multisegments}

Following \cite{Z2}*{Section 3.1}, a segment is a finite set 
\[
[\rho,\nu^k\rho]\ceq \{ \rho, \nu\rho,\nu^2\rho, \ldots , \nu^k\rho\},
\]
where $\rho$ is an (equivalence class of) irreducible cuspidal representation of $\GL_n(F)$, for some positive integer $n$, and $\nu$ is the character of $\GL_n(F)$ defined by $\nu(g)=\abs{\det(g)}$.
A multisegment $\underline{m} = \{ \Delta_1,\Delta_2,\ldots,\Delta_r \} $, is a multiset of segments, without fixing $n$.
As explained in \cite{Z2}*{Theorem 6.1}, there is a natural bijection between multisegments and (equivalence classes of) irreducible admissible representations of $\GL_n(F)$, allowing $n$ to range over all positive integers. 

After adapting the theory by consistently replacing irreducible submodules with irreducible quotients, the bijection of \cite{Z2}*{Theorem 6.1} attaches the following multisegments to the representations appearing in Sections~\ref{ssec:psi} and \ref{ssec:KS}: 
the multisegment for $\pi_\psi$ is
\[
\underline{m}_\psi 
=  \{[\nu^{1},\nu^{2}], [\nu^{-1},\nu^{2}], [\nu^{0},\nu^{1}], [\nu^{-1},\nu^{0}],[\nu^{-2},\nu^{1}], [\nu^{-2},\nu^{-1}]\},
\]
and this is self-dual;
the multisegment for $\pi_\KS$ is
\[
\underline{m}_\KS 
=  \{2[\nu^{1},\nu^{2}] ,  2[\nu^{-1},\nu^{1}], 2\{\nu^{0}\}, 2[\nu^{-2},\nu^{-1}] \},
\]
which is also self-dual.

\subsection{Moduli space of Langlands parameters}\label{ssec:V}

The infinitesimal parameters of $\pi_\psi$ and $\pi_\KS$, in the sense of \cite{CFMMX}*{Section 4.1}, are equal and henceforth denoted by $\lambda : W_F \to \GL_{16}(\CC)$, defined by
\[
\lambda(w) \ceq 2\abs{w}^{2} \oplus 4\abs{w}^{1} \oplus 4\abs{w}^{0} \oplus 4\abs{w}^{-1} \oplus 2\abs{w}^{-2},
\]
where the coefficients $2$, $4$, $4$, $4$ and $2$ denote multiplicities.
%
If we extend $\lambda$ trivially from $W_F$ to $W'_F$ and thus view $\lambda$ as a Langlands parameter for $\GL_{16}(F)$, the multisegment for $\lambda$ is
\[
\underline{m}_\lambda = 
\{ 2\{\nu^{2}\}, 4\{\nu^{1}\}, 4\{\nu^{0}\}, 4\{\nu^{-1}\}, 2\{\nu^{-2}\} \}.
\]

The calculation of $\Pi^\ABV_{\phi_\KS}(\GL_{16}(F))$ begins with the moduli space of Langlands parameters with infinitesimal parameter $\lambda$. We now describe that moduli space.
Following \cite{Vogan:Langlands}*{} and \cite{CFMMX}*{}, consider 
\[
V \ceq \{ x\in \gl_{16}(\CC) \tq \Ad(\lambda(w))x = \abs{w} x, \ \forall w\in W_F\}
\]
on which
\[
H := Z_{\GL_{16}(\CC)}(\lambda) \ceq \{ g\in \GL_{16}(\CC) \tq \lambda(w)g\lambda(w)^{-1} = x, \ \forall w\in W_F\}
\]
acts by conjugation in $\gl_{16}$.
Using \cite{CFMMX}*{Section 4.3}, we see that $V$ is a moduli space of unramified Langlands parameters $\phi : W'_F \to \GL_{16}(\CC)$ such that $\phi(w,\diag(\abs{w}^{1/2},\abs{w}^{-1/2})) = \lambda(w)$.  
Now
\[
X\ceq \GL_{16}(\CC)\times_{H} V
\]
is a moduli space of unramified Langlands parameters $\phi$ for $\GL_{16}(F)$ whose infinitesimal parameters $\GL_{16}(\CC)$-conjugate to $\lambda$.
Recall that {\it L}-parameters are $\GL_{16}(\CC)$-conjugacy classes of Langlands parameters and are therefore identified with $\GL_{16}(\CC)$-orbits $X$, or equivalently, with $H$-orbits in $V$.
Since the categories $\Perv_{\dualgroup{G}}(X)$ and $\Perv_{H}(V)$ are equivalent, as discussed in \cite{CFMMX}*{Section 4.5}, we henceforth dispense with $X$ and work exclusively with $V$.

We now describe $V$ in more detail.
Let $q$ be the cardinality of the residue field of $F$.
We fix an element $\Frob\in W_F$ such that $\abs{\Frob} = q$.
The semisimple element $\lambda(\Frob) \in \GL_{16}(\CC)$ has eigenvalues $q^{-2}, q^{-1}, q^{0}, q^{1}, q^{2}$ occurring with multiplicities $2,4,4,4,2$, respectively. 
We label these eigenvalues 
 $\lambda_0\ceq q^{-2}$, $\lambda_{1} \ceq q^{-1}$, $\lambda_{2}\ceq q^{0}$, $\lambda_{3}\ceq q^{1}$ and $\lambda_{4}\ceq q^{2}$.
Using this convention,
\[
V = \left\{ 
\begin{pmatrix}  
0 & x_4 & 0 & 0 & 0 \\
0 & 0 & x_3 & 0 & 0 \\
0 & 0 & 0 & x_2 & 0 \\
0 & 0 & 0 & 0 & x_1 \\
0 & 0 & 0 & 0 & 0 
\end{pmatrix}
\ \Bigg\vert\ 
\begin{array}{l}
x_4\in \Mat_{2,4}(\CC) \\
x_3, x_2\in \Mat_{4,4}(\CC) \\
x_1\in \Mat_{4,2}(\CC)
\end{array}
\right\}.
\]
If we write $E_{\lambda_i}$ for the eigenspace of $\lambda(\Frob)$ with eigenvalue $\lambda_i$, then
\[
V  = \Hom(E_{\lambda_3},E_{\lambda_4}) \times \Hom(E_{\lambda_2},E_{\lambda_3})\times  \Hom(E_{\lambda_1},E_{\lambda_2})\times \Hom(E_{\lambda_0},E_{\lambda_1}).
\]
In other words, $V$ is a representation variety for the quiver of type $A_5$:
\[
\begin{tikzcd}
\mathop{\bullet}\limits^{\lambda_0}  \arrow{r}{x_1} &  \mathop{\bullet}\limits^{\lambda_1}  \arrow{r}{x_2} &  \mathop{\bullet}\limits^{\lambda_2} \arrow{r}{x_3} & \mathop{\bullet}\limits^{\lambda_3} \arrow{r}{x_4} & \mathop{\bullet}\limits^{\lambda_4}.
\end{tikzcd}
\]
Consequently, for elements $x\in V$ we use the notation 
\[
x = (x_4,x_3,x_2,x_1) \in \Mat_{2,4}(\CC)\times \Mat_{4,4}(\CC)\times \Mat_{4,4}(\CC)\times \Mat_{4,2}(\CC).
\]
Similarly,
\[
H
= \Aut(E_{\lambda_4}) \times \Aut(E_{\lambda_3}) \times \Aut(E_{\lambda_2}) \times \Aut(E_{\lambda_1}) \times \Aut(E_{\lambda_0}).
\]
Let us use the notation
\[
h = (h_4,h_3,h_2,h_1,h_0) \in \GL_{2}(\CC) \times  \GL_{4}(\CC) \times \GL_{4}(\CC) \times \GL_{4}(\CC)\times \GL_{2}(\CC)
\]
for elements of $H$.
From this we see $\dim V = 48$ and $\dim H = 56$.
In this notation, the action $H\times V \to V$ is given by
\[
(h_{4},h_{3}, h_{2}, h_{1}, h_{0})\cdot(x_4, x_3,x_2,x_1) 
\ceq 
(h_4x_4h_3^{-1}, h_3x_3h_2^{-1}, h_2x_2h_1^{-1}, h_1x_1h_0^{-1} ).
\]
With this $H$-action, $V$ is a prehomogenous vector space; the open, dense $H$-orbit in $V$ is the set of $x\in V$ satisfying the conditions
\[ 
\begin{array}{llll}
\rank(x_4) =2, & \rank(x_3)=4, &\rank(x_2)=4, &\rank(x_1)=2,\\
\rank(x_4x_3) =2, & \rank(x_3x_2)=4, &\rank(x_2x_1)=2, & \\
\rank(x_4x_3x_2)=2, & \rank(x_3x_2x_1) =2, & & \\
\rank(x_4x_3x_2x_1)=2 .& & & 
\end{array}
\]

\subsection{{\it L}-parameters, rank triangles and multisegments}
\label{ssec:ranktriangles}

We refer to $\GL_{16}(\CC)$-conjugacy classes of Langlands parameters as $L$-parameter. 
By \cite{CFMMX}*{Section 4.3} we know that $L$-parameters with infinitesimal parameter $\lambda$ are in bijection with $H$-orbits in $V$. 
Recall that the open dense $H$-orbit in $V$ is determined by the ranks of the matrices $x_4$, $x_3$, $x_2$, $x_1$ and the ranks of all permitted products of these matrices.
More generally, every $H$-orbit $C$ in $V$ is determined by the values $r_{ij}$, for $1\leq j \leq i\leq 4$, where $r_{ii}\ceq \rank(x_i)$ and 
\[
r_{ij}\ceq\rank(x_i \circ x_{i-1} \circ \cdots \circ x_{j+1} \circ x_j), \qquad j<i,
\]
as $x$ ranges over $C$; these equations exactly describe $C$ as a variety. 
We arrange these ranks into a triangle to reflect the corresponding combinations of the $x_i$ and refer to this as a {\it rank triangle}:
\begin{center}
\begin{tikzpicture}
\node(A) at (0,0) {$m_4$};
\node(B) at (1,0) {$m_3$};
\node(C) at (2,0) {$m_2$};
\node(L) at (3,0) {$m_1$};
\node(M) at (4,0) {$m_0$};
\draw (0,-0.25) -- (4,-0.25);
\node(D) at (0.5,-0.5) {$r_{44}$};
\node(E) at (1.5,-0.5) {$r_{33}$};
\node(d) at (2.5,-0.5) {$r_{22}$};
\node(e) at (3.5,-0.5) {$r_{11}$};
\node(F) at (1,-1) {$r_{43}$};
\node(G) at (2,-1) {$r_{32}$};
\node(H) at (3,-1) {$r_{21}$};
\node(I) at (1.5,-1.5) {$r_{42}$};
\node(J) at (2.5,-1.5) {$r_{31}$};
\node(K) at (2,-2) {$r_{41}$};.
\end{tikzpicture}
\end{center}
From left to right, the values in the top row of the rank triangle correspond, respectively, to the multiplicities $m_i$ of the eigenvalues $\lambda_4,\lambda_3,\lambda_2,\lambda_1,$ and $\lambda_0$. 
The row below these eigenvalue multiplicities show ranks $r_{ii}$ of $x_i$, which are subject only to the condition that every $r_{ii}$ is less than or equal to the two eigenvalue multiplicities above it, so $r_{ij} \leq m_i$ and $r_{ij} \leq m_{i-1}$.
For $1\leq j<i \leq 4$, the ranks $r_{ij}$ are subject to the condition that every rank is less than or equal to the two ranks above it, so $r_{ij} \leq r_{i,j+1}$ and $r_{ij} \leq r_{i-1,j}$, also also the condition
\[
r_{ik} + r_{lj} \leq r_{ij} + r_{lk}, \qquad 1\leq j < k \leq l < i \leq 4.
\]
This last condition is a consequence of the Frobenius inequality from linear algebra: if $AB$, $BC$ and $ABC$ are defined then $\rank(AB) + \rank(BC) \leq \rank(ABC) + \rank(B)$.
The set of $H$-orbits in $V$ is naturally in bijection with rank triangles subject to these conditions.
Of the $1138$ $H$-orbits in $V$, we are mainly interested in $C_\KS$ and $C_\psi$; in Appendix~\ref{ssec:generalcomputations}, we find that $\dim C_\KS = 32$ and $\dim C_\psi=40$.

Let us say that the support of a multisegment is the multiset obtained by breaking each segment into a union of segments of length $1$.
We now explain how to pass from a rank triangle with eigenvalues for $\lambda_0 = q^{-2}$, $\lambda_1 = q^{-1}$, $\lambda_2 = q^{0}$, $\lambda_3 = q^{1}$ and $\lambda_4 = q^{2}$, with respective multiplicities (top row) $m_0=2$, $m_1=4$, $m_2=4$, $m_3=4$ and $m_4=2$ to a multisegment with support $\underline{m}_\lambda$. 
Define $e_i$ by $\lambda_i = q^{e_i}$; thus, $e_0=-2$, $e_1=-1$, $e_2=0$, $e_3=1$ and $e_4=2$.
 Henceforth, in this paper by a segment $\Delta$ we mean a (non-empty) consecutive sequence of integers
 $ \Delta = \{{j},{j+1},\ldots, {i-1},{i}\} =: [{j}, {i}]$, and by a multisegment $\underline{m}$ we mean a multiset of these segments.
To each single segment $\Delta = [e_{j},e_i]$ we associate the rank triangle $T_{\Delta}$ with
 \[ r_{l,k} = \begin{cases} 1, & j< k\leq l\leq i; \\ 0, & \text{otherwise} \end{cases} \qquad m_k = \begin{cases} 1, & j\leq k\leq i; \\ 0, & \text{otherwise}. \end{cases} \]
A multisegment $\underline{m}$ then determines the triangle 
 \[ T_{\underline{m}} = \sum_{\Delta\in \underline{m}} T_{\Delta}. \]
Conversely, every rank triangle determines a multisegment by the following inductive rule:
 \begin{enumerate}
 \item Set $\underline{m}:=\emptyset$.
  \item Let $r_{ij}$ be the lowest and most to the left non-zero entry in the rank triangle. Note that $j\leq i$.
  \item Add the segment $\Delta = [e_{j-1},e_i]$ to $\underline{m}$. 
  \item Decrease the quantity $r_{k l}$ by one for each value of $k,l$ with $j\leq l\leq k\leq i$.
  Decrease $m_k$ by $1$ for each $j-1\leq k\leq i$. That is, subtract $T_\Delta$ from the rank triangle.
  \item Repeat steps (ii) through (iv) until all $r_{ij}$ are zero.
  \item If any $m_k\ne 0$, add the singleton $\{ e_k\}$ with multiplicity $m_k$ to $\underline{m}$. 
  \end{enumerate}
 Now $\underline{m}$ is the multisegment determined by the rank triangle. These procedures, $\underline{m}\leftrightarrow  T_{\underline{m}}$ establish a bijection between rank triangles and multisegments.
 
Table~\ref{table:ranks} lists the rank triangles for the multisegments $\underline{m}_\psi$ and $\underline{m}_\KS$.
\begin{table}[tp]
\caption{Rank triangles for the $H$-orbits $C_\psi$ and $C_\KS$ in $V$ corresponding to the admissible representations $\pi_\psi$ and $\pi_\KS$.}
\label{table:ranks}
\begin{center}
$
\begin{array}{c}
C_\psi  \\
\begin{tikzpicture}
\node(A) at (0,0) {$2$};
\node(B) at (\rtwidth,0) {$4$};
\node(C) at (2*\rtwidth,0) {$4$};
\node(L) at (3*\rtwidth,0) {$4$};
\node(M) at (4*\rtwidth,0) {$2$};
\draw (0,-0.5*\rtheight) -- (4*\rtwidth,-0.5*\rtheight);
\node(D) at (0.5*\rtwidth,-\rtheight) {$2$};
\node(E) at (1.5*\rtwidth,-\rtheight) {$3$};
\node(d) at (2.5*\rtwidth,-\rtheight) {$3$};
\node(e) at (3.5*\rtwidth,-\rtheight) {$2$};
\node(F) at (\rtwidth,-2*\rtheight) {$1$};
\node(G) at (2*\rtwidth,-2*\rtheight) {$2$};
\node(H) at (3*\rtwidth,-2*\rtheight) {$1$};
\node(I) at (1.5*\rtwidth,-3*\rtheight) {$1$};
\node(J) at (2.5*\rtwidth,-3*\rtheight) {$1$};
\node(K) at (2*\rtwidth,-4*\rtheight) {$0$};
\end{tikzpicture}
\end{array}
\qquad
\begin{array}{c}
 C_\KS\\
 \begin{tikzpicture}
\node(A) at (0,0) {$2$};
\node(B) at (\rtwidth,0) {$4$};
\node(C) at (2*\rtwidth,0) {$4$};
\node(L) at (3*\rtwidth,0) {$4$};
\node(M) at (4*\rtwidth,0) {$2$};
\draw (0,-0.5*\rtheight) -- (4*\rtwidth,-0.5*\rtheight);
\node(D) at (0.5*\rtwidth,-\rtheight) {$2$};
\node(E) at (1.5*\rtwidth,-\rtheight) {$2$};
\node(d) at (2.5*\rtwidth,-\rtheight) {$2$};
\node(e) at (3.5*\rtwidth,-\rtheight) {$2$};
\node(F) at (\rtwidth,-2*\rtheight) {$0$};
\node(G) at (2*\rtwidth,-2*\rtheight) {$2$};
\node(H) at (3*\rtwidth,-2*\rtheight) {$0$};
\node(I) at (1.5*\rtwidth,-3*\rtheight) {$0$};
\node(J) at (2.5*\rtwidth,-3*\rtheight) {$0$};
\node(K) at (2*\rtwidth,-4*\rtheight) {$0$};
\end{tikzpicture}
\end{array}
$
\end{center}
\end{table}

The set of $H$-orbits in $V$ carries a partial order defined by the Zariski topology: $C\leq C'$ means $C \subseteq \overline{C'}$. This partial order can be read easily from the corresponding rank triangles as follows:
If the ranks for $C$ are $r_{ij}$ and $C'$ are $r'_{ij}$ then $C\leq C'$ if and only if $r_{ij}\leq r'_{ij}$ for all $i,j$.

\begin{remark}
In Section~\ref{sec:KSisnotArthur} we will show that $\phi_\KS$ is not of Arthur type.
Here we sketch a possible alternate argument: the bijection between $H$-orbits in $V$ and multisegments with the same support as $\underline{m}_\lambda$ can be used to deduce that the Langlands parameter $\phi_\KS$ is not Arthur type by directly inspecting the associated multisegment.
To see this, suppose $\phi_\KS$ is of Arthur type and suppose there is an Arthur parameter $\psi_\KS : W''_F \to \GL_{16}(\CC)$ of the form
\[ \psi_\KS(w,x,y) = \bigoplus_i \Sym^{a_i}(x)\otimes \Sym^{b_i}(y) \]
such that $\phi_\KS(w,x) = \psi_\KS(w,x,d_w)$ where $d_w = \diag(\abs{w}^{1/2},\abs{w}^{-1/2})$.
 We note that the Langlands parameter associated to $\Sym^{a}(x)\otimes \Sym^{b}(y)$ has multisegment of the form 
\begin{align*}
  \{ &\{-(a+b)/2+0,-(a+b)/2+1,\ldots,-(a+b)/2+a+0\},\\& \{-(a+b)/2+1,-(a+b)/2+2,\ldots,-(a+b)/2+a+1\},\\&\ldots,\\&\{-(a+b)/2+b,-(a+b)/2+b+1,\ldots,-(a+b)/2+a+b\}\}
\end{align*}
and so multisegments for Langlands parameters of Arthur type must be unions of multisegments of this shape.
By inspection, $\underline{m}_{\KS}$ is not of this sort, so $\psi_\KS$ does not exist.
\end{remark}

\subsection{Relation to the Kashiwara-Saito singularity}\label{ssec:KSsingularity}

Representations $\pi_\psi$ and $\pi_\KS$ are reverse-engineered from the Kashiwara-Saito singularity, as we now explain.
In \cite{KS}*{Section 7}, Kashiwara and Saito provide a counter-example to a conjecture articulated in \cite{Lusztig:Quivers}*{Section 13.7}. The counter-example concerns the quiver representation variety $V$ of type $A_5$ appearing in Section~\ref{ssec:V}. They specify two elements of $b,b'\in V$, both given explicitly. The element $b'$ satisfies the rank conditions for $C_\KS$, while the element $b$  satisfies the rank conditions for $C_\psi$; see \cite{KS}*{Lemma~7.2.2}.
We pick $x_\KS = (x_4,x_3,x_2,x_1)\in C_\KS$ given by 
\begin{equation}\label{xKS}
x_4 =
\begin{pmatrix}
0 & 1 
\end{pmatrix}, \quad
x_3 = 
\begin{pmatrix}
1 & 0\\
0 & 0
\end{pmatrix}, \quad
x_2 =
\begin{pmatrix}
0 & 1\\
0 & 0
\end{pmatrix}, \quad
x_1 =
\begin{pmatrix}
1 \\ 0
\end{pmatrix},
\end{equation}
where $1$ is the $2\times 2$ identity matrix, so it follows that $x_\KS$ is $H$-conjugate to $b'$.
The representations $\pi_\psi$ and $\pi_\KS$ are defined exactly so that their Langlands parameters correspond to the two orbits $C_\psi$ and $C_\KS$, respectively, in the quiver representation variety.
To see this, observe that $\phi_\KS(1,e)$ is a unipotent element of $\GL_{16}(\CC)$ and it lies in the unipotent orbit corresponding to the partition $16=3+2+2+1+3+2+2+1$, which is the same partition that classifies the nilpotent orbit of $\mathfrak{gl}_{16}(\CC)$ that $x_\KS$ occupies. 
Using \cite{CFMMX}*{Proposition 4.2}, it follows that $C_\KS$ is the $H$-conjugacy class of $V$ matching the $\GL_{16}(\CC)$-conjugacy class of $\phi_\KS$, which is the {\it L}-parameter for $\pi_\KS$.
In the same way, $\phi_\psi(1,e)$ is in the same conjugacy class as $\exp x_\psi$ as elements of $\mathfrak{gl}_{16}(\CC)$, classified by the partition $16=4+4+2+2+2+2$, so $C_\psi$ is the $H$-conjugacy class of $V$ matching the $\GL_{16}(\CC)$-conjugacy class of $\phi_\psi$, which is the {\it L}-parameter for $\pi_\psi$.

\section{Geometry of the Kashiwara-Saito representation}\label{sec:KSisnotArthur}

\subsection{The conormal bundle}\label{ssec:Conormalbundle}

Using the Killing form for $\mathfrak{gl}_{16}(\CC)$, the dual variety $V^*$ may be realized as
\[
V^* \ceq \{ x\in \mathfrak{gl}_{16}(\CC) \tq \Ad(\lambda(\Frob))x = q^{-1} x \},
\]
which, in turn, is naturally identified with
\[
V^*  
= \Hom(E_{\lambda_1},E_{\lambda_0}) \times  \Hom(E_{\lambda_2},E_{\lambda_1})\times  \Hom(E_{\lambda_3},E_{\lambda_2})\times  \Hom(E_{\lambda_4},E_{\lambda_3}).
\]
Then $V^*$ is also a representation variety for the quiver of type $A_5$:
\[
\begin{tikzcd}
\mathop{\bullet}\limits^{\lambda_0} & \arrow[swap]{l}{y_1}  \mathop{\bullet}\limits^{\lambda_1} & \arrow[swap]{l}{y_2}  \mathop{\bullet}\limits^{\lambda_2} & \arrow[swap]{l}{y_3}  \mathop{\bullet}\limits^{\lambda_3} & \arrow[swap]{l}{y_4}  \mathop{\bullet}\limits^{\lambda_4}.
\end{tikzcd}
\]
We use the notation:
\[
y = (y_1,y_2,y_3,y_4) \in \Mat_{2,4}(\CC)\times \Mat_{4,4}(\CC)\times \Mat_{4,4}(\CC)\times \Mat_{4,2}(\CC) 
\]
for elements of $V^*$.
In this notation the action $H\times V^* \to V^*$ is given by
\[
(h_{4},h_{3}, h_{2}, h_{1}, h_{0})\cdot(y_1, y_2, y_3, y_4)  \ceq (h_0 y_1h_1^{-1}, h_1 y_2 h_2^{-1}, h_2 y_3 h_3^{-1}, h_3 y_4 h_4^{-1}).
\]
We note that $V^*$ is a prehomogeneous vector space for this action. 

The cotangent variety for $V$ is
\[
T^*(V) = V\times V^* \iso \left\{ 
\begin{pmatrix}  
0 & x_4 & 0 & 0 & 0 \\
y_4 & 0 & x_3 & 0 & 0 \\
0 & y_3 & 0 & x_2 & 0 \\
0 & 0 & y_2 & 0 & x_1 \\
0 & 0 & 0 & y_1 & 0 
\end{pmatrix}
\ \Bigg\vert\ 
\begin{array}{l}
x_4,y_1\in \Mat_{2,4}(\CC) \\
x_3, x_2, y_2,y_3 \in \Mat_{4,4}(\CC) \\
x_1, y_4\in \Mat_{4,2}(\CC)
\end{array}
\right\}.
\]
We use the following notation for elements of $T^*(V)$:
\begin{equation*}
(x,y) = (x_4,x_3,x_2,x_1,y_1,y_2,y_3,y_4) \in T^*(V),
\end{equation*}
with $x_i$ and $y_j$ as above.

The conormal variety, $\Lambda \subset T^*(V)$, is defined as the kernel of $[\ ,  ] : T^*(V) \to \Lie H $ given by
\begin{equation}\label{eqn:bracket}
[x, y] \ceq (x_4y_4, x_3 y_3 - y_4 x_4,  x_2 y_2 - y_3 x_3, x_1 y_1- y_2 x_2, -y_1x_1) .
\end{equation}
As the notation suggests, this is nothing more than the Lie bracket of $x$ and $y$ viewed as elements of $\gl_{16}(\CC)$.
Thus,
\[
\Lambda =
 \left\{ 
\begin{pmatrix}  
0 & x_4 & 0 & 0 & 0 \\
y_4 & 0 & x_3 & 0 & 0 \\
0 & y_3 & 0 & x_2 & 0 \\
0 & 0 & y_2 & 0 & x_1 \\
0 & 0 & 0 & y_1 & 0 
\end{pmatrix}
\in T^*(V)
\ \Bigg\vert\ 
\begin{array}{rl} 
x_4y_4 &= 0\\
x_3 y_3 &= y_4 x_4 \\
x_2 y_2 &= y_3 x_3 \\ 
x_1 y_1 &= y_2 x_2 \\ 
0 &= y_1x_1
\end{array} 
\right\}.
\]
Observe that $\Lambda$ is the quiver representation variety for
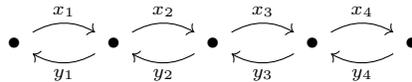
\[
\begin{tikzcd}
& 
\mathop{\bullet}\limits^{} \arrow[bend left=30]{r}{x_1} & \arrow[bend left=30]{l}{y_1} \mathop{\bullet}\limits^{} \arrow[bend left=30]{r}{x_2} & \arrow[bend left=30]{l}{y_2} \mathop{\bullet}\limits^{} \arrow[bend left=30]{r}{x_3} & \arrow[bend left=30]{l}{y_3} \mathop{\bullet}\limits^{} \arrow[bend left=30]{r}{x_4} & \arrow[bend left=30]{l} {y_4}\mathop{\bullet}\limits^{} 
 & 
\end{tikzcd}
\]
with the relations specified above.

By \cite{CFMMX}*{Proposition 6.3} for each orbit $C$, 
\[ \Lambda_C \ceq \{ (x,y) \in \Lambda \ \vert\ x\in C \} \] 
is the conormal bundle to $C$.
For  $x\in C$, we write 
\[ \Lambda_x =   \{ y\in V^* \ \vert\  (x,y) \in \Lambda \} \]
 for the fibre of the bundle $\Lambda_C$.
By identifying $V^{**} \simeq V$, for $C^*$ an $H$-orbit in $V^*$ we may write
 \[ \Lambda_{C^*} = \{ (x,y) \in \Lambda \ \vert\ y\in C^* \}, \]
where $\Lambda_{C^\ast}$ is the conormal bundle to $C^*$.

\subsection{The Kashiwara-Saito representation is not of Arthur type}\label{ssec:KSisnotArthur}

Let $G$ be a connected, reductive algebraic group over a $p$-adic field $F$.
A Langlands parameter $\phi : W'_F\to \Lgroup{G}$ is said to be of Arthur type if there is an Arthur parameter $\psi : W''_F\to \Lgroup{G}$ such that $\phi(w,x)= \psi(w,x,d_w)$ where $d_w = \diag(\abs{w}^{1/2},\abs{w}^{-1/2})$. 
An irreducible admissible representation $\pi$ of $G(F)$ is said to be of Arthur type if there is an Arthur parameter $\psi : W''_F\to \Lgroup{G}$ such that $\pi\in \Pi_\psi(G(F))$.
For general linear groups $G$, since A-packets are {\it L}-packets, these two notions agree: $\pi$ is of Arthur type if and only if its Langlands parameter $\phi_\pi$ is of Arthur type.
In general, it is not true that the Langlands parameters of Arthur type representations are of Arthur type.
In fact, in  \cite{CFMMX} we provide an example of a representation of Arthur type whose Langlands parameter is not of Arthur type; see \cite{CFZ:unipotent}*{Remark 2.6} for an explanation of this example.

\begin{theorem}\label{thm:KSisnotArthur}
The conormal bundle $\Lambda_{C_\KS}$ does not have an open $H$-orbit and, 
consequently, the Kashiwara-Saito representation $\pi_\KS$ is not of Arthur type.
Its coronal representation, $\pi_\psi$, is of Arthur type.
\end{theorem}

\begin{proof}
It is clear that $\pi_\psi$ is of Arthur type since, in Section~\ref{ssec:psi}, we exhibited an Arthur parameter $\psi$ whose Langlands parameter $\phi_\psi$ is the Langlands parameter for $\pi_\psi$.
We will prove that $\pi_\KS$ is not of Arthur type by contradiction.

First we must recall some notions. 
For every $H$-orbit $C$ in $V$, define
\begin{equation}
\Lambda^\text{reg}_{C} \ceq \Lambda_{C} \setminus \mathop{\bigcup}\limits_{C'< C}\overline{\Lambda_{C'}}.
\end{equation}
If $(x,y)\in \Lambda^\text{reg}_{C}$ we say $y$ is a regular conormal vector to $x\in C$.
In \cite{CFMMX}*{Section 6.5} we introduce the variety $\Lambda^\text{sreg}$ of {\it strongly regular} elements of $\Lambda$ defined by the property: $(x,y)\in \Lambda$ is strongly regular if the $H$-orbit $\mathcal{O}_H(x,y)\subseteq \Lambda$ is open and dense in $\Lambda_{C}$, where $C$ is the $H$-orbit of $x$ in $V$.
It follows that $\Lambda^\text{sreg}_x$ is non-empty if and only if $\Lambda_x$ is a prehomogeneous vector space for the $Z_{H}(x)$-action. 
In \cite{CFMMX}*{Proposition 6.8} we show 
\[
\Lambda^\text{sreg} \subseteq \Lambda^\text{reg}.
\]

For this proof, we find it convenient to introduce, for any $H$-orbit $C$ in $V$, the variety 
\begin{equation}\label{eqn:Lambda'}
\Lambda'_{C}:=\{(x,y)\in \Lambda_C  \tq y\in C^*\}
\end{equation}
situated between the regular part of the conormal variety and the conormal variety itself, above $C$:
\begin{equation*}
\Lambda^\text{reg}_{C} \subseteq \Lambda'_{C} \subseteq  \Lambda_{C}.
\end{equation*}

Now, for a contradiction, suppose $\pi_\KS$ is of Arthur type, with Arthur parameter $\psi_\KS: W''_F \to \GL_{16}(\CC)$ where $W''_F \ceq W_F \times \SL_2(\CC)\times\SL_2(\CC)$; then $\phi_\KS(w,x) = \psi_\KS(w,x,d_w)$ where $d_w = \diag(\abs{w}^{1/2},\abs{w}^{-1/2})$. 
Note that $\exp x_{\KS} = \psi_\KS(1,e,1) \in V$ where $e = \left(\begin{smallmatrix} 1 & 1 \\ 0 & 1 \end{smallmatrix}\right)$. 
Now define $y_\KS \in V^*$ by $\exp y_\KS = \psi_\KS(1,1,f)$ where $f = \left(\begin{smallmatrix} 1 & 0 \\ 1 & 1 \end{smallmatrix}\right)$; see \cite{CFMMX}*{Proposition 4.2}.
By \cite{CFMMX}*{Proposition 6.9}, $(x_\KS,y_\KS)\in \Lambda^\text{sreg}_{C_\KS}$, which means the $H$-orbit of $(x_\KS,y_\KS)$ is open and dense in $\Lambda_{C_\KS}$.
By \cite{CFMMX}*{Proposition 6.8}, $y_\KS \in C^*_{\KS}$, so $(x_\KS,y_\KS)\in \Lambda'_{C_\KS}$ and the $H$-orbit of $(x_\KS,y_\KS)$ is open and dense in $\Lambda_{C_\KS}$.
By Proposition~\ref{prop:codim2}, the $H$-orbit of any $(x,y)\in \Lambda'_{C_\KS}$ has codimension at least $2$ in $\Lambda_{C_\KS}$, so the $H$-orbit of $(x_\KS,y_\KS)$ is not open and dense in $\Lambda_{C_\KS}$. (As a consequence, $\Lambda^\text{sreg}_{C_\KS}$ is empty.) This contradiction proves the theorem. 
\end{proof}

Recall that $\Lambda_{C}\rightarrow C$ is a bundle and $\mathcal{O}_H(x,y)$ is a subbundle of $\Lambda_{C}$. Also recall the dense open subbundle $\Lambda'_{C}\rightarrow C$ from Equation~\eqref{eqn:Lambda'}. For any $x\in C$, the fibres of these bundles are
\[
\Lambda_x = \{y\in V^\ast \tq [x,y]=0\}
\qquad\text{and}\qquad
\Lambda'_x = \{y\in C^* \tq [x,y]=0\}.
\] 
Then $\Lambda'_x$ is a dense, locally closed subvariety of the vector space $\Lambda_x$.

\begin{lemma}\label{lemma:abcd1}
If $x\in C_\KS$ then
\[
 \Lambda'_{x}
 \iso \left\{ (a,b,c,d)\in \Mat_{2,2}(\CC)^{\times 4}
 \ \bigg\vert\   \begin{array}{c} \det(ac)\neq 0, \\ \det(ac+bd)\neq 0 \end{array}
 \right\}.
\]
In particular, $\dim  \Lambda'_{x} = 16$.
 \end{lemma}
 
\begin{proof}
For any $x\in C_\KS$, there exists an $h\in H$ such that $h\cdot x=x_\KS$. Then, using the $H$-stability of $\Lambda'_{C_\KS}$, the morphism
$
\Lambda'_{x} \rightarrow\Lambda'_{x_\KS}
$
given by $y\mapsto h\cdot y$ defines an isomorphism. 
By a direct calculation of $[x_\KS, y]=0$ for $y\in C^*_\KS$, we find that $\Lambda_{x_\KS}$ is precisely the set of $y\in V^*$ that take the form
\begin{equation}\label{eqn:abcd1}
y_1 =
\begin{pmatrix}
0 & a
\end{pmatrix},\quad
y_2 = 
\begin{pmatrix}
a &b\\
0 & 0
\end{pmatrix}, \quad
y_3 =
\begin{pmatrix}
0 & c\\
0 & d
\end{pmatrix}, \quad
y_4 =
\begin{pmatrix}
c \\ 0
\end{pmatrix}
\end{equation}
where $a$, $b$, $c$, and $d$ are $2$ by $2$ matricies.
Since $\widehat{C}_\KS = C_\KS$, the rank conditions which define $\Lambda_{x_\KS}' \subset \Lambda_{x_\KS}$ imply $\det(ac)\neq 0$ and $\det(ac+bd)\neq 0$. 
This gives us the conditions that $a$, $c$ and $ac+bd$ are invertible. 
\end{proof}
%


%

\begin{lemma}\label{lemma:notPHV}
Every $Z_H(x_\KS)$-orbit in the $16$-dimensional vector space $\Lambda_{x_\KS}$ has dimension between $10$ and $14$.
Consequently, the vector space $\Lambda_{x_\KS}$ with $Z_H(x_\KS)$-action is not a prehomogeneous vector space. 
\end{lemma}

\begin{proof}
Since $\Lambda'_{x_\KS}$ is dense in $\Lambda_{x_\KS}$, to prove the lemma it suffices to study the dimensions of orbits passing through $\Lambda'_{x_\KS}$.
Let $y_\KS(a,b,c,d)  \in\Lambda'_{x_\KS}$ be defined by Equation~\eqref{eqn:abcd1}.
A direct calculation shows that any $h = (h_4,h_3,h_2,h_1,h_0) \in Z_{H}(x_\KS)$ takes the form
\begin{equation}\label{eqn:k}
h_1 =
\begin{pmatrix}
h_0 & u\\
0 & k_1
\end{pmatrix}, \quad
h_2 = 
\begin{pmatrix}
k_1 & 0\\
0 & k_2
\end{pmatrix}, \quad
h_3 =
\begin{pmatrix}
k_1 & v\\
0 & h_4
\end{pmatrix}, 
\end{equation}
for unique $h_0, h_4, k_1, k_2 \in \GL_2(\CC)$  and $u,v\in \Mat_{2,2}(\CC)$.
Observe that we have
$
\dim Z_{H}(x_\KS) = 6\times4= 24.
$
With reference to the notation above, the action of $h$ on $\Lambda'_{x_\KS}$ is given by
\begin{equation}\label{eqn:action}
h\cdot y_\KS(a,b,c,d) = y_\KS(h_0ak_1^{-1}, h_0b k_2^{-1}, k_1ch_4^{-1}, k_2dh_4^{-1}).
\end{equation}
From this we see that the conjugacy class of $(ac)^{-1}(bd)$ is an invariant of this action. 

We now find $Z_{H}(x_\KS,y_\KS(a,b,c,d))$ by solving the equations 
\[
a= h_0ak_1^{-1},\quad b=h_0b k_2^{-1}, \quad c= k_1ch_4^{-1},\quad d= k_2dh_4^{-1},
\]
for $(h_0, h_4, k_1, k_2, u, v)$ in terms of $(a,b,c,d)$.
Clearly, $u$ and $v$ are arbitrary and $h_0$ and $k_1$ may be expressed in terms of $h_4$ by
 \[
 h_0 = (ac) h_4 (ac)^{-1}, \quad k_1 =c h_4 c^{-1}.
 \]
\begin{enumerate}
\item[$\bullet$]
If $b=0$ and $d=0$, then $h_4$ and $k_2$ are arbitrary; in this case, \[\dim Z_{H}(x_\KS,y_\KS(a,b,c,d)) = 16\]
and this is largest this centralizer can be.
\item[$\bullet$]
At the other extreme, if $b$ and $d$ are both invertible, then $k_2$ may also be expressed in terms of $h_4$ by 
$
k_2 = d h_4 d^{-1},
$
but now $h_4$ is constrained by
\[
(ac)^{-1} (bd) h_4 = h_4 (ac)^{-1}(bd),
\]
so $h_4$ centralizes $(ac)^{-1}(bd)$. 
The dimension of the centralizer of $(ac)^{-1}(bd)$ is minimal when $(ac)^{-1}(bd)$ is regular (either regular semisimple or regular unipotent modulo the centre), in which case its centralizer is $2$-dimensional. It follows that
\[
\dim Z_{H}(x_\KS,y_\KS(a,b,c,d)) = 10,
\]
and this is the smallest this centralizer can be. 
\end{enumerate}
It follows that, for any $(x,y)\in \Lambda'_{C_\KS}$,
\[
10\leq \dim Z_{H}(x,y)\leq 16
\]
and that these extreme dimensions are attained. Since $\dim Z_{H}(x) = 24$, for any $x\in C_\KS$, it follows that
\[
8 = 24-16
\leq
\dim \mathcal{O}_{Z_{H}(x)}(y) 
\leq 
24-10 = 14
\]
with the lower bound is attained when $(x,y)$ is $H$-conjugate to $(x_\KS,y_\KS(a,b,c,d))$ where $(ac)^{-1}(bd)$ is invertible and regular semisimple or regular unipotent. \qedhere

\end{proof}

\begin{proposition}\label{prop:codim2}
For every $(x,y)\in \Lambda'_{C_\KS}$, the $H$-orbit $\mathcal{O}_{H}(x,y)$ of $(x,y)$ in $\Lambda_{C_\KS}$ has codimension at least $2$ and at most $8$, with these values attained.
\end{proposition}

\begin{proof}
Recall that $\Lambda_{C_\KS}\rightarrow C_\KS$ is a bundle and $\mathcal{O}_H(x,y)$ is a subbundle of $\Lambda_{C_\KS}$. 
Restricting to $\mathcal{O}_H(x,y)$, the fibre of $\mathcal{O}_H(x,y)\rightarrow C_\KS$ above $x\in C_\KS$ is $\mathcal{O}_{Z_H(x)}(y)$. Since $\Lambda_{C_\KS}$ and $\mathcal{O}_H(x,y)$ are bundles over $C_\KS$, the codimension of $\mathcal{O}_H(x,y)$ in $\Lambda_{C_\KS}$ is equal to
\begin{align*}
\dim \Lambda_{C_\KS} - \dim\mathcal{O}_H(x,y)
&= (\dim C_\KS + \dim\Lambda'_{x}) - (\dim C_\KS + \dim\mathcal{O}_{Z_H(x)}(y))\\
&= \dim\Lambda'_{x} - \dim\mathcal{O}_{Z_H(x)}(y).
\end{align*}
Thus, the codimension of $\mathcal{O}_H(x,y)$ in $\Lambda_{C_\KS}$ coincides with the codimension of $\mathcal{O}_{Z_H(x)}(y)$ in $\Lambda'_{x}$. 
Using $\dim \Lambda'_{x} =16$ from Lemma~\ref{lemma:abcd1}, we conclude that
\[
2 = 16-14 
\leq 
\operatorname{codim}\mathcal{O}_{Z_{H}(x)}(y) 
\leq 16-8 = 8
\]
with both extremes attained by Lemma~\ref{lemma:notPHV}.
\end{proof}

\subsection{The generic conormal bundle}\label{ssec:generic}

For a moment, let $G$ be any connected reductive algebraic group over $F$ and let $\lambda: W_F\to \Lgroup{G}$ be any infinitesimal parameter in the sense of \cite{CFMMX}*{Section 4.1}.
The $H$-variety $\Lambda^\text{gen}_{C}$ appearing in the codomain of the functor 
\[
\Evs_{C} : \Perv_{H}(V) \to \Loc_{H}(\Lambda^\text{gen}_{C})
\]
is defined in \cite{CFMMX}*{Section 7.9};
it is an open, dense sub-bundle of $\Lambda_{C}$ situated between $\Lambda^\text{sreg}_{C}$ and $\Lambda^\text{reg}_{C}$.
\[
\Lambda^\text{sreg}_{C} \subseteq \Lambda^\text{gen}_{C}\subseteq \Lambda^\text{reg}_{C}
\] 
We denote the equivariant fundamental group of $\Lambda^\text{gen}_{C}$ by $A^\ABV_{\phi}$ for a Langlands parameter $\phi$ with $x_\phi \in C$.
If there is an Arthur parameter $\psi$ with Langlands parameter $\phi$, we say that $C$ is of Arthur type; in this case $\emptyset \ne \Lambda^\text{sreg}_{C} = \Lambda^\text{gen}_{C}\subseteq \Lambda^\text{reg}_{C}$ and $A^\ABV_{\phi} \iso A_\psi \ceq \pi_0(Z_{\dualgroup{G}}(\psi))$.

Return to $G= \GL_{16}$ over $F$.
In this section we find an explicit description of $\Lambda^\text{gen}_{C_\KS}$ and we show in Proposition \ref{prop:L} that its equivariant fundamental group, denoted by $A^\ABV_{\phi_\KS}$, is non-trivial. 
This was an unexpected result, since if $\phi$ is a Langlands parameter of Arthur type for a general linear group, then $A^\ABV_{\phi} = A_\psi$ is trivial, where $\psi$ is the Arthur parameter for $\phi$.
The proof that the open dense subbundle $\Lambda^\text{gen}_{C_\KS}$ described in this Section satisfies the definition appearing in \cite{CFMMX}*{Section 7.9} is a consequence of the calculations performed in Section~\ref{sec:geometric}.

We begin by improving our parametrization of $H$-orbits in $\Lambda'_{C_\KS}$, or more precisely, by making a more detailed study of the $Z_H(x_\KS)$-orbits in the $16$-dimensional variety $\Lambda'_{x_\KS}$.
From Section~\ref{ssec:KSisnotArthur}, recall
\[
\Lambda'_{x_\KS}
= \left\{ y_\KS(a,b,c,d)\in \Mat_{2,2}(\CC)^{\times 4}
\ \bigg\vert\   \begin{array}{c} \det(ac)\neq 0, \\ \det(ac+bd)\neq 0 \end{array}
\right\}.
\]
As we remarked in the proof of Lemma~\ref{prop:codim2}, a direct calculation shows that the $\GL_2(\CC)$-conjugacy class of $(ac)^{-1}(bd)\in \GL_2(\CC)$ is a fundamental invariant of the $Z_H(x_\KS)$-orbit of $y_\KS(a,b,c,d)$.

Let $Z$ be the Steinberg quotient for $\Mat_{2,2}$; we write $[g]\in Z$ for the characteristic polynomial or the $\GL_2$-conjugacy class of $g\in \Mat_{2,2}$, according to taste.
Set $Z' = \{ [g]\in Z \tq \det(1+g)\ne 0\}$. 
The map
\begin{equation}
\begin{array}{rcl}
\Lambda'_{x_\KS} &\to& Z' \\
y_\KS(a,b,c,d) &\mapsto& [(ac)^{-1}(bd)]
\end{array}
\end{equation}
is a useful tool in the study of $ \Lambda'_{x_\KS}$.
We find the smooth locus of this map.

\begin{lemma}\label{lem:yKS}
If $y\in \Lambda'_{x_\KS}$ and $2\leq \mathop{codim}\mathcal{O}_{Z_H(x_\KS)}(y)\leq 4$ then there is an $h\in Z_H(x_\KS)$ such that $h\cdot y = y_\KS(1,b,1,1)$ for $b\in \GL_2(\CC)$ such that $\det(1+b)\ne 0$. 
The $\GL_2(\CC)$-conjugacy class of $b$ is an invariant of the $Z_H(x_\KS)$-orbit of $y$.
\end{lemma}

\begin{proof}
From the proof of Lemma~\ref{lemma:notPHV}, if $y\in \Lambda'_{x_\KS}$ then $y= y_\KS(a,b,c,d)$ for $a,c\in \GL_2(\CC)$ and $b,d\in \Mat_{2,2}(\CC)$ such that $\det(ac+bd)\ne0$.
Further, the proof can be extended to show that $2\leq \operatorname{codim}\mathcal{O}_{Z_H(x_\KS)}(y)\leq 4$ if and only if $b,d \in \GL_2(\CC)$. In this case, we set
\[
k_1\ceq h_4c^{-1},\qquad k_2\ceq h_4d^{-1},\qquad h_0\ceq h_4(ac)^{-1},
\]
and let $u$ and $v$ be arbitrary; using Equations~\eqref{eqn:k} and \eqref{eqn:action}, this determines an $h\in Z_{H}(x_\KS)$ such that 
\[
h\cdot y_\KS(a,b,c,d) = y_\KS(1,b',1,1),
\]
where $b'\ceq h_4 (ac)^{-1}(bd) h_4^{-1}$. 
Henceforth we use the notation
\begin{equation}\label{eqn:yKSb}
y_\KS(b')\ceq y_\KS(1,b',1,1),
\end{equation}
for any $b'\in \GL_2(\CC)$.
We have already seen that $q(y) = [b'] = [(ac)^{-1}(bd)]$ is an invariant of the $Z_H(x_\KS)$-orbit of $y$. A direct calculation shows that if
\[
h\cdot y_\KS(b) = y_\KS(b')
\]
then $[b]=[b']$. 
\end{proof}

We now extend $\Lambda'_{x_\KS} \to Z'$ to
\[
q' : \Lambda'_{C_\KS} \to Z',
\]
which we will use to give a concrete definition of $\Lambda^\text{gen}_{C_\KS}$ which agrees with \cite{CFMMX}*{Section 7.9}.
For every $(x,y)\in \Lambda'_{C_\KS}$, recall from Section~\ref{ssec:Conormalbundle} that $x=(x_4,x_3,x_2,x_1)$ and $y=(y_1,y_2,y_3,y_4)$ where $x_i : E_{\lambda_{i-1}} \to E_{\lambda_i}$ and $y_i : E_{\lambda_i} \to E_{\lambda_{i_1}}$ for $1\leq i\leq 4$.
The condition $y\in C_\KS^*$ implies $\rank y_3=2$, so $\dim \ker y_2 =2$. 
Likewise, $x\in C_\KS$ implies $\rank x_2 = \rank x_3 = \rank x_2 x_3 =2$; note that 
\[
E_{\lambda_2} = \image x_2 \oplus \ker x_3.
\] 
Also note that $\dim \image y_2 =2$. 
Let $a,b,c, d$ be the linear transformations of $2$-dimensional vector spaces defined by
\[
\begin{array}{rcl}
a &\ceq& y_2\vert_{\image x_2} : \image x_2 \to \image y_2 \\
b &\ceq& y_2\vert_{\ker x_3} : \ker x_3 \to \image y_2 \\
c &\ceq& y_3 : E_{\lambda_3}/\ker y_3 \to E_{\lambda_2} \to E_{\lambda_2}/\ker x_3 \\
d &\ceq& y_3 : E_{\lambda_3}/\ker y_3 \to E_{\lambda_2} \to E_{\lambda_2}/\image x_2,
\end{array}
\]
as pictured below,
\begin{equation}\label{abcd}
\begin{tikzcd}
& \arrow{dl}[swap]{a} \image x_2 \arrow{dr} && \arrow{ll}[swap]{\iso} E_{\lambda_2}/\ker x_3 & \\
\image y_2 && \arrow{ll}[swap]{y'_2} E_{\lambda_2} \arrow{ur} \arrow{dr} && \arrow{ll}[swap]{y'_3} \arrow{ul}[swap]{c} \arrow{dl}{d} E_{\lambda_3}/\ker y_3   \\
& \arrow{ul}{b} \ker x_3 \arrow{ur} && \arrow{ll}{\iso} E_{\lambda_2}/\image x_2, &
\end{tikzcd}
\end{equation}
where the maps $y_2'$ and $y_3'$ are induced from $y_2$ and $y_3$.
This notation is compatible with Lemma~\ref{lemma:notPHV}, which is to say, if $x=x_\KS$ then $y=y_\KS(a,b,c,d)$.
Returning to the general case of $(x,y)\in \Lambda'_{C_\KS}$, we know $\det(ac) \ne 0$.
A direct calculation shows that the $\GL_2(\CC)$-conjugacy class of the endomorphism 
\[
(ac)^{-1}(bd) : E_{\lambda_3}/\ker y_3  \to E_{\lambda_3}/\ker y_3
\]
is an invariant of the $H$-conjugacy class of $(x,y)\in \Lambda'_{C_\KS}$.
In this way we define
\begin{equation}
\begin{array}{rcl}
q' : \Lambda'_{C_\KS} &\to& Z' \\
(x,y) &\mapsto& [ (ac)^{-1}(bd)].
\end{array}
\end{equation}

Let $Z_\text{reg}$ be the open dense subvariety of $Z$ over which the Steinberg quotient is smooth and recall that the fibre of the Steinberg quotient over $Z_\text{reg}$ is $\GL_2(\CC)_\text{rss}$, the subvariety of regular semisimple elements of $\GL_2(\CC)$. 
Then $Z_\text{reg} = T_\text{reg}//W$, the moduli space of regular semisimple conjugacy classes in $\GL_2(\CC)$, where $T$ is the maximal torus in $\GL_2(\CC)$.
Set $T'\ceq \{ t\in T\tq \det(1+t)\ne 0\}$ and recall $Z'\ceq \{ [g]\in Z\tq \det(1+g)\ne 0\}$.
Set $T'_\text{reg} = T'\cap T_\text{reg}$ and $Z'_\text{reg} = Z'\cap Z_\text{reg}$.
Let $\Lambda^\text{gen}_{C_\KS}$ be the dense, open subvariety of $\Lambda_{C_\KS}$ defined by the pullback $\Lambda^\text{gen}_{C_\KS} \to Z'_\text{reg}$ of $q'$ along $Z'_\text{reg}\hookrightarrow Z'$:
\[
\begin{tikzcd}
\Lambda^\text{gen}_{C_\KS} \arrow{r}{q} \arrow{d} & Z'_\text{reg}\arrow{d} \\
\Lambda'_{C_\KS} \arrow{r}{q'} & Z' .
\end{tikzcd}
\]

\begin{proposition}\label{prop:L}
If $(x,y)\in \Lambda^\text{gen}_{C_\KS}$ then $\mathop{codim} \mathcal{O}_{H}(x,y) = 2$.
The bundle $q: \Lambda^\text{gen}_{C_\KS}\to Z'_\text{reg}$ is smooth.
 The equivariant fundamental group $A^\ABV_{\phi_\KS}$ of $\Lambda^\text{gen}_{C_\KS}$ is non-trivial.
\end{proposition}

\begin{proof}
The first point is baked into the definition of $\Lambda^\text{gen}_{C_\KS}$ using Proposition~\ref{prop:codim2}. 
For the second, observe that $q$ is locally trivializable and 
\[
q^{-1}([t]) = \mathcal{O}_H(x_\KS,y_\KS(t)) \iso H/K
\]
where $K= Z_{H}(x_\KS,y_\KS(t))$ is the group of $h= (h_0,h_1,h_2,h_3,h_4)\in H$ such that $h_0\in T(\CC)$ and
\[
h_1 =
\begin{pmatrix}
h_0 & u\\
0 & h_0
\end{pmatrix}, \quad
h_2 = 
\begin{pmatrix}
h_0 & 0\\
0 & h_0
\end{pmatrix}, \quad
h_3 =
\begin{pmatrix}
h_0 & v\\
0 & h_0
\end{pmatrix}, \quad
h_4 = h_0;
\]
this group $K$ does not depend on $t$.
Finally, consider the pull-back $z' : \widetilde{\Lambda}^\text{gen}_{C_\KS}\to \Lambda^\text{gen}_{C_\KS}$ of $z$ along $q$:
\[
\begin{tikzcd}
\arrow{d}[swap]{z'} \widetilde{\Lambda}^\text{gen}_{C_\KS} \arrow{r}{q'}&  \arrow[dashed]{dl} T'_\text{reg} \arrow{d}{z} \\
\Lambda^\text{gen}_{C_\KS} \arrow{r}{q}&  Z'_\text{reg}.
\end{tikzcd}
\]
Then $z' : \widetilde{\Lambda}^\text{gen}_{C_\KS}\to \Lambda^\text{gen}_{C_\KS}$ is an $H$-equivariant double cover. 
Since the equivariant fundamental group of a connected $H$-variety $S$ is the fundamental group for the site of $H$-equivariant etale covers of $S$, it follows that $A^\ABV_{\phi_\KS}$ is not trivial.
\end{proof}

\subsection{A curious ABV-packet}

\begin{corollary}\label{cor:main}
For any $p$-adic field $F$, the ABV-packet for $\phi_\KS$ contains exactly two representations:
\[
\Pi^\ABV_{\phi_\KS}(\GL_{16}(F)) := 
\{ \pi_\KS, \pi_\psi \}.
\]
\end{corollary}

\begin{proof}
Using \cite{CFMMX}*{Definition 1}, this corollary is a direct consequence of Theorem~\ref{thm:geometric}, which we will prove in Section~\ref{sec:geometric}, relying on tools explained in Appendix~\ref{sec:computations}.
We note that general linear groups have no pure inner forms.
\end{proof}

\section{The main geometric result}\label{sec:geometric}

Recall the functor 
$
\Evs_C : \Perv_{H}(V) \to \Loc_{H}(\Lambda_C^\text{gen})
$
from \cite{CFMMX}*{Section 7.9}.

\begin{theorem}\label{thm:geometric}
\[
\Evs_{C_\KS} \IC(\1_{C}) = 
\begin{cases} 
\1_{\Lambda^\text{gen}_{C_\KS}} & C= C_\KS,\\
\mathcal{L}_{\Lambda^\text{gen}_{C_\KS}} & C= C_\psi.\\
0, & \text{otherwise},
\end{cases}
\]
where $\mathcal{L}_{\Lambda^\text{gen}_{C_\KS}}$ is the rank-$1$ local system defined by the double cover $z' :  \widetilde{\Lambda}^\text{gen}_{C_\KS} \to \Lambda^\text{gen}_{C_\KS}$ appearing in the proof of Proposition~\ref{prop:L}. 
\end{theorem}

The proof of Theorem~\ref{thm:geometric} will occupy all of Section~\ref{sec:geometric}.
Some of the supporting computations for the proof are explained in Appendix~\ref{sec:computations}.

\begin{corollary}\label{cor:KS}
The rank of $\Lambda_{C_\KS}$ in the characteristic cycles of $\IC(\1_{C_\psi})$ is $1$ and the rank of $\Lambda_{C_\psi}$ in the characteristic cycles of $\IC(\1_{C_\psi})$ is $1$.
\end{corollary}

\begin{proof}
By \cite{CFMMX}*{Proposition 7.13}, the rank of $\Evs_{C_\psi} \IC(\1_{C_\psi})$ is $1$. 
By Theorem~\ref{thm:geometric}, the rank of $\Ev_{C_\KS} \IC(\1_{C_\psi})$ is $1$. 
By \cite{CFMMX}*{Section 7.11}, this proves the corollary.
\end{proof}

\begin{remark}\label{rem:KS}
By \cite{KS}*{Theorem 7.2.1}, the rank of $\Lambda_{C_\KS}$ in the characteristic cycles of $\IC(\1_{C_\psi})$ is at least $1$ and the rank of $\Lambda_{C_\psi}$ in the characteristic cycles of $\IC(\1_{C_\psi})$ is at least $1$. 
Thus, Theorem~\ref{thm:geometric} is stronger than \cite{KS}*{Theorem 7.2.1}.
The remark following \cite{KS}*{Theorem 7.2.1} promises -- but does not prove -- that the characteristic cycles of $\IC(\1_{C_\psi})$ is exactly $\Lambda_{C_\KS}$ and $\Lambda_{C_\psi}$. 
This is consistent with Theorem~\ref{thm:geometric}.
\end{remark}

For use in the arguments below, we define $y_\KS:  T'_\text{reg} \to \Lambda^\text{gen}_{x_\KS}$ by 
\begin{equation}\label{eqn:yKSt}
y_\KS(t_1,t_2) \ceq y_\KS\left(\begin{smallmatrix} t_1 & 1 \\ 0 & t_2 \end{smallmatrix}\right).
\end{equation}
Then $q(y_\KS(t_1,t_2)) = z(t_1,t_2)$ for $\diag(t_1,t_2)\in T'_\text{reg}$. 
The results of Section~\ref{ssec:generic} show that $y_\KS(T'_\text{reg})$ is a parametrized slice of $\Lambda^\text{gen}_{x_\KS}$, intersecting each $Z_H(x_\KS)$-orbit exactly twice.
We use this parametrized slice of $\Lambda^\text{gen}_{x_\KS}$ because it makes the singularities of the cover $\rho_\psi : \widetilde{C}_\psi \to \overline{C}_\psi$ above $x_\KS\in C_\KS\subset \overline{C}_\psi$ easily accessible in the affine chart for $\widetilde{C}_\psi$ that intersects $\rho_\psi^{-1}(x_\KS)$ appearing in Section~\ref{ssec:case-psi}. 
We return to the consequences of this choice in Section~\ref{ssec:L}.

\subsection{$\Evs_{C_\KS}\IC(\1_{C_\KS})$}\label{ssec:case-KS}

We start the proof of Theorem~\ref{thm:geometric} with
\begin{equation}\label{KS}
\Evs_{C_\KS}\IC(\1_{C_\KS}) = \1_{\Lambda^\text{gen}_{C_\KS}} .
\end{equation}
The local system $\Evs_{C_\KS}\IC(\1_{C_\KS})$ is studied in \cite{CFMMX}*{Section 7.10}, where it is shown that it is a rank-$1$ local system corresponding to a character of $A^\ABV_{\phi_\KS}$ of order at most $2$.
Equation~\eqref{KS} is equivalent to the claim that this character is trivial.
Pick $(x,y) = (x_\KS,y_\KS(t))\in \Lambda^\text{gen}_{C_\KS}$ so that we may identify $A^\ABV_{\phi_\KS}$ with the equivariant fundamental group of $\Lambda^\text{gen}_{C_\KS}$ at $(x,y)$ and find the character as a representation of the one-dimensional vector space 
$\Evs_{C_\KS}\IC(\1_{C_\KS})_{(x,y)}$. 
Use \cite{CFMMX}*{Theorem 7.19} to see that the character of $A^\ABV_{\phi_\KS}$ associated to $\Evs_{C_\KS}\IC(\1_{C_\KS})$ is trivial if and only if the Hessian determinant of $f$ at $(x,y)$ is a square in the local coordinate ring of the singular locus at this point. 
To that end, we employ the following method, using Macaulay2 code documented in Appendix~\ref{app:Hessian}.
\begin{enumerate}
\item Use the Jacobian of the ideal defining the closure of $C_\KS\times C^*_\KS$ to identify a system of variables, from the polynomial ring over which $C_\KS\times C^*_\KS$ is defined, that may be used as local coordinates for the (completed) local ring at the point $(x_\KS,y_\KS(t))$.
\item Use implicit partial differentiation to solve for the partial derivatives of the other variables with respect to the chosen system of local coordinates.
\item Use partial differentiation with the chain rule to compute the Hessian of $f$ as a function in the completed local ring. Recall the Hessian is a symmetric matrix of functions from the completed local ring.
\item Verify that the rank of the Hessian evaluated at $(x,y)$ is $e_{C_\KS} = \dim C_\KS - \codim C_\KS^*  = 32-(48-32)=16$.
\item Find a $16\times 16$  minor of the Hessian whose rank when evaluated at $(x,y)$ is 16.
\item Find an ordering of these $16$ variables that makes it clear the determinant of this minor is the square of a function in the local ring. 
\end{enumerate}
Using this method, we found that the Hessian determinant is locally the perfect square of a function and hence $\Evs_{C_\KS}\IC(\1_{C_\KS})$ is the constant local system.

\begin{remark}
These calculations verify the claim about points of $ \Lambda^\text{gen}_{C_\KS}$ even if $(x,y)$ is not generic because the rank of the Hessian is $16$ and the calculations
give us the same conclusion on an open neighbourhood of $(x,y)$ in $ \Lambda^\text{gen}_{C_\KS}$, which must include generic vectors.
\end{remark}

\subsection{Reduction of the problem}
\label{ssec:reduction}

For each orbit $C$ in $V$, there exists a unique orbit $C^*$ for $V^*$ with
$ \overline{\Lambda_C} \iso \overline{\Lambda_{C^*}}$; see, for example, \cite[Corollary 2]{Pyasetskii}.
The map $C \mapsto C^*$ defines a bijection between the set of $H$-orbits in $V$ and the set of $H$-orbits in $V^*$.
Another bijection between $H$-orbits in $V^*$ and $H$-orbits in $V$, called transposition,
is $t: V^*\to V$ defined by $y=(y_1,y_2,y_3,y_4) \mapsto \,^ty\ceq (\,^ty_4,\,^ty_3,\,^ty_2,\,^ty_1)$, where $\,^ty_i$ is the transpose of the matrix $y_i$. 
Now set 
\[
{\widehat C} \ceq \,^tC^*.
\]
Then $C\mapsto {\widehat C}$ is an involution on the set of $H$-orbits in $V$.
The involution $C\mapsto {\widehat C}$ on $H$-orbits in $V$ is the Zelevinsky involution on multisegments; we refer the reader to \cite{KZ} for more in this direction. 
The orbits $C_\KS$ and $C_\psi$ are self-dual:
\[
\widehat{C}_\KS = C_\KS\qquad\text{ and }\qquad\widehat{C}_\psi = C_\psi.
\]

\begin{proposition}\label{prop:ordering}
If $\Evs_{C}\IC(\1_{C'}) \neq 0$ then
$\Evs_{\widehat C}\IC(\1_{\widehat C'}) \neq 0$ and
$C\leq C'$ and $ {\widehat C} \leq {\widehat C'}$.
\end{proposition}
\begin{proof}
By combining \cite{CFMMX}*{Proposition 7.8} and \cite{KS_sheaves}*{Proposition 8.6.4} we have that $\Evs_{C}\IC(\1_{C'})$ is non-zero if and only if $\Lambda_C$ is contained in the characteristic cycles of $\IC(\1_{C'})$.
Combining \cite[Proposition 7.2]{evens1997} and \cite[Theorem 5.5.5]{KS_sheaves} it follows that $\Lambda_C$ is contained in the characteristic cycles of $\IC(\1_{C'}) $ if and only if
$\Lambda_{\widehat{C}}$  is contained in the characteristic cycles of $\IC(\1_{\widehat{C}'}) $.
Thus $\Evs_{C}\IC(\1_{C'}) \neq 0$ if and only if $\Evs_{\widehat C}\IC(\1_{\widehat C'}) \neq 0$.
By the definition of $\Evs_{C}$ it is clear that $\Evs_{C}\IC(\1_{C'}) \neq 0$ implies $C\leq C'$ and correspondingly that $\Evs_{\widehat C}\IC(\1_{\widehat C'}) \neq 0$ implies
$ {\widehat C} \leq {\widehat C'}.$
\end{proof}

\begin{remark}
Proposition~\ref{prop:ordering} is equivalent to the following statement:
\[
a^* \left(\left(\RPhi_{f} \left(\IC(\1_{C})\boxtimes\1^!_{C^*}\right) \right)\vert_{\Lambda^\text{gen}_{C}}\right)
= 0 \iff
\left(\RPhi_{f}\left( \IC(\1_{C^*})\boxtimes\1^!_{C}\right)\right) \vert_{\Lambda^\text{gen}_{C^*}} = 0,
\]
where $a: T^*(V)\to T^*(V^*)$ is defined by $a(x,y) = (y,-x)$ using $(V^*)^*\iso V$.
\end{remark}

\begin{lemma}\label{lem:Kittsix}
Besides $C_\KS$, there are exactly six $H$-orbits $C$ that satisfy the conditions
\[
C_\KS \leq C\quad \text{and} \quad C_\KS \leq \widehat{C}.
\]
These six orbits are listed in Table~\ref{table:Kittsix}.
\end{lemma}

\begin{proof}
The proof is given by a direct calculation using Macaulay2 as explained in Appendix~\ref{ssec:generalcomputations}.
\end{proof}

\begin{table}[htp]
\caption{The six $H$-orbits $C\subset V$ that satisfy the relations $C_\KS < C$ and $C_\KS< \widehat{C}$.}
\label{table:Kittsix}
\begin{center}
$
\begin{array}{ccc}
C_L & C_\psi & C_R\\ 
\begin{tikzpicture}
\node(A) at (0,0) {$2$};
\node(B) at (\rtwidth,0) {$4$};
\node(C) at (2*\rtwidth,0) {$4$};
\node(L) at (3*\rtwidth,0) {$4$};
\node(M) at (4*\rtwidth,0) {$2$};
\draw (0,-0.5*\rtheight) -- (4*\rtwidth,-0.5*\rtheight);
\node(D) at (0.5*\rtwidth,-\rtheight) {$2$};
\node(E) at (1.5*\rtwidth,-\rtheight) {$4$};
\node(d) at (2.5*\rtwidth,-\rtheight) {$2$};
\node(e) at (3.5*\rtwidth,-\rtheight) {$2$};
\node(F) at (\rtwidth,-2*\rtheight) {$2$};
\node(G) at (2*\rtwidth,-2*\rtheight) {$2$};
\node(H) at (3*\rtwidth,-2*\rtheight) {$0$};
\node(I) at (1.5*\rtwidth,-3*\rtheight) {$0$};
\node(J) at (2.5*\rtwidth,-3*\rtheight) {$0$};
\node(K) at (2*\rtwidth,-4*\rtheight) {$0$};
\end{tikzpicture}
&
\begin{tikzpicture}
\node(A) at (0,0) {$2$};
\node(B) at (\rtwidth,0) {$4$};
\node(C) at (2*\rtwidth,0) {$4$};
\node(L) at (3*\rtwidth,0) {$4$};
\node(M) at (4*\rtwidth,0) {$2$};
\draw (0,-0.5*\rtheight) -- (4*\rtwidth,-0.5*\rtheight);
\node(D) at (0.5*\rtwidth,-\rtheight) {$2$};
\node(E) at (1.5*\rtwidth,-\rtheight) {$3$};
\node(d) at (2.5*\rtwidth,-\rtheight) {$3$};
\node(e) at (3.5*\rtwidth,-\rtheight) {$2$};
\node(F) at (\rtwidth,-2*\rtheight) {$1$};
\node(G) at (2*\rtwidth,-2*\rtheight) {$2$};
\node(H) at (3*\rtwidth,-2*\rtheight) {$1$};
\node(I) at (1.5*\rtwidth,-3*\rtheight) {$1$};
\node(J) at (2.5*\rtwidth,-3*\rtheight) {$1$};
\node(K) at (2*\rtwidth,-4*\rtheight) {$0$};
\end{tikzpicture}
&
\begin{tikzpicture}
\node(A) at (0,0) {$2$};
\node(B) at (\rtwidth,0) {$4$};
\node(C) at (2*\rtwidth,0) {$4$};
\node(L) at (3*\rtwidth,0) {$4$};
\node(M) at (4*\rtwidth,0) {$2$};
\draw (0,-0.5*\rtheight) -- (4*\rtwidth,-0.5*\rtheight);
\node(D) at (0.5*\rtwidth,-\rtheight) {$2$};
\node(E) at (1.5*\rtwidth,-\rtheight) {$2$};
\node(d) at (2.5*\rtwidth,-\rtheight) {$4$};
\node(e) at (3.5*\rtwidth,-\rtheight) {$2$};
\node(F) at (\rtwidth,-2*\rtheight) {$0$};
\node(G) at (2*\rtwidth,-2*\rtheight) {$2$};
\node(H) at (3*\rtwidth,-2*\rtheight) {$2$};
\node(I) at (1.5*\rtwidth,-3*\rtheight) {$0$};
\node(J) at (2.5*\rtwidth,-3*\rtheight) {$0$};
\node(K) at (2*\rtwidth,-4*\rtheight) {$0$};
\end{tikzpicture}
\\
C_l & C_m & C_r\\
\begin{tikzpicture}
\node(A) at (0,0) {$2$};
\node(B) at (\rtwidth,0) {$4$};
\node(C) at (2*\rtwidth,0) {$4$};
\node(L) at (3*\rtwidth,0) {$4$};
\node(M) at (4*\rtwidth,0) {$2$};
\draw (0,-0.5*\rtheight) -- (4*\rtwidth,-0.5*\rtheight);
\node(D) at (0.5*\rtwidth,-\rtheight) {$2$};
\node(E) at (1.5*\rtwidth,-\rtheight) {$3$};
\node(d) at (2.5*\rtwidth,-\rtheight) {$2$};
\node(e) at (3.5*\rtwidth,-\rtheight) {$2$};
\node(F) at (\rtwidth,-2*\rtheight) {$1$};
\node(G) at (2*\rtwidth,-2*\rtheight) {$2$};
\node(H) at (3*\rtwidth,-2*\rtheight) {$0$};
\node(I) at (1.5*\rtwidth,-3*\rtheight) {$0$};
\node(J) at (2.5*\rtwidth,-3*\rtheight) {$0$};
\node(K) at (2*\rtwidth,-4*\rtheight) {$0$};
\end{tikzpicture}
&
\begin{tikzpicture}
\node(A) at (0,0) {$2$};
\node(B) at (\rtwidth,0) {$4$};
\node(C) at (2*\rtwidth,0) {$4$};
\node(L) at (3*\rtwidth,0) {$4$};
\node(M) at (4*\rtwidth,0) {$2$};
\draw (0,-0.5*\rtheight) -- (4*\rtwidth,-0.5*\rtheight);
\node(D) at (0.5*\rtwidth,-\rtheight) {$2$};
\node(E) at (1.5*\rtwidth,-\rtheight) {$3$};
\node(d) at (2.5*\rtwidth,-\rtheight) {$3$};
\node(e) at (3.5*\rtwidth,-\rtheight) {$2$};
\node(F) at (\rtwidth,-2*\rtheight) {$1$};
\node(G) at (2*\rtwidth,-2*\rtheight) {$2$};
\node(H) at (3*\rtwidth,-2*\rtheight) {$1$};
\node(I) at (1.5*\rtwidth,-3*\rtheight) {$0$};
\node(J) at (2.5*\rtwidth,-3*\rtheight) {$0$};
\node(K) at (2*\rtwidth,-4*\rtheight) {$0$};
\end{tikzpicture}
&
\begin{tikzpicture}
\node(A) at (0,0) {$2$};
\node(B) at (\rtwidth,0) {$4$};
\node(C) at (2*\rtwidth,0) {$4$};
\node(L) at (3*\rtwidth,0) {$4$};
\node(M) at (4*\rtwidth,0) {$2$};
\draw (0,-0.5*\rtheight) -- (4*\rtwidth,-0.5*\rtheight);
\node(D) at (0.5*\rtwidth,-\rtheight) {$2$};
\node(E) at (1.5*\rtwidth,-\rtheight) {$2$};
\node(d) at (2.5*\rtwidth,-\rtheight) {$3$};
\node(e) at (3.5*\rtwidth,-\rtheight) {$2$};
\node(F) at (\rtwidth,-2*\rtheight) {$0$};
\node(G) at (2*\rtwidth,-2*\rtheight) {$2$};
\node(H) at (3*\rtwidth,-2*\rtheight) {$1$};
\node(I) at (1.5*\rtwidth,-3*\rtheight) {$0$};
\node(J) at (2.5*\rtwidth,-3*\rtheight) {$0$};
\node(K) at (2*\rtwidth,-4*\rtheight) {$0$};
\end{tikzpicture}
\end{array}
$
\end{center}
\end{table}%

%

The proof of Theorem~\ref{thm:geometric} is now reduced to showing 
 $\Evs_{C_\KS} \IC(\1_{C}) =0$ for the six orbits $C$ from Lemma~\ref{lem:Kittsix}
 and
 $
\Evs_{C_\KS} \IC(\1_{C_\psi}) = \mathcal{L}_{\Lambda^\text{gen}_{C_\KS}}.
$
We have already seen that the multisegments for $C_{\psi}$ and  $C_\KS$ are self-dual, hence $\widehat{C}_{\psi} = C_{\psi}$ and $\widehat{C}_\KS = C_\KS$. 
A simple calculation, illustrated in Appendix~\ref{ssec:generalcomputations},  shows that $C_{R} = \widehat{C}_{L}$ and $C_r = \widehat{C}_l$.
It now follows from Proposition~\ref{prop:ordering} that
\[
\Evs_{C_\KS} \IC(\1_{C_L}) = 0 \qquad\iff \qquad \Evs_{C_\KS} \IC(\1_{C_R}) =0
\]
and
\[
\Evs_{C_\KS} \IC(\1_{C_l}) = 0 \qquad\iff \qquad \Evs_{C_\KS} \IC(\1_{C_r}) =0.
\]
Therefore, the proof of Theorem~\ref{thm:geometric} reduces to the following four statements:
\begin{eqnarray}
\label{r}
\Evs_{C_\KS}\IC(\1_{C_r}) &=& 0\\
\label{m}
\Evs_{C_\KS}\IC(\1_{C_m}) &=& 0 \\
\label{R}
\Evs_{C_\KS}\IC(\1_{C_R}) &=& 0\\
\label{psi}
\Evs_{C_\KS}\IC(\1_{C_\psi }) &= & \mathcal{L}_{\Lambda^\text{gen}_{C_\KS}} .
\end{eqnarray}

\subsection{Overview of the calculations}\label{ssec:overview}

We now explain the strategy of the proofs of Equations~\eqref{r}, \eqref{m}, \eqref{R} and \eqref{psi}.

We have seen that the rank triangle for every $H$-orbit $C$ records the equations that define $C$ as a variety in $V$:
\[
C = \{ x\in V \tq \rank(x_i \cdots x_j) = r_{ij}, 1\leq j\leq i\leq 4 \},
\]
where $r_{ij}$ refer to the terms appearing in the rank triangle for $C$.
The closure of $C$ in $V$ is then given by 
\[
\overline{C} = \{ x\in V \tq \rank(x_i \cdots x_j) \leq r_{ij}, 1\leq j\leq i\leq 4 \}.
\]
The rank triangle can also be used to construct a smooth $H$-variety $\widetilde{C}$ and an $H$-equivariant proper morphism
\[
\rho: {\widetilde C} \to {\overline C}.
\]
We show how this is done in each of the cases $C_r$, $C_m$, $C_R$, $C_\psi$ in Sections~\ref{ssec:case-r} through \ref{ssec:case-psi}.

Now let $C$ be any of the orbits $C_r$, $C_m$, $C_R$ or $C_\psi$. Then, by construction, $\rho : \widetilde{C} \to \overline{C}$ is an isomorphism over $C$. It follows from the decomposition theorem \cite{BBD}*{Th\'eor\`eme 6.2.5}, \cite{deCM} that the push forward $\rho_! \1_{\widetilde{C}}[\dim \widetilde{C}]$ is a semisimple complex containing $\IC(\1_{C})$:
so
\[
\rho_! \1_{\widetilde{C}}[\dim \widetilde{C}] = \IC(\1_{C}) \oplus \mathop{\bigoplus}\limits_{C',i} m_i(C',C) \IC(\1_{C'})[d_{C',C,i}],
\]
for non-negative integers $m_i(C',C)$ and integers $d_{C',C,i}$, where $C'$ ranges over all orbits $C'< C$ and $i$ ranges over $\ZZ$.

\begin{remark}
In fact, each $\rho$ is semi-small, so the shift integers $d_{C,C',i}$ are all $0$, and we can calculate all strata $C'$ that are relevant to the cover, in the sense of \cite{deCM}*{Definition 4.2.3}, namely those $C'$ for which $m_i(C',C) \ne 0$; 
we only actually need to verify this for a small number of orbits in one case, see Remark \ref{rem:check15}.
The code which does these checks is documented in Appendix~\ref{ssec: MVC appendix}. We expect that the procedure we use to construct $\rho$ will always yield a semi-small cover.
\end{remark}

By \cite{CFMMX}*{Proposition 7.8}, 
\begin{equation}\label{eqn:Evs}
\Evs_{C_\KS} \rho_! \1_{\widetilde{C}}[\dim \widetilde{C}] = \Evs_{C_\KS} \IC(\1_{C}) \oplus \mathop{\bigoplus}\limits_{C',i} m_i(C',C) \Evs_{C_\KS} \IC(\1_{C'})[d_{C',C,i}],
\end{equation}
where the sum is taken over the $H$-orbits $C'$ such that $C_\KS\leq C'$ and $C_\KS \leq \widehat{C'}$ and $C'<C$; by Lemma~\ref{lem:Kittsix} there are at most 7 such $C'$.
Our strategy to calculate $\Evs_{C_\KS} \IC(\1_{C})$ is to use Equation~\eqref{eqn:Evs} inductively, starting with small orbits $C$, and in each case calculate $\Evs_{C_\KS} \rho_! \1_{\widetilde{C}}[\dim \widetilde{C}]$.

Using \cite{CFMMX}*{Lemma 7.11}, the left-hand side of Equation~\eqref{eqn:Evs} is given by
\[
\Evs_{C_\KS} \rho_! \1_{\widetilde{C}} [\dim \widetilde{C}]
=
\rho''_!
\left(
\left(
\RPhi_{\tilde{f}} \1_{\widetilde{C}\times C^*_\KS} 
\right)
\big\vert_{\mathcal{\tilde O}}
\right)
[\dim \widetilde{C}-\codim C_\KS^* ],
\]
where ${\tilde f} : \widetilde{C} \times C^*_\KS \to \mathbb{A}^1$ is defined by ${\tilde f}({\tilde x},y) = \KPair{\rho({\tilde x})}{y}$, $\KPair{}{} : V\times V^*\to \mathbb{A}^1$ is the usual pairing, $\rho'=\rho'$, $\mathcal{\tilde O} \ceq (\rho')^{-1}(\Lambda^\text{gen}_{C_\KS})$ and $\rho'' : \mathcal{\tilde O}\to \Lambda^\text{gen}_{C_\KS}$ is the pullback of $\rho'$ along $\Lambda^\text{gen}_{C_\KS} \to {\overline C}\times C_\KS^*$; see the following diagram:
\[
\begin{tikzcd}
U \arrow{d} & \arrow{l} U\times C_\KS^* \arrow{d} &   &   & \\
\widetilde{C} \arrow{d}{\rho} & \arrow{l} \widetilde{C}\times C_\KS^* \arrow{d}{\rho'} & \arrow{l} \mathcal{\tilde O} \arrow{d}{\rho''} & \arrow{l} \sing {\tilde f}^{-1}(0) \cap \mathcal{\tilde O} \arrow{dl}{\rho'''}  \\
\overline{C} &  \arrow{l} \overline{C}\times C_\KS^* \arrow{d}{\KPair{}{}} &  \arrow{l} \Lambda^\text{gen} _{C_\KS} \arrow{d} & \\
& \mathbb{A}^1 &  \arrow{l} \{ 0 \}  & 
\end{tikzcd}
\]
Both ${\tilde f}$ and $\rho''$ are $H$-equivariant.
By \cite{CFMMX}*{Proposition 7.20 and Lemma 7.21}, we know that
\[
\RPhi_{\tilde{f}}[-1] \1_{\widetilde{C}\times C^*_\KS} [\dim C+\dim C^*_\KS]
\]
is a perverse sheaf supported on the singular locus of $\tilde{f}$ which is contained in $\overline{\mathcal{\tilde O}}$ (note that, {\it a priori} it was supported on  ${\tilde f}^{-1}(0)$). 
We are interested in the restriction of this perverse sheaf to $\mathcal{\tilde O}$ which, when shifted by $-\dim V$, is a finite rank local system with support $(\sing{\tilde f}^{-1}(0))\cap \mathcal{\tilde O}$.
Set
\begin{equation}\label{eqn:M}
\mathcal{M} \ceq \left(\RPhi_{\tilde{f}}[-1] \1_{\widetilde{C}\times C^*_\KS}\right) \vert_{ (\sing{\tilde f}^{-1}(0))\cap \mathcal{\tilde O}} [\dim C-\codim C^*_\KS].
\end{equation}
Then 
\[
\Evs_{C_\KS} \rho_! \1_{\widetilde{C}} [\dim \widetilde{C}]
=
\rho'''_!
\mathcal{M}
\]
where $\rho'''$ is the restriction of $\rho''$ to $(\sing{\tilde f}^{-1}(0))\cap \mathcal{\tilde O}$.
In each of the four cases appearing in Equations~\eqref{r} through \eqref{psi}, we use the Jacobian condition for smoothness to study $(\sing{\tilde f}^{-1}(0))\cap \mathcal{\tilde O}$.
We use the following general strategy.
\begin{enumerate}
\item We compute, by hand, the fibre of $\rho$ above $x_\KS$. In all cases, we find that the fibre is a single point or projective.

\item We choose an affine chart $U$ for $\widetilde{C}$ that intersects $\rho^{-1}(x_\KS)$. 
We form the Jacobian for the system of equations that describe $(U \times C_{\KS}^*) \cap \tilde{f}^{-1}(0)$. We then perform Steps 4-7 for each of these Jacobians. In the case that $\rho^{-1}(x_\KS)$ is projective we will need to vary $U$ over an atlas for $\widetilde{C}$.

\item We compute the generic rank of the Jacobian from Step 2. 

\item We evaluate the Jacobian at the generic element $(x_\KS,y_\KS(t_1,t_2))$ and further add the conditions on the variables in the Jacobian that correspond to $U \cap \rho^{-1}(x_\KS)$.
The Jacobian is now a matrix $M$ over some polynomial ring $\mathbb{Q}[t_1,t_2, X_1,\dots, X_s]$, where the $X_i$ are variables used to describe the fibre from Step 2; in the case that the fibre is a single point, there are no $X_i$.
\item Next, we perform elementary row and column operations on the matrix $M$. In each case, we are left with a block diagonal matrix whose blocks consist of an identity matrix $I_m$ of some size $m$, and a matrix $B$ over $\mathbb{Q}[t_1,t_2,X_1,\dots,X_s]$.

\item The conditions that describe when the rank of $M$ is strictly less than the generic rank from Step 4 are now encoded in the matrix $B$. This gives simple equations for $\sing({\tilde f}^{-1}(0)\cap (U\times C_\KS^*))\cap (\rho')^{-1}(x_\KS,y_\KS(t_1,t_2))$.

\item Return to Step 2 until $\widetilde{C}\cap\rho^{-1}(x_{\KS})$ has been covered by affine charts $U$ for $\widetilde{C}$. This determines $(\sing{\tilde f}^{-1}(0))\cap (\rho')^{-1}(x_\KS,y_\KS(t_1,t_2))$.
Finally, letting $H$ act on this set we find $(\sing{\tilde f}^{-1}(0))\cap \mathcal{\tilde O}$.
\end{enumerate}

In the cases where $C$ is either $C_r$, $C_m$ or $C_R$, the intersection $(\sing{\tilde f}^{-1}(0))\cap \mathcal{\tilde O}$ is empty, so
$
\Evs_{C_\KS} \rho_! \1_{\widetilde{C}}[\dim \widetilde{C}]
= \rho'''_! \mathcal{M}
=0.
$
This gives Equations~\eqref{r}, \eqref{m} and \eqref{R}. In the only remaining case, when $C=C_\psi$, we find that the intersection 
$
(\sing{\tilde f}^{-1}(0))\cap \mathcal{\tilde O}
$
is non-empty and the map $\rho_\psi'''$ is $2:1$.
We find the rank of the local system $\mathcal{M}$ and the corresponding representation of the fundamental group $\Lambda^\text{gen}_{C_\KS}$ by computing the Hessian of ${\tilde f}$ near a singular point of ${\tilde f}$ in $\mathcal{\tilde O}$.
After verifying that the rank of this Hessian is the codimension of the singular locus, we conclude that the rank of $\mathcal{M}$ is $1$. Further, the determinant of this Hessian determinant then identifies the cover which trivializes $\mathcal{M}$.
The strategy to compute the Hessian locally near a point in the singular locus of ${\tilde f}$ is the same strategy employed in Section~\ref{ssec:case-KS} and documented in detail in Appendix~\ref{app:Hessian}.

\begin{remark}
All of the calculations we are describing are performed without finding a Gr\"{o}bner basis for the ideal defining  ${\mathcal{\tilde O}}$ as computing this would have been intractable.
\end{remark}

\subsection{$\Evs_{C_\KS}\IC(\1_{C_r})$}\label{ssec:case-r}

In this section we prove Equation~\eqref{r}; see Appendix~\ref{Assec: Cr} for a detailed description of the code used for the computations in this section.

We begin by defining $\rho_r : \widetilde{C}_r \to \overline{C}_r$.
%
Recall the rank triangle for $C_r$ from Table~\ref{table:Kittsix}.
As we saw in Section~\ref{ssec:ranktriangles},
the top row describes the infinitesimal parameter of the representations by listing the multiplicites of the eigenvalues of $\lambda(\Frob)$; these eigenvalues are, from right to left in the rank triangle, $\lambda_0=q^{-2}$, $\lambda_1 = q^{-1}$, $\lambda_2 = q^{0}$, $\lambda_3= q^{1}$, $\lambda_4 = q^{2}$ and the multiplicities of these eigenvalues are, in the same order, $2$, $4$, $4$, $4$, $2$.
For each eigenvalue $\lambda_i$, consider the complete flag variety $\mathcal{F}(E_{\lambda_i})$ for $\GL(E_{\lambda_i})$, which is to say, consider the projective variety of complete flags in $E_{\lambda_i}$; points on this variety are given by a chain of subspaces of $E_{\lambda_i}$
\[
0 \subset E_{\lambda_i}^1 \subset E_{\lambda_i}^2 \subset \cdots \subset E_{\lambda_i},
\]
where $\dim E_{\lambda_i}^j = j$.
The action of $H$ on the flag variety is defined by letting $h\in H$ send this chain to
\[
0 \subset h_i(E_{\lambda_i}^1) \subset h_i(E_{\lambda_i}^2) \subset \cdots \subset h_i(E_{\lambda_i}),
\]
We make one complete chain for each eigenspace and then use the rank-triangle to introduce relations on these flags, thus defining a subvariety of $\prod_{i=0}^4 \mathcal{F}(E_{\lambda_i})$, as follows.
\begin{enumerate}
\item
Looking at the second row of the rank triangle for $C_r$ we see the ranks $(2,2,3,2)$. This means, for any $x=(x_4,x_3,x_2,x_1)\in C_r$, we have $\rank  x_1 = 2$, $\rank x_2 = 3$, $\rank x_3 = 2$ and $\rank x_4 = 2$. 
\begin{enumerate}
\item Since $x_1 : E_{\lambda_0} \to E_{\lambda_1}$ and $\rank x_1 = 2$, we declare that the image of $x_1$ should be contained in $E_{\lambda_1}^2$, which is to say $x_1(E_{\lambda_0})\subseteq E_{\lambda_1}^2$.
\item Since $x_2 : E_{\lambda_1} \to E_{\lambda_2}$ and $\rank x_2 = 3$, we declare that the image of $x_2$ should be contained in $E_{\lambda_2}^3$, so $x_2(E_{\lambda_1})\subseteq E_{\lambda_2}^3$. 
\item Since $x_3 : E_{\lambda_2} \to E_{\lambda_3}$ and $\rank x_3 = 2$, we declare that the image of $x_3$ should be contained in $E_{\lambda_3}^2$, so $x_3(E_{\lambda_2})\subseteq E_{\lambda_3}^2$. 
\item Since $x_4 : E_{\lambda_3} \to E_{\lambda_4}$ and $\rank x_4 = 2$, we declare that the image of $x_4$ should be contained in $E_{\lambda_4}^2$. This is no condition at all since $E_{\lambda_4}^2= E_{\lambda_4}$.
\end{enumerate}
\item
Looking at the third row of the rank triangle for $C_r$ we see the ranks $(0,2,1)$. This means $\rank  x_2 x_1\leq 1$, $\rank x_3 x_2 \leq 2$ and $\rank x_4 x_3 \leq 0$.
\begin{enumerate}
\item Since $x_2 x_1 : E_{\lambda_0} \to E_{\lambda_2}$ and $\rank x_2 x_1 \leq 1$, we declare $x_2(E_{\lambda_1}^2)\subseteq E_{\lambda_2}^1$. Note that this is compatible with the condition
$x_2 x_1(E_{\lambda_0})\subseteq E_{\lambda_2}^1$ from the rank triangle and condition $x_1(E_{\lambda_0})\subseteq E_{\lambda_1}^2$ from above.

\item Since $x_3 x_2 : E_{\lambda_1} \to E_{\lambda_3}$ and $\rank x_3 x_2 \leq 2$, we declare $x_3(E_{\lambda_2}^3)\subseteq E_{\lambda_3}^2$. Note that this is compatible with $x_3x_2(E_{\lambda_1})\subseteq E_{\lambda_3}^2$ and $x_2(E_{\lambda_1})\subseteq E_{\lambda_2}^3$.

\item Since $x_4 x_3 : E_{\lambda_2} \to E_{\lambda_4}$ and $\rank x_4 x_3 \leq 0$, we declare 
$x_4(E_{\lambda_3}^2)=0$. Note that this is compatible with $x_4 x_3(E_{\lambda_2})=0$ and $x_3(E_{\lambda_2})\subseteq E_{\lambda_3}^2$.
\end{enumerate}
\item
Looking at the fourth row of the rank triangle for $C_r$ we see the ranks $(0,0)$. This means $\rank  x_3 x_2 x_1 \leq 0$ and $\rank x_4 x_3 x_2 \leq 0$, or equivalently, $x_3 x_2 x_1 = 0$ and $x_4 x_3 x_2 = 0$.
\begin{enumerate}
\item Since $x_3 x_2 x_1 : E_{\lambda_0} \to E_{\lambda_3}$ and $x_3 x_2 x_1 = 0$, we declare $x_3(E_{\lambda_2}^1)=0$. Note that this is compatible with $x_3 x_2 x_1(E_{\lambda_0})=0$ and $x_2 x_1(E_{\lambda_0})\subseteq E_{\lambda_2}^1$.

\item Since $x_4 x_3 x_2 : E_{\lambda_1} \to E_{\lambda_4}$ and $ x_4 x_3 x_2  = 0$, we declare $x_4(E_{\lambda_3}^2)=0$, which is not a new condition.
\end{enumerate}
\item
Looking at the bottom entry in the rank triangle for $C_r$ we see the rank $(0)$. This means $\rank  x_4 x_3 x_2 x_1 \leq 0$ or equivalently, $x_4 x_3 x_2 x_1 = 0$. Since $x_3 x_2 x_1 = 0$, this gives no new condition.
\end{enumerate}
 The diagram appearing in Table~\ref{table:Cr} defines a variety $\widetilde{C}_r$ in two steps, as follows. First, consider the subvariety of $\overline{C}_r\times\prod_{i=0}^4\mathcal{F}(E_{\lambda_i})$ defined by the relations: 
\begin{align*}
&x_1(E_{\lambda_0}) \subseteq E_{\lambda_1}^2  & &x_2(E_{\lambda_1})\subseteq E_{\lambda_2}^3 & &x_3(E_{\lambda_2}) \subseteq E_{\lambda_3}^2\\
&x_2(E_{\lambda_1}^2)\subseteq E_{\lambda_2}^1 & &x_3(E_{\lambda_2}^3)\subseteq E_{\lambda_3}^2\\
&x_3(E_{\lambda_2}^1)=0 & &x_4(E_{\lambda_3}^2)=0.
\end{align*}
Second, let $\widetilde{C}_r$ be the quotient defined by now removing the subspaces not appearing in the relations above, which is to say, by removing the vector spaces $E_{\lambda_0}^1$, $E_{\lambda_1}^1$, $E_{\lambda_1}^3$, $E_{\lambda_2}^2$, $E_{\lambda_3}^1$, $E_{\lambda_3}^3$ and $E_{\lambda_4}^1$. 
Now define $\rho_r : \widetilde{C_r} \to \overline{C_r}$ by projection, so $\rho_r(x,E) = x$ where $E$ is the partial flag $(E_{\lambda_0}, E_{\lambda_1}^2, E_{\lambda_1},E_{\lambda_2}^1, E_{\lambda_2}^3, E_{\lambda_2}, E_{\lambda_3}^2)$. Then $\rho_r$ is $H$-equivariant.

\begin{table}[tp]
\caption{Equations defining $\widetilde{C}_r$.}
\label{table:Cr}
\begin{center}
\begin{tikzpicture}
\node (A) at (0, 0)    {$0$};
\node (B) at (0.5,0)   {$\subset$};
\node (C) at (1,0)     {$E_{\lambda_0}^1$};
\node (a) at (1.5,0)   {$\subset$};
\node (b) at (2,0)     {$E_{\lambda_0}^2$};
\node (D) at (2.5,0)   {$=$};
\node (E) at (3,0)     {$E_{\lambda_0}$};
\node (F) at (0, -1.25)   {0};
\node (G) at (0.5,-1.25)  {$\subset$};
\node (H) at (1,-1.25)    {$E_{\lambda_1}^1$};
\node (I) at (1.5,-1.25)  {$\subset$};
\node (J) at (2,-1.25)    {$E_{\lambda_1}^2$};
\node (K) at (2.5,-1.25)  {$\subset$};
\node (L) at (3,-1.25)    {$E_{\lambda_1}^3$};
\node (M) at (3.5,-1.25)  {$\subset$};
\node (N) at (4,-1.25)    {$E_{\lambda_1}^4$};
\node (u) at (4.5,-1.25)  {$=$};
\node (uu)at (5,-1.25)    {$E_{\lambda_1}$};
\node (O) at (0, -2.5)   {0};
\node (P) at (0.5,-2.5)  {$\subset$};
\node (Q) at (1,-2.5)    {$E_{\lambda_2}^1$};
\node (R) at (1.5,-2.5)  {$\subset$};
\node (S) at (2,-2.5)    {$E_{\lambda_2}^2$};
\node (T) at (2.5,-2.5)  {$\subset$};
\node (U) at (3,-2.5)    {$E_{\lambda_2}^3$};
\node (V) at (3.5,-2.5)  {$\subset$};
\node (W) at (4,-2.5)    {$E_{\lambda_2}^4$};
\node (ww)at (4.5,-2.5)  {$=$};
\node (xx)at (5,-2.5)    {$E_{\lambda_2}$};
\node (X) at (0, -3.75)   {0};
\node (Y) at (0.5,-3.75)  {$\subset$};
\node (Z) at (1,-3.75)    {$E_{\lambda_3}^1$};
\node (AA) at (1.5,-3.75)  {$\subset$};
\node (BB) at (2,-3.75)    {$E_{\lambda_3}^2$};
\node (CC) at (2.5,-3.75)  {$\subset$};
\node (DD) at (3,-3.75)    {$E_{\lambda_3}^3$};
\node (EE) at (3.5,-3.75)  {$\subset$};
\node (FF) at (4,-3.75)    {$E_{\lambda_3}^4$};
\node (ff)at  (4.5,-3.75)  {$=$};
\node (gg)at  (5,-3.75)    {$E_{\lambda_3}$};
\node (g) at (0, -5)    {$0$};
\node (h) at (0.5,-5)   {$\subset$};
\node (i) at (1,-5)     {$E_{\lambda_4}^1$};
\node (j) at (1.5,-5)   {$\subset$};
\node (k) at (2,-5)     {$E_{\lambda_4}^2$};
\node (l) at (2.5,-5)   {$=$};
\node (m) at (3,-5)     {$E_{\lambda_4}$};

\draw[->] (b) to node [right] {} (J);
\draw[->] (J) to node [right] {} (Q);
\draw[->] (Q) to node [right] {} (X);
\draw[->] (N) to node [right] {} (U);
\draw[->] (BB) to node [right] {} (g);
\draw[->] (U) to node [right] {} (BB);
\draw[->] (W) to node [right] {} (BB);
\end{tikzpicture}
\end{center}
\end{table}%
\begin{remark}
By construction, $\rho_r$ induces an isomorphism over $C_r$. One can calculate the dimensions of the fibres for each orbit $C<C_r$ to verify the map is semi-small; see Appendix~\ref{Assec: Cr}.
In this way we find from the decomposition theorem that
\[ {\rho_r}_! \1_{\widetilde{C}_r}[\dim \widetilde{C}_r] 
= \IC(\1_{C_r}) \oplus \mathop{\bigoplus}\limits_{C<C_r} m(C,C_r) \IC(\1_{C}) . \]
The $C<C_r$ for which $m(C,C_r)\neq 0$ are:
\begin{center}
\begin{tikzpicture}
\node(A) at (0,0) {$2$};
\node(B) at (\rtwidth,0) {$4$};
\node(C) at (2*\rtwidth,0) {$4$};
\node(L) at (3*\rtwidth,0) {$4$};
\node(M) at (4*\rtwidth,0) {$2$};
\draw (0,-0.5*\rtheight) -- (4*\rtwidth,-0.5*\rtheight);
\node(D) at (0.5*\rtwidth,-\rtheight) {$2$};
\node(E) at (1.5*\rtwidth,-\rtheight) {$2$};
\node(d) at (2.5*\rtwidth,-\rtheight) {$2$};
\node(e) at (3.5*\rtwidth,-\rtheight) {$1$};
\node(F) at (\rtwidth,-2*\rtheight) {$0$};
\node(G) at (2*\rtwidth,-2*\rtheight) {$1$};
\node(H) at (3*\rtwidth,-2*\rtheight) {$1$};
\node(I) at (1.5*\rtwidth,-3*\rtheight) {$0$};
\node(J) at (2.5*\rtwidth,-3*\rtheight) {$0$};
\node(K) at (2*\rtwidth,-4*\rtheight) {$0$};
\end{tikzpicture}
\qquad
\begin{tikzpicture}
\node(A) at (0,0) {$2$};
\node(B) at (\rtwidth,0) {$4$};
\node(C) at (2*\rtwidth,0) {$4$};
\node(L) at (3*\rtwidth,0) {$4$};
\node(M) at (4*\rtwidth,0) {$2$};
\draw (0,-0.5*\rtheight) -- (4*\rtwidth,-0.5*\rtheight);
\node(D) at (0.5*\rtwidth,-\rtheight) {$2$};
\node(E) at (1.5*\rtwidth,-\rtheight) {$2$};
\node(d) at (2.5*\rtwidth,-\rtheight) {$2$};
\node(e) at (3.5*\rtwidth,-\rtheight) {$2$};
\node(F) at (\rtwidth,-2*\rtheight) {$0$};
\node(G) at (2*\rtwidth,-2*\rtheight) {$1$};
\node(H) at (3*\rtwidth,-2*\rtheight) {$1$};
\node(I) at (1.5*\rtwidth,-3*\rtheight) {$0$};
\node(J) at (2.5*\rtwidth,-3*\rtheight) {$0$};
\node(K) at (2*\rtwidth,-4*\rtheight) {$0$};
\end{tikzpicture}.
\end{center}
As we know by  Lemma~\ref{lem:Kittsix} that for these two $C$ that $\Evs_{C_\KS}\IC(\1_{C}) = 0$ we can conclude that  Equation~\eqref{eqn:Evs} ultimately takes the form
\begin{equation}\label{eqn:Evs-r}
 \Evs_{C_\KS} {\rho_r}_! \1_{\widetilde{C}_r}[\dim \widetilde{C}_r]  =  \Evs_{C_\KS} \IC(\1_{C_r}) 
\end{equation}
Below, we calculate the left hand side and find that $\Evs_{C_\KS} {\rho_r}_! \1_{\widetilde{C}_r}[\dim \widetilde{C}_r] =0$. From this it follows from Equation~\eqref{eqn:Evs}, without using the above, that we have
$ \Evs_{C_\KS} \IC(\1_{C_r})=0.$
\end{remark}


We now begin the calculation showing that the left-hand side of Equation~\eqref{eqn:Evs-r} is $0$. As in Section~\ref{ssec:overview}, we define $\tilde{f}_{r} : \widetilde{C}_{r} \times C^*_\KS \to \A^1$ by $\tilde{f}_{r}(x,E,y) = (x\vert y)$, where  $( - \, \vert \, - ) : V \times V^* \to \A^1$ is the dot-product pairing.

By a direct computation, we find that the fibre $\rho_r^{-1}(x_{\KS})$ of $\rho_r$ above $x_\KS$ is:
\[
E_{\lambda_1}^2=\begin{bmatrix}
1 & 0\\
0 & 1\\
0 & 0\\
0 & 0
\end{bmatrix},\quad
E_{\lambda_2}^1=\begin{bmatrix}
0\\
0\\
a\\
b
\end{bmatrix}, \quad
E_{\lambda_2}^3=\begin{bmatrix}
1 & 0 & 0\\
0 & 1 & 0\\
0 & 0 & a\\
0 & 0 & b
\end{bmatrix}
,\quad
E_{\lambda_3}^2=\begin{bmatrix}
1 & 0\\
0 & 1\\
0 & 0 \\
0 & 0
\end{bmatrix};
\]
thus, $\rho_r^{-1}(x_{\KS}) \iso \mathbb{P}^1$; we use the global coordinate $[a:b]$ for $\rho_r^{-1}(x_{\KS})$.

Introduce local coordinates for $\widetilde{C}_r$ for an affine chart $U$ that contains the affine part $a\ne 0$ of $\rho_r^{-1}(x_{\KS})$ under the isomorphism $\rho_r^{-1}(x_{\KS}) \iso \{ [a:b] \}$ above:
\[
E_{\lambda_1}^2=\begin{bmatrix}
1 & 0\\
0 & 1\\
a_1 & a_2\\
a_3 & a_4
\end{bmatrix},\quad
E_{\lambda_2}^1=\begin{bmatrix}
b_1\\
b_2\\
1\\
b_3
\end{bmatrix}, \quad
E_{\lambda_2}^3=\begin{bmatrix}
1 & 0 & 0\\
0 & 1 & 0\\
0 & 0 & 1\\
c_1 & c_2 & c_3
\end{bmatrix},\quad
E_{\lambda_3}^2=\begin{bmatrix}
1 & 0\\
0 & 1\\
d_1 & d_2\\
d_3 & d_4
\end{bmatrix}.
\]
Form the Jacobian for the equations that describe $U\times C_\KS^*$ and $\tilde{f}_r^{-1}(0)$. The generic rank of this Jacobian matrix is
\[
(\text{number of variables}) - \dim\tilde{f}_r^{-1}(0) = 110 - (35+32-1)=44.
\] 
We evaluate the  Jacobian at the generic point $(x_\KS,y_\KS(t_1,t_2))$, with $\diag(t_1,t_2)\in T'_\text{reg}$, and further add the conditions that describe $U\cap\rho_r^{-1}(x_{\KS})$. Now the Jacobian becomes a matrix $M_r$ over $\mathbb{Q}[t_1,t_2,c_3]$. Performing elementary row and column operations on $M_r$ results in a $45\times 45$ block diagonal matrix whose blocks consist of a $43\times 43$ identity matrix and the matrix 
\[
B_r\ceq\begin{bmatrix}
t_1c_3 + c_3^2 & -t_1-c_3\\
t_2c_3^2 & -t_2c_3
\end{bmatrix}.
\]
The rank of $M_r$ will drop from $44$ to $43$ if and only if $B_r=0$. If $B_r=0$ then $-t_1-c_3=0$ and $-t_2c_3=0$, so $t_1=-c_3$ and hence $t_1t_2=-t_2c_3=0$; 
but $t_1t_2\neq 0$ since $\diag(t_1,t_2)\in T'_\text{reg}$, so $B_r\neq 0$ and the rank of $M_r$ is $44$. This shows that 
\[
(\sing{\tilde f}_r^{-1}(0))\cap (U\times C_{\KS}^*)  \cap(\rho_r')^{-1}(x_\KS,y_\KS(t_1,t_2))=\emptyset.
\]

Since the fibre of $\rho_r$ above $x_\KS$ is isomorphic to $\mathbb{P}^1$, we must also check for singularities in an affine chart for $\widetilde{C}_r$ that contains the affine part $b\neq 0$ of $\rho_r^{-1}(x_{\KS})$. Smoothness is verified using the same approach; Appendix~\ref{Assec: Cr} shows how this calculation is made.

We conclude that $\tilde{f}_r^{-1}(0)$ is smooth on $(\rho')^{-1}(x_\KS,y_\KS(t_1t_2))$, and so, letting $H$ act we get
$
\text{sing}(\tilde{f}_r^{-1}(0))\cap \mathcal{\tilde O}_r = \emptyset.
$
It follows that $\Evs_{C_\KS} {\rho_r}_! \1_{\widetilde{C}_r}=0$.
As explained above, Equation~\eqref{r} follows immediately.

\subsection{$\Evs_{C_\KS}\IC(\1_{C_m})$}\label{ssec:case-m}

We now provide the proof of Equation~\eqref{m}; see Appendix~\ref{Assec: Cm} for a detailed description of the code used for the computations in this section.

As in Section~\ref{ssec:case-r}, we use the rank triangle for $C_m$ to construct a proper cover $\rho_m:\widetilde{C}_m\rightarrow\overline{C}_m$. 
The diagram appearing in Table~\ref{table:Cm} defines a subvariety of $\overline{C}_m\times\prod_{i=1}^4\mathcal{F}(E_{\lambda_i})$ by the relations:
\begin{align*}
&x_1(E_{\lambda_0})\subseteq E_{\lambda_1}^2 & &x_2(E_{\lambda_1})\subseteq E_{\lambda_2}^3 & &x_3(E_{\lambda_2})\subseteq E_{\lambda_3}^3\\
&x_2(E_{\lambda_1}^2)\subseteq E_{\lambda_2}^1 & &x_3(E_{\lambda_2}^3)\subseteq E_{\lambda_3}^2 & &x_4(E_{\lambda_3}^3)\subseteq E_{\lambda_4}^1\\
&x_3(E_{\lambda_2}^1)=0 & &x_4(E_{\lambda_3}^2)=0.
\end{align*}
Let $\widetilde{C}_m$ be the variety produced by removing the subspaces not appearing in the above relations; in this case, by removing $E_{\lambda_0}^1$, $E_{\lambda_1}^1$, $E_{\lambda_1}^3$,$E_{\lambda_2}^2$ and $E_{\lambda_3}^1$. 
Let $\rho_m : \widetilde{C}_m \to \overline{C}_m$ be the proper cover defined by projection $\rho_m(x,E) = x$ where $E$ is the partial flag $(E_{\lambda_0}, E_{\lambda_1}^2, E_{\lambda_1}, E_{\lambda_2}^1, E_{\lambda_2}^3, E_{\lambda_2}, E_{\lambda_3}^2, E_{\lambda_3}^3, E_{\lambda_4}^1)$.

\begin{table}[tp]
\caption{Equations defining $\widetilde{C}_m$.}
\label{table:Cm}
\begin{center}
\begin{tikzpicture}
\node (A) at (0, 0)    {$0$};
\node (B) at (0.5,0)   {$\subset$};
\node (C) at (1,0)     {$E_{\lambda_0}^1$};
\node (a) at (1.5,0)   {$\subset$};
\node (b) at (2,0)     {$E_{\lambda_0}^2$};
\node (D) at (2.5,0)   {$=$};
\node (E) at (3,0)     {$E_{\lambda_0}$};
\node (F) at (0, -1.25)   {0};
\node (G) at (0.5,-1.25)  {$\subset$};
\node (H) at (1,-1.25)    {$E_{\lambda_1}^1$};
\node (I) at (1.5,-1.25)  {$\subset$};
\node (J) at (2,-1.25)    {$E_{\lambda_1}^2$};
\node (K) at (2.5,-1.25)  {$\subset$};
\node (L) at (3,-1.25)    {$E_{\lambda_1}^3$};
\node (M) at (3.5,-1.25)  {$\subset$};
\node (N) at (4,-1.25)    {$E_{\lambda_1}^4$};
\node (u) at (4.5,-1.25)  {$=$};
\node (uu)at (5,-1.25)    {$E_{\lambda_1}$};
\node (O) at (0, -2.5)   {0};
\node (P) at (0.5,-2.5)  {$\subset$};
\node (Q) at (1,-2.5)    {$E_{\lambda_2}^1$};
\node (R) at (1.5,-2.5)  {$\subset$};
\node (S) at (2,-2.5)    {$E_{\lambda_2}^2$};
\node (T) at (2.5,-2.5)  {$\subset$};
\node (U) at (3,-2.5)    {$E_{\lambda_2}^3$};
\node (V) at (3.5,-2.5)  {$\subset$};
\node (W) at (4,-2.5)    {$E_{\lambda_2}^4$};
\node (ww)at (4.5,-2.5)  {$=$};
\node (xx)at (5,-2.5)    {$E_{\lambda_2}$};
\node (X) at (0, -3.75)   {0};
\node (Y) at (0.5,-3.75)  {$\subset$};
\node (Z) at (1,-3.75)    {$E_{\lambda_3}^1$};
\node (AA) at (1.5,-3.75)  {$\subset$};
\node (BB) at (2,-3.75)    {$E_{\lambda_3}^2$};
\node (CC) at (2.5,-3.75)  {$\subset$};
\node (DD) at (3,-3.75)    {$E_{\lambda_3}^3$};
\node (EE) at (3.5,-3.75)  {$\subset$};
\node (FF) at (4,-3.75)    {$E_{\lambda_3}^4$};
\node (ff)at  (4.5,-3.75)  {$=$};
\node (gg)at  (5,-3.75)    {$E_{\lambda_3}$};
\node (g) at (0, -5)    {$0$};
\node (h) at (0.5,-5)   {$\subset$};
\node (i) at (1,-5)     {$E_{\lambda_4}^1$};
\node (j) at (1.5,-5)   {$\subset$};
\node (k) at (2,-5)     {$E_{\lambda_4}^2$};
\node (l) at (2.5,-5)   {$=$};
\node (m) at (3,-5)     {$E_{\lambda_4}$};

\draw[->] (b) to node [right] {} (J);
\draw[->] (J) to node [right] {} (Q);
\draw[->] (Q) to node [right] {} (X);
\draw[->] (N) to node [right] {} (U);
\draw[->] (BB) to node [right] {} (g);
\draw[->] (U) to node [right] {} (BB);
\draw[->] (DD) to node [right] {} (i);
\draw[->] (W) to node [right] {} (DD);
\end{tikzpicture}
\end{center}
\end{table}

\begin{remark}
As in the previous case, by construction, $\rho_m$ is an isomorpism over $C_m$. We may again verify the map is semi-small, see Appendix~\ref{Assec: Cm}.
So again we have from the decomposition theorem that
\[ {\rho_m}_! \1_{\widetilde{C}_m}[\dim \widetilde{C}_m] 
= \IC(\1_{C_m}) \oplus \mathop{\bigoplus}\limits_{C<C_m} m(C,C_m) \IC(\1_{C}) . \]
 The $C<C_m$ for which $m(C,C_m)\neq 0$ are:

\begin{center}
\begin{tikzpicture}
\node(A) at (0,0) {$2$};
\node(B) at (\rtwidth,0) {$4$};
\node(C) at (2*\rtwidth,0) {$4$};
\node(L) at (3*\rtwidth,0) {$4$};
\node(M) at (4*\rtwidth,0) {$2$};
\draw (0,-0.5*\rtheight) -- (4*\rtwidth,-0.5*\rtheight);
\node(D) at (0.5*\rtwidth,-\rtheight) {$1$};
\node(E) at (1.5*\rtwidth,-\rtheight) {$2$};
\node(d) at (2.5*\rtwidth,-\rtheight) {$3$};
\node(e) at (3.5*\rtwidth,-\rtheight) {$2$};
\node(F) at (\rtwidth,-2*\rtheight) {$0$};
\node(G) at (2*\rtwidth,-2*\rtheight) {$1$};
\node(H) at (3*\rtwidth,-2*\rtheight) {$1$};
\node(I) at (1.5*\rtwidth,-3*\rtheight) {$0$};
\node(J) at (2.5*\rtwidth,-3*\rtheight) {$0$};
\node(K) at (2*\rtwidth,-4*\rtheight) {$0$};
\end{tikzpicture}
\quad
\begin{tikzpicture}
\node(A) at (0,0) {$2$};
\node(B) at (\rtwidth,0) {$4$};
\node(C) at (2*\rtwidth,0) {$4$};
\node(L) at (3*\rtwidth,0) {$4$};
\node(M) at (4*\rtwidth,0) {$2$};
\draw (0,-0.5*\rtheight) -- (4*\rtwidth,-0.5*\rtheight);
\node(D) at (0.5*\rtwidth,-\rtheight) {$1$};
\node(E) at (1.5*\rtwidth,-\rtheight) {$3$};
\node(d) at (2.5*\rtwidth,-\rtheight) {$2$};
\node(e) at (3.5*\rtwidth,-\rtheight) {$1$};
\node(F) at (\rtwidth,-2*\rtheight) {$0$};
\node(G) at (2*\rtwidth,-2*\rtheight) {$1$};
\node(H) at (3*\rtwidth,-2*\rtheight) {$1$};
\node(I) at (1.5*\rtwidth,-3*\rtheight) {$0$};
\node(J) at (2.5*\rtwidth,-3*\rtheight) {$0$};
\node(K) at (2*\rtwidth,-4*\rtheight) {$0$};
\end{tikzpicture}
\quad
\begin{tikzpicture}
\node(A) at (0,0) {$2$};
\node(B) at (\rtwidth,0) {$4$};
\node(C) at (2*\rtwidth,0) {$4$};
\node(L) at (3*\rtwidth,0) {$4$};
\node(M) at (4*\rtwidth,0) {$2$};
\draw (0,-0.5*\rtheight) -- (4*\rtwidth,-0.5*\rtheight);
\node(D) at (0.5*\rtwidth,-\rtheight) {$1$};
\node(E) at (1.5*\rtwidth,-\rtheight) {$3$};
\node(d) at (2.5*\rtwidth,-\rtheight) {$2$};
\node(e) at (3.5*\rtwidth,-\rtheight) {$2$};
\node(F) at (\rtwidth,-2*\rtheight) {$0$};
\node(G) at (2*\rtwidth,-2*\rtheight) {$1$};
\node(H) at (3*\rtwidth,-2*\rtheight) {$1$};
\node(I) at (1.5*\rtwidth,-3*\rtheight) {$0$};
\node(J) at (2.5*\rtwidth,-3*\rtheight) {$0$};
\node(K) at (2*\rtwidth,-4*\rtheight) {$0$};
\end{tikzpicture}
\quad
\begin{tikzpicture}
\node(A) at (0,0) {$2$};
\node(B) at (\rtwidth,0) {$4$};
\node(C) at (2*\rtwidth,0) {$4$};
\node(L) at (3*\rtwidth,0) {$4$};
\node(M) at (4*\rtwidth,0) {$2$};
\draw (0,-0.5*\rtheight) -- (4*\rtwidth,-0.5*\rtheight);
\node(D) at (0.5*\rtwidth,-\rtheight) {$1$};
\node(E) at (1.5*\rtwidth,-\rtheight) {$3$};
\node(d) at (2.5*\rtwidth,-\rtheight) {$3$};
\node(e) at (3.5*\rtwidth,-\rtheight) {$2$};
\node(F) at (\rtwidth,-2*\rtheight) {$0$};
\node(G) at (2*\rtwidth,-2*\rtheight) {$2$};
\node(H) at (3*\rtwidth,-2*\rtheight) {$1$};
\node(I) at (1.5*\rtwidth,-3*\rtheight) {$0$};
\node(J) at (2.5*\rtwidth,-3*\rtheight) {$0$};
\node(K) at (2*\rtwidth,-4*\rtheight) {$0$};
\end{tikzpicture}
\end{center}

\begin{center}
\begin{tikzpicture}
\node(A) at (0,0) {$2$};
\node(B) at (\rtwidth,0) {$4$};
\node(C) at (2*\rtwidth,0) {$4$};
\node(L) at (3*\rtwidth,0) {$4$};
\node(M) at (4*\rtwidth,0) {$2$};
\draw (0,-0.5*\rtheight) -- (4*\rtwidth,-0.5*\rtheight);
\node(D) at (0.5*\rtwidth,-\rtheight) {$2$};
\node(E) at (1.5*\rtwidth,-\rtheight) {$2$};
\node(d) at (2.5*\rtwidth,-\rtheight) {$2$};
\node(e) at (3.5*\rtwidth,-\rtheight) {$1$};
\node(F) at (\rtwidth,-2*\rtheight) {$0$};
\node(G) at (2*\rtwidth,-2*\rtheight) {$1$};
\node(H) at (3*\rtwidth,-2*\rtheight) {$1$};
\node(I) at (1.5*\rtwidth,-3*\rtheight) {$0$};
\node(J) at (2.5*\rtwidth,-3*\rtheight) {$0$};
\node(K) at (2*\rtwidth,-4*\rtheight) {$0$};
\end{tikzpicture}
\quad
\begin{tikzpicture}
\node(A) at (0,0) {$2$};
\node(B) at (\rtwidth,0) {$4$};
\node(C) at (2*\rtwidth,0) {$4$};
\node(L) at (3*\rtwidth,0) {$4$};
\node(M) at (4*\rtwidth,0) {$2$};
\draw (0,-0.5*\rtheight) -- (4*\rtwidth,-0.5*\rtheight);
\node(D) at (0.5*\rtwidth,-\rtheight) {$2$};
\node(E) at (1.5*\rtwidth,-\rtheight) {$2$};
\node(d) at (2.5*\rtwidth,-\rtheight) {$2$};
\node(e) at (3.5*\rtwidth,-\rtheight) {$2$};
\node(F) at (\rtwidth,-2*\rtheight) {$0$};
\node(G) at (2*\rtwidth,-2*\rtheight) {$1$};
\node(H) at (3*\rtwidth,-2*\rtheight) {$1$};
\node(I) at (1.5*\rtwidth,-3*\rtheight) {$0$};
\node(J) at (2.5*\rtwidth,-3*\rtheight) {$0$};
\node(K) at (2*\rtwidth,-4*\rtheight) {$0$};
\end{tikzpicture}
\quad
\begin{tikzpicture}
\node(A) at (0,0) {$2$};
\node(B) at (\rtwidth,0) {$4$};
\node(C) at (2*\rtwidth,0) {$4$};
\node(L) at (3*\rtwidth,0) {$4$};
\node(M) at (4*\rtwidth,0) {$2$};
\draw (0,-0.5*\rtheight) -- (4*\rtwidth,-0.5*\rtheight);
\node(D) at (0.5*\rtwidth,-\rtheight) {$2$};
\node(E) at (1.5*\rtwidth,-\rtheight) {$2$};
\node(d) at (2.5*\rtwidth,-\rtheight) {$3$};
\node(e) at (3.5*\rtwidth,-\rtheight) {$2$};
\node(F) at (\rtwidth,-2*\rtheight) {$0$};
\node(G) at (2*\rtwidth,-2*\rtheight) {$2$};
\node(H) at (3*\rtwidth,-2*\rtheight) {$1$};
\node(I) at (1.5*\rtwidth,-3*\rtheight) {$0$};
\node(J) at (2.5*\rtwidth,-3*\rtheight) {$0$};
\node(K) at (2*\rtwidth,-4*\rtheight) {$0$};
\end{tikzpicture}
\end{center}
Again we know by  Lemma~\ref{lem:Kittsix} that for these $C$ that $\Evs_{C_\KS}\IC(\1_{C}) = 0$ hence we can conclude that  Equation~\eqref{eqn:Evs} ultimately takes the form
\begin{equation}\label{eqn:Evs-m}
 \Evs_{C_\KS} {\rho_m}_! \1_{\widetilde{C}_m}[\dim \widetilde{C}_m]  =  \Evs_{C_\KS} \IC(\1_{C_m}) .
\end{equation}
Below, we calculate the left hand side and find that $\Evs_{C_\KS} {\rho_m}_! \1_{\widetilde{C}_m}[\dim \widetilde{C}_m] =0$ and again conclude that we have
$ \Evs_{C_\KS} \IC(\1_{C_m})=0.$
\end{remark}


By a direct computation, we find that the fibre $\rho_m^{-1}(x_{\KS})$ of $\rho_m$ above $x_\KS$ is:
\[
\begin{array}{lll}
E_{\lambda_1}^2 = \begin{bmatrix}
1 & 0\\
0 & 1\\
0 & 0\\
0 & 0
\end{bmatrix},\quad&
E_{\lambda_2}^1=\begin{bmatrix}
0\\
0\\
a\\
b
\end{bmatrix} ,\quad&
E_{\lambda_2}^3=\begin{bmatrix}
1 & 0 & 0\\
0 & 1 & 0\\
0 & 0 & a\\
0 & 0 & b
\end{bmatrix},
\\\\[-3pt]
E_{\lambda_3}^2=\begin{bmatrix}
1 & 0\\
0 & 1\\
0 & 0\\
0 & 0
\end{bmatrix}, \quad&
 E_{\lambda_3}^3=\begin{bmatrix}
1 & 0 & 0\\
0 & 1 & 0\\
0 & 0 & c\\
0 & 0 & d
\end{bmatrix},\quad&
E_{\lambda_4}^1=\begin{bmatrix}
c\\
d
\end{bmatrix};
\end{array}
\]
thus, $\rho_m^{-1}(x_{\KS})\cong\mathbb{P}^1\times\mathbb{P}^1$.
%
Introduce local coordinates for $\widetilde{C}_R$ for an affine chart $U$ that contains the affine part $a\ne 0$ and $c\neq 0$ of $\rho_m^{-1}(x_{\KS})$ under the isomorphism $\rho_m^{-1}(x_{\KS}) \iso \{ [a:b] \}\times \{[c:d]\}$ above:
\[\begin{array}{lll}
E_{\lambda_1}^2=\begin{bmatrix}
1 & 0\\
0 & 1\\
a_1 & a_2\\
a_3 & a_4
\end{bmatrix},\quad&
E_{\lambda_2}^1=\begin{bmatrix}
b_1\\
b_2\\
1\\
b_3
\end{bmatrix}, \quad&
E_{\lambda_2}^3=\begin{bmatrix}
1 & 0 & 0\\
0 & 1 & 0\\
0 & 0 & 1\\
c_1 & c_2 & c_3
\end{bmatrix},
\\\\[-3pt]
E_{\lambda_3}^2=\begin{bmatrix}
1 & 0\\
0 & 1\\
d_1 & d_2\\
d_3 & d_4
\end{bmatrix}, \quad&
E_{\lambda_3}^3=\begin{bmatrix}
1 & 0 & 0\\
0 & 1 & 0\\
0 & 0 & 1\\
e_1 & e_2 & e_3
\end{bmatrix},\quad&
E_{\lambda_4}^1=\begin{bmatrix}
1\\
g
\end{bmatrix}.\end{array}
\]
Next, we form the Jacobian for the equations that describe $U\times C_\KS^*$ and $\tilde{f}_m^{-1}(0)$. The generic rank of this Jacobian matrix is
\[
(\text{number of variables}) - \dim\tilde{f}_m^{-1}(0) = 114 - (37+32-1)=46.
\] 
We evaluate the  Jacobian at the generic point $(x_\KS,y_\KS(t_1,t_2))$, with $\diag(t_1,t_2)\in T'_\text{reg}$, and further add the conditions that describe $U\cap\rho_m^{-1}(x_{\KS})$. 
We are left with a matrix $M_m$ over $\mathbb{Q}[t_1,t_2,c_3,g]$. Performing elementary row and column operations on $M_m$ results in a $48\times 47$ block diagonal matrix whose blocks consists of a $45\times 45$ identity matrix and the matrix

\[
B_m\ceq\begin{bmatrix}
c_3-g & -c_3+g\\
-t_1-c_3 & t_1+c_3\\
-t_2c_3 & t_2c_3
\end{bmatrix}.
\]

\noindent The rank of $M_m$ will drop from $46$ to $45$ if and only if $B_m=0$. If $B_m=0$ we get the relations $ t_1+c_3=0$ and $t_2c_3=0$, so $t_1=-c_3$ and hence $t_1t_2=-t_2c_3=0$; but, $t_1t_2\neq 0$ since $\diag(t_1,t_2)\in T'_\text{reg}$, so $B_m\neq 0$ and the rank of $M_m$ is $46$. This shows that 
\[
(\sing{\tilde f}_m^{-1}(0))\cap (U\times C_{\KS}^*)  \cap(\rho_m')^{-1}(x_\KS,y_\KS(t_1,t_2))=\emptyset.
\]

Since $\rho_m^{-1}(x_\KS)\iso \mathbb{P}^1\times\mathbb{P}^1$ we must also check for singularities in a collection of affine charts for $\widetilde{C}_m$ that, together with $U$, cover $\widetilde{C}_m\cap\rho_m^{-1}(x_\KS)$. Smoothness is verified using the same method above, but with a different affine choice for the $E_{\lambda_i}^k$ in $\widetilde{C}_m$. There are $3$ such affine choices that need to be made; Appendix~\ref{Assec: Cm} shows how these extra calculations were made.
We conclude that $\tilde{f}_m^{-1}(0)$ is smooth on $\mathcal{\tilde O}_m$, that is,
$
\text{sing}(\tilde{f}_m^{-1}(0))\cap\mathcal{\tilde O}_m=\emptyset.
$
From this, Equation~\eqref{m} follows immediately:
$
\Evs_{C_\KS}\IC(\1_{C_m})=0.
$

\subsection{$\Evs_{C_\KS}\IC(\1_{C_R})$}\label{ssec:case-R}

In this section we prove Equation~\eqref{R}; see Appendix~\ref{Assec: CR} for a detailed description of the code used for the computations in this section. 

Recall the rank triangle for $C_R$ from Table~\ref{table:Kittsix}.
\begin{table}[tp]
\caption{Equations defining $\widetilde{C}_R$.}
\label{table:CR}
\begin{center}
\begin{tikzpicture}
\node (A) at (0, 0)    {$0$};
\node (B) at (0.5,0)   {$\subset$};
\node (C) at (1,0)     {$E_{\lambda_0}^1$};
\node (a) at (1.5,0)   {$\subset$};
\node (b) at (2,0)     {$E_{\lambda_0}^2$};
\node (D) at (2.5,0)   {$=$};
\node (E) at (3,0)     {$E_{\lambda_0}$};
\node (F) at (0, -1.25)   {0};
\node (G) at (0.5,-1.25)  {$\subset$};
\node (H) at (1,-1.25)    {$E_{\lambda_1}^1$};
\node (I) at (1.5,-1.25)  {$\subset$};
\node (J) at (2,-1.25)    {$E_{\lambda_1}^2$};
\node (K) at (2.5,-1.25)  {$\subset$};
\node (L) at (3,-1.25)    {$E_{\lambda_1}^3$};
\node (M) at (3.5,-1.25)  {$\subset$};
\node (N) at (4,-1.25)    {$E_{\lambda_1}^4$};
\node (u) at (4.5,-1.25)  {$=$};
\node (uu)at (5,-1.25)    {$E_{\lambda_1}$};
\node (O) at (0, -2.5)   {0};
\node (P) at (0.5,-2.5)  {$\subset$};
\node (Q) at (1,-2.5)    {$E_{\lambda_2}^1$};
\node (R) at (1.5,-2.5)  {$\subset$};
\node (S) at (2,-2.5)    {$E_{\lambda_2}^2$};
\node (T) at (2.5,-2.5)  {$\subset$};
\node (U) at (3,-2.5)    {$E_{\lambda_2}^3$};
\node (V) at (3.5,-2.5)  {$\subset$};
\node (W) at (4,-2.5)    {$E_{\lambda_2}^4$};
\node (ww)at (4.5,-2.5)  {$=$};
\node (xx)at (5,-2.5)    {$E_{\lambda_2}$};
\node (X) at (0, -3.75)   {0};
\node (Y) at (0.5,-3.75)  {$\subset$};
\node (Z) at (1,-3.75)    {$E_{\lambda_3}^1$};
\node (AA) at (1.5,-3.75)  {$\subset$};
\node (BB) at (2,-3.75)    {$E_{\lambda_3}^2$};
\node (CC) at (2.5,-3.75)  {$\subset$};
\node (DD) at (3,-3.75)    {$E_{\lambda_3}^3$};
\node (EE) at (3.5,-3.75)  {$\subset$};
\node (FF) at (4,-3.75)    {$E_{\lambda_3}^4$};
\node (ff)at  (4.5,-3.75)  {$=$};
\node (gg)at  (5,-3.75)    {$E_{\lambda_3}$};
\node (g) at (0, -5)    {$0$};
\node (h) at (0.5,-5)   {$\subset$};
\node (i) at (1,-5)     {$E_{\lambda_4}^1$};
\node (j) at (1.5,-5)   {$\subset$};
\node (k) at (2,-5)     {$E_{\lambda_4}^2$};
\node (l) at (2.5,-5)   {$=$};
\node (m) at (3,-5)     {$E_{\lambda_4}$};

\draw[->] (b) to node [right] {} (J);
\draw[->] (J) to node [right] {} (S);
\draw[->] (S) to node [right] {} (X);
\draw[->] (BB) to node [right] {} (g);
\draw[->] (W) to node [right] {} (BB);
\end{tikzpicture}
\end{center}
\end{table}
As in Section~\ref{ssec:case-r}, we use this rank triangle to construct a proper cover $\rho_R:\widetilde{C}_R\rightarrow\overline{C}_R$ which is semi-small, and realizes $\widetilde{C}_R$ as a smooth variety. The diagram appearing in Table~\ref{table:CR} defines a subvariety of $\overline{C}_R\times\prod_{i=0}^4\mathcal{F}(E_{\lambda_i})$ by the relations:
\begin{align*}
&x_1(E_{\lambda_0})\subseteq E_{\lambda_1}^2 & &x_3(E_{\lambda_2}) \subseteq E_{\lambda_3}^2\\
&x_2(E_{\lambda_1}^2)\subseteq E_{\lambda_2}^2 & &x_4(E_{\lambda_3}^2)=0\\
&x_3(E_{\lambda_2}^2)=0.
\end{align*}
We remove the subspaces not appearing in the relations above; in this case, remove the vector spaces $E_{\lambda_0}^1, E_{\lambda_1}^1, E_{\lambda_1}^3,E_{\lambda_2}^1,E_{\lambda_2}^3, E_{\lambda_3}^1, E_{\lambda_3}^3$ and $E_{\lambda_4}^1$. Let $\widetilde{C}_R$ be the subvariety in the product of $\overline{C}_R$ and this partial flag variety defined by the equations above and summarized in Table~\ref{table:CR}. For elements in $\widetilde{C}_R$ we use the notation $(x,E)$, where $x\in\overline{C}_R$ and $E$ is a point in the partial flag variety described above.

\begin{remark}
As in the previous case, by construction, $\rho_R$ is an isomorpism over $C_R$. In this case, the map is small and hence $m_i(C,C_R)= 0$ for all $C<C_R$; see Appendix~\ref{Assec: CR}.
\end{remark}
Once again, Equation~\eqref{eqn:Evs} takes the form
\begin{equation}\label{eqn:Evs-R}
 \Evs_{C_\KS} {\rho_R}_! \1_{\widetilde{C_R}}[\dim \widetilde{C_R}] 
= \Evs_{C_\KS} \IC(\1_{C_R}) 
\end{equation} 
By a direct computation, we find that the fibre $\rho_R^{-1}(x_{\KS})$ of $\rho_R$ above $x_\KS$ is:
\[
E_{\lambda_1}^2=\begin{bmatrix}
1 & 0\\
0 & 1\\
0 & 0\\
0 & 0
\end{bmatrix} ,
\quad
E_{\lambda_2}^2=\begin{bmatrix}
0 & 0\\
0 & 0\\
1 & 0\\
0 & 1
\end{bmatrix},
\quad
E_{\lambda_3}^2=
\begin{bmatrix}
1 & 0\\
0 & 1\\
0 & 0\\
0 & 0
\end{bmatrix}.
\]
And so, $\rho_R : \widetilde{C}_{R} \to \overline{C}_{R}$ is an isomorphism over $C_\KS$. 
Now choose local coordinates by introducing the following variables for an affine chart $U$ in $\widetilde{C}_R$ that contains $\rho_R^{-1}(x_{\KS})$, by specifying flags $E_{\lambda_i}^k$ as follows:
\[
E_{\lambda_1}^2=\begin{bmatrix}
1 & 0\\
0 & 1\\
a_1 & a_2\\
a_3 & a_4
\end{bmatrix},
\quad
E_{\lambda_2}^2=
\begin{bmatrix}
b_1 & b_2\\
b_3 & b_4\\
1 & 0\\
0 & 1
\end{bmatrix},
\quad
E_{\lambda_3}^2=
\begin{bmatrix}
1 & 0\\
0 & 1\\
c_1 & c_2\\
c_3 & c_4
\end{bmatrix}.
\]
We compute the Jacobian for the system of equations that describe $U \times C_\KS^*$ and $\tilde{f}_{R}^{-1}(0)$.  
The generic rank of this Jacobian is 
\[
(\text{number of variables}) - \dim\tilde{f}_R^{-1}(0) =  108 - (36 + 32 -1) = 41.
\]
We evaluate the  Jacobian at the generic point $(x_\KS,y_\KS(t_1,t_2))$, with $\diag(t_1,t_2)\in T'_\text{reg}$, and further add the conditions that correspond to $\rho_R(x_\KS,E) = x_{\KS}$.
We are left with a matrix $M_R$ over $\mathbb{Q}[t_1,t_2]$. Performing elementary row and column operations on $M_R$ results in a $41\times 41$ identity matrix. We conclude that $M_R$ has rank $41$.
In this case, it is not necessary to consider more than one affine $U\subset {\widetilde C_R}$, so
 $\tilde{f}_{R}^{-1}(0)$ is smooth on $(\rho_R')^{-1}(x_\KS,y_\KS(t_1,t_2))$, and so
$
(\sing\tilde{f}^{-1}_R(0))\cap \mathcal{\tilde O}_R = \emptyset.
$
From this, Equation~\eqref{R} follows immediately:
$
\Evs_{C_\KS}\IC(\1_{C_R})=0.
$

\subsection{$\Evs_{C_\KS}\IC(\1_{C_\psi})$}\label{ssec:case-psi}

In this section we prove Equation~\eqref{psi}; see Appendix~\ref{Assec: Cpsi} for a detailed description of the code used for the computations in this section.
 
Recall the rank triangle for $C_\psi$ from Table~\ref{table:Kittsix}. We use this rank triangle to define $\rho_\psi:\widetilde{C}_\psi\to\overline{C}_\psi$ as follows.

\begin{table}[tp]
\caption{Equations defining $\widetilde{C}_\psi$.}
\label{table:Cpsi}
\begin{center}
\begin{tikzpicture}
\node (A) at (0, 0)    {$0$};
\node (B) at (0.5,0)   {$\subset$};
\node (C) at (1,0)     {$E_{\lambda_0}^1$};
\node (a) at (1.5,0)   {$\subset$};
\node (b) at (2,0)     {$E_{\lambda_0}^2$};
\node (D) at (2.5,0)   {$=$};
\node (E) at (3,0)     {$E_{\lambda_0}$};
\node (F) at (0, -1.25)   {0};
\node (G) at (0.5,-1.25)  {$\subset$};
\node (H) at (1,-1.25)    {$E_{\lambda_1}^1$};
\node (I) at (1.5,-1.25)  {$\subset$};
\node (J) at (2,-1.25)    {$E_{\lambda_1}^2$};
\node (K) at (2.5,-1.25)  {$\subset$};
\node (L) at (3,-1.25)    {$E_{\lambda_1}^3$};
\node (M) at (3.5,-1.25)  {$\subset$};
\node (N) at (4,-1.25)    {$E_{\lambda_1}^4$};
\node (u) at (4.5,-1.25)  {$=$};
\node (uu)at (5,-1.25)    {$E_{\lambda_1}$};
\node (O) at (0, -2.5)   {0};
\node (P) at (0.5,-2.5)  {$\subset$};
\node (Q) at (1,-2.5)    {$E_{\lambda_2}^1$};
\node (R) at (1.5,-2.5)  {$\subset$};
\node (S) at (2,-2.5)    {$E_{\lambda_2}^2$};
\node (T) at (2.5,-2.5)  {$\subset$};
\node (U) at (3,-2.5)    {$E_{\lambda_2}^3$};
\node (V) at (3.5,-2.5)  {$\subset$};
\node (W) at (4,-2.5)    {$E_{\lambda_2}^4$};
\node (ww)at (4.5,-2.5)  {$=$};
\node (xx)at (5,-2.5)    {$E_{\lambda_2}$};
\node (X) at (0, -3.75)   {0};
\node (Y) at (0.5,-3.75)  {$\subset$};
\node (Z) at (1,-3.75)    {$E_{\lambda_3}^1$};
\node (AA) at (1.5,-3.75)  {$\subset$};
\node (BB) at (2,-3.75)    {$E_{\lambda_3}^2$};
\node (CC) at (2.5,-3.75)  {$\subset$};
\node (DD) at (3,-3.75)    {$E_{\lambda_3}^3$};
\node (EE) at (3.5,-3.75)  {$\subset$};
\node (FF) at (4,-3.75)    {$E_{\lambda_3}^4$};
\node (ff)at  (4.5,-3.75)  {$=$};
\node (gg)at  (5,-3.75)    {$E_{\lambda_3}$};
\node (g) at (0, -5)    {$0$};
\node (h) at (0.5,-5)   {$\subset$};
\node (i) at (1,-5)     {$E_{\lambda_4}^1$};
\node (j) at (1.5,-5)   {$\subset$};
\node (k) at (2,-5)     {$E_{\lambda_4}^2$};
\node (l) at (2.5,-5)   {$=$};
\node (m) at (3,-5)     {$E_{\lambda_4}$};

\draw[->] (b) to node [right] {} (J);
\draw[->] (J) to node [right] {} (Q);
\draw[->] (Q) to node [right] {} (Z);
\draw[->] (Z) to node [right] {} (g);
\draw[->] (N) to node [right] {} (U);
\draw[->] (U) to node [right] {} (BB);
\draw[->] (BB) to node [right] {} (i);
\draw[->] (W) to node [right] {} (DD);
\draw[->] (DD) to node [right] {} (i);
\end{tikzpicture}
\hspace{0.5cm}
\end{center}
\end{table}

\begin{enumerate}
\item
Looking at the second row of the rank triangle for $C_\psi$ we see the ranks $(2,3,3,2)$. This means, for any $x=(x_4,x_3,x_2,x_1)\in C_\psi$, we have $\rank  x_1 = 2$, $\rank x_2 = 3$, $\rank x_3 = 3$ and $\rank x_4 = 2$. 
\begin{enumerate}
\item Since $x_1 : E_{\lambda_0} \to E_{\lambda_1}$ and $\rank x_1=2$, we declare $x_1(E_{\lambda_0}) \subseteq E_{\lambda_1}^2$.
\item Since $x_2 : E_{\lambda_1} \to E_{\lambda_2}$ and $\rank x_2 = 3$, we declare $x_2(E_{\lambda_1}) \subseteq E_{\lambda_2}^3$.
\item Since $x_3 : E_{\lambda_2} \to E_{\lambda_3}$ and $\rank x_3 = 3$, we declare $x_3(E_{\lambda_2})  \subseteq E_{\lambda_3}^3$.
\item Since $x_4 : E_{\lambda_3} \to E_{\lambda_4}$ and $\rank x_4 = 2$, we declare $x_4(E_{\lambda_3})  \subseteq E_{\lambda_4}^2$; of course, this is no condition at all since $E_{\lambda_4}^2=E_{\lambda_4}$.
\end{enumerate}
\item
Looking at the third row of the rank triangle for $C_\psi$ we see the ranks $(1,2,1)$. This means $\rank  x_2 x_1 \leq 1$, $\rank x_3 x_2 \leq 2$ and $\rank x_4 x_3 \leq 1$.
\begin{enumerate}
\item Since $x_2 x_1 : E_{\lambda_0} \to E_{\lambda_2}$ and $\rank x_2 x_1 \leq 1$, we declare $x_2(E_{\lambda_1}^2)  \subseteq E_{\lambda_2}^1$. This is compatible with  $x_2 x_1(E_{\lambda_0})  \subseteq E_{\lambda_2}^1$ and $x_1(E_{\lambda_0})  \subseteq E_{\lambda_1}^2$.
\item Since $x_3 x_2 : E_{\lambda_1} \to E_{\lambda_3}$ and $\rank x_3 x_2 \leq 2$, we declare $x_3(E_{\lambda_2}^3) \subseteq E_{\lambda_3}^2$. This is compatible with $x_3 x_2(E_{\lambda_1})  \subseteq E_{\lambda_3}^2$ and $x_2(E_{\lambda_1})  \subseteq E_{\lambda_2}^3$.
\item Since $x_4 x_3 : E_{\lambda_2} \to E_{\lambda_4}$ and $\rank x_4 x_3 \leq 1$, we declare $x_4(E_{\lambda_3}^3)  \subseteq E_{\lambda_4}^1$. This is compatible with $x_4 x_3(E_{\lambda_2}) \subseteq E_{\lambda_4}^1$ and $x_3(E_{\lambda_2})  \subseteq E_{\lambda_3}^3$.
\end{enumerate}
\item
Looking at the fourth row of the rank triangle for $C_\psi$ we see the ranks $(1,1)$. This means $\rank  x_3 x_2 x_1 \leq 1$ and $\rank x_4 x_3 x_2 \leq 1$.
\begin{enumerate}
\item Since $x_3 x_2 x_1 : E_{\lambda_0} \to E_{\lambda_3}$ and $\rank x_3 x_2 x_1 \leq 1$, we declare $x_3(E_{\lambda_2}^1)  \subseteq E_{\lambda_3}^1$. This is compatible with $x_3 x_2 x_1(E_{\lambda_0})  \subseteq E_{\lambda_3}^1$ and $x_2 x_1(E_{\lambda_0})  \subseteq E_{\lambda_2}^1$.
\item Since $x_4 x_3 x_2 : E_{\lambda_1} \to E_{\lambda_4}$ and $\rank x_4 x_3 x_2  \leq 1$, we declare $x_4(E_{\lambda_3}^2)  \subseteq E_{\lambda_4}^1$. This is compatible with $x_4 x_3 x_2(E_{\lambda_1} )  \subseteq E_{\lambda_4}^1$ and $x_3 x_2(E_{\lambda_1}) \subseteq E_{\lambda_3}^2$.
\end{enumerate}
\item
Looking at the bottom entry in the rank triangle for $C_\psi$ we see the rank $(0)$. This means $\rank  x_4 x_3 x_2 x_1 \leq 0$ or equivalently, $x_4 x_3 x_2 x_1 = 0$. We declare $x_4(E_{\lambda_3}^1)=0$, which is compatible with $x_3x_2x_1(E_{\lambda_0})\subseteq E_{\lambda_3}^1$ and $x_4x_3x_2x_1=0$.
\end{enumerate}
Now consider the subvariety of $\overline{C}_\psi\times \prod_{i=0}^4 \mathcal{F}(E_{\lambda_i})$ defined by the relations:
\begin{align*}
&x_1(E_{\lambda_0}) \subseteq E_{\lambda_1}^2 & &x_2(E_{\lambda_1}) \subseteq E_{\lambda_2}^3 & &x_3(E_{\lambda_2})  \subseteq E_{\lambda_3}^3\\
&x_2(E_{\lambda_1}^2)  \subseteq E_{\lambda_2}^1 & &x_3(E_{\lambda_2}^3) \subseteq E_{\lambda_3}^2 & &x_4(E_{\lambda_3}^3)  \subseteq E_{\lambda_4}^1\\
&x_3(E_{\lambda_2}^1)  \subseteq E_{\lambda_3}^1 & &x_4(E_{\lambda_3}^2)  \subseteq E_{\lambda_4}^1 \\
&x_4(E_{\lambda_3}^1)=0.
\end{align*}
We remove the subspaces not appearing in the relations above; in this case, remove $E_{\lambda_0}^1$, $E_{\lambda_1}^1$, $E_{\lambda_1}^3$ and $E_{\lambda_2}^2$. Let $\widetilde{C}_\psi$ be the subvariety in the product of $\overline{C}_\psi$ and this partial flag variety defined by the equations above and summarized in Table~\ref{table:Cpsi}.
We define $\rho_\psi(x,E) = x$ where $E$ is a partial flag.

By a direct computation, we find that the fibre $\rho_\psi^{-1}(x_{\KS})$ of $\rho_\psi$ above $x_\KS$ is
\[
E_{\lambda_1}^2=\begin{bmatrix}
1 & 0\\
0 & 1\\
0 & 0\\
0 & 0
\end{bmatrix},\quad
E_{\lambda_2}^1=\begin{bmatrix}
u\\
v\\
w\\
z
\end{bmatrix}, \quad
E_{\lambda_2}^3=\begin{bmatrix}
1 & 0 & 0\\
0 & 1 & 0\\
0 & 0 & c\\
0 & 0 & d
\end{bmatrix},
\]
\[
E_{\lambda_3}^1=\begin{bmatrix}
a\\
b\\
0\\
0
\end{bmatrix}, \quad
E_{\lambda_3}^2=\begin{bmatrix}
1 & 0\\
0 & 1\\
0 & 0\\
0 & 0 
\end{bmatrix}, \quad
E_{\lambda_3}^3=\begin{bmatrix}
1 & 0 & 0\\
0 & 1 & 0\\
0 & 0 & e\\
0 & 0 & h
\end{bmatrix},\quad
E_{\lambda_4}^1=\begin{bmatrix}
e\\
h
\end{bmatrix},
\]
subject to the conditions $av-bu=0$ and $cz-dw=0$; this is the product of $\mathbb{P}^1\cong \{ [e:h] \tq (e,h)\ne (0,0)\} $ with the variety
\[
F\ceq \{([a:b], [c:d], [u:v:w:z]) \in \mathbb{P}^1\times \mathbb{P}^1\times \mathbb{P}^3 \tq av-bu=0,\ cz-dw=0 \},
\]
which is the blow up of  $\mathbb{P}^3$ along the two copies of $\mathbb{P}^1$ given by $(u,v)=(0,0)$ and $(w,z)=(0,0)$.

\begin{remark}\label{rem:check15}
Again, by construction, $\rho_\psi$ is an isomorpism over $C_\psi$. 
For the 15 orbits with $C_\KS < C < C_\psi$ we must verify that the cover is small in the sense of decomposition theorem. This check is performed using code as documented in Appendix~\ref{Assec: Cpsi}.
A cohomology calculation on the fibre above then gives that $m_0(C_\KS,C_\psi)=1$ and $m_i(C_\KS,C_\psi) = 0$ for $i\neq 0$. 
By checking all other orbits we verify the map is again semi-small. We find that there are no other $C<C_\psi$ for which $m_i(C,C_\psi)\neq 0$.
\end{remark}
Using that $ \Evs_{C_\KS} \IC(\1_{C_\KS})  =   \1_{\Lambda^\text{gen}_{C_\KS}}$,  Equation~\eqref{eqn:Evs} takes the form
\begin{equation}\label{eqn:Evs-psi}
\Evs_{C_\KS} {\rho_\psi}_! \1_{\widetilde{C}_\psi}[\dim \widetilde{C}_\psi] 
= \Evs_{C_\KS} \IC(\1_{C_\psi}) \oplus \1_{\Lambda^\text{gen}_{C_\KS}}.
\end{equation}
We calculate the left-hand side of Equation~\eqref{eqn:Evs-psi}.
Introduce local coordinates for $\widetilde{C}_\psi$ for an affine chart $U$ that contains the affine part $a\ne 0, c\ne 0, e\ne 0,$ and $u\neq 0$ of $\rho_\psi^{-1}(x_{\KS})$ under the isomorphism $\rho_\psi^{-1}(x_{\KS}) \iso \mathbb{P}^1\times F$ above:
\[
E_{\lambda_1}^2=\begin{bmatrix}
1 & 0\\
0 & 1\\
a_1 & a_2\\
a_3 & a_4
\end{bmatrix},\quad
E_{\lambda_2}^1=\begin{bmatrix}
1\\
b_1\\
b_2\\
b_3
\end{bmatrix}, \quad
E_{\lambda_2}^3=\begin{bmatrix}
1 & 0 & 0\\
0 & 1 & 0\\
0 & 0 & 1\\
c_1 & c_2 & c_3
\end{bmatrix},
\]
\[
E_{\lambda_3}^1=\begin{bmatrix}
1\\
d_1\\
d_2\\
d_3
\end{bmatrix}, \quad
E_{\lambda_3}^2=\begin{bmatrix}
1 & 0\\
0 & 1\\
e_1 & e_2\\
e_3 & e_4
\end{bmatrix}, \quad
E_{\lambda_3}^3=\begin{bmatrix}
1 & 0 & 0\\
0 & 1 & 0\\
0 & 0 & 1\\
f_1 & f_2 & f_3
\end{bmatrix},\quad
E_{\lambda_4}^1=\begin{bmatrix}
1\\
g
\end{bmatrix}.
\] 
Next, we form the Jacobian for the equations that describe $U\times C_\KS^\ast$ and $\tilde{f}_\psi^{-1}(0)$. The generic rank of this Jacobian matrix is
\[
(\text{number of variables}) - \dim\tilde{f}_\psi^{-1}(0) = 117 - (40+32-1)=46.
\] 
We evaluate the  Jacobian at the generic point $(x_\KS,y_\KS(t_1,t_2))$, with $\diag(t_1,t_2)\in T'_\text{reg}$, and further add the conditions that describe $U\cap\rho_\psi^{-1}(x_{\KS})$. 
We are left with a matrix $M_\psi$ over $\mathbb{Q}[t_1,t_2,b_1,b_2,c_3,f_3]$.
Performing row and column operations on $M_\psi$ results in a $52\times 51$ diagonal block matrix whose blocks consist of a $45\times 45$ identity matrix and a $7\times 6$ matrix $B_\psi$, displayed in the Section~\ref{Assec: Cpsi}.
The rank of $M_\psi$ will drop from its generic rank of $46$ to $45$ if and only if $B_\psi=0$. If $B_\psi=0$, then we get the system of equations:
\begin{gather*}
-b_1+f_3=0,\qquad
-c_3+f_3=0,\\
t_1b_2+b_2c_3+1=0,\qquad
t_2b_2c_3+b_1=0.
\end{gather*}
There are exactly two solutions to the above system: if $c_3=0$, then
\[
b_1=f_3=c_3=0, b_2=-t_1^{-1}
\]
is a solution; and if $c_3\neq 0$, then
\[
b_1=f_3=c_3=t_2-t_1, b_2=-t_2^{-1}
\]
is a solution. This calculation shows that, in $(U\times C_{\KS}^*)\cap(\rho_\psi')^{-1}(x_\KS,y_\KS(t_1,t_2))$, there are exactly two singularities of $\tilde{f}_\psi^{-1}(0)$: one of these singularities is 
\[
E_{\lambda_1}^2=\begin{bmatrix}
1 & 0\\
0 & 1\\
0 & 0\\
0 & 0
\end{bmatrix},\quad
E_{\lambda_2}^1=\begin{bmatrix}
t_1\\
0\\
-1\\
0
\end{bmatrix}, \quad
E_{\lambda_2}^3=\begin{bmatrix}
1 & 0 & 0\\
0 & 1 & 0\\
0 & 0 & 1\\
0 & 0 & 0
\end{bmatrix},
\]
\[
E_{\lambda_3}^1=\begin{bmatrix}
1\\
0\\
0\\
0
\end{bmatrix}, \quad
E_{\lambda_3}^2=\begin{bmatrix}
1 & 0\\
0 & 1\\
0 & 0\\
0 & 0 
\end{bmatrix}, \quad
E_{\lambda_3}^3=\begin{bmatrix}
1 & 0 & 0\\
0 & 1 & 0\\
0 & 0 & 1\\
0 & 0 & 0
\end{bmatrix},\quad
E_{\lambda_4}^1=\begin{bmatrix}
1\\
0
\end{bmatrix},
\]
while the other is
\[
E_{\lambda_1}^2=\begin{bmatrix}
1 & 0\\
0 & 1\\
0 & 0\\
0 & 0
\end{bmatrix},\quad
E_{\lambda_2}^1=\begin{bmatrix}
t_2\\
t_2(t_2-t_1)\\
-1\\
-(t_2-t_1)
\end{bmatrix}, \quad
E_{\lambda_2}^3=\begin{bmatrix}
1 & 0 & 0\\
0 & 1 & 0\\
0 & 0 & 1\\
0 & 0 & t_2-t_1
\end{bmatrix},
\]
\[
E_{\lambda_3}^1=\begin{bmatrix}
1\\
t_2-t_1\\
0\\
0
\end{bmatrix}, \quad
E_{\lambda_3}^2=\begin{bmatrix}
1 & 0\\
0 & 1\\
0 & 0\\
0 & 0 
\end{bmatrix}, \quad
E_{\lambda_3}^3=\begin{bmatrix}
1 & 0 & 0\\
0 & 1 & 0\\
0 & 0 & 1\\
0 & 0 & t_2-t_1
\end{bmatrix},\quad
E_{\lambda_4}^1=\begin{bmatrix}
1\\
t_2-t_1
\end{bmatrix}.
\]
Using the same approach, we check for singularities in a further 15 affine charts for $\widetilde{C}_\psi$ that cover $\widetilde{C}_\psi\cap\rho_\psi^{-1}(x_{\KS})$, corresponding to different choices for the subspaces $E_{\lambda_i}^k$ in $\widetilde{C}_\psi$. In all the other cases, no new singularities appear. 
We conclude that $\vert (\sing{\tilde f}_\psi^{-1}(0))\cap (\rho_\psi')^{-1}(x_\KS,y_\KS(t_1,t_2)) \vert = 2$ for every $\diag(t_1,t_2)\in T'_\text{reg}$.
Recall by  \cite{CFMMX}*{Proposition 7.20 and Lemma 7.21} that $(\sing{\tilde f}_\psi^{-1}(0)) \subset \overline{\mathcal{\tilde O}_\psi}$ so that by letting $H$ act on $ (\sing{\tilde f}_\psi^{-1}(0))\cap (\rho_\psi')^{-1}(x_\KS,y_\KS(t_1,t_2))$,
and checking dimensions it follows that 
\[
\rho_\psi''' : (\sing{\tilde f}_\psi^{-1}(0))\cap \mathcal{\tilde O}_\psi \to \Lambda^\text{gen}_{C_\KS}
\]
is a double cover.

Having found the singular locus of $\tilde{f}_\psi^{-1}(0)$ in $\mathcal{\tilde O}_\psi$, we next compute the rank of the Hessian of $\tilde{f}_\psi$ at a point in the singular locus and verify the determinant of a maximal minor of the Hessian is a square.  
This first calculation verifies the rank of the local system $\mathcal{M}_\psi$, defined by Equation~\eqref{eqn:M}, is $1$ and the later identifies the character of $A^\ABV_{\phi_\KS}$ associated to the local system.
We remark that because in an open neighbourhood of our point the singular locus of $\tilde f$ is contained in $\mathcal{\tilde O}_\psi$ the intersection with $ \mathcal{\tilde O}_\psi$ plays no role in these calculations.
We make these calculations by adapting the Hessian calculations described in Section~\ref{ssec:case-KS}, which in turn are adapted from {\cite{CFMMX}*{Theorem 7.19} and documented in Appendix \ref{app:Hessian}.
We use the Jacobian of the ideal defining $U\times C_\KS^*$ to identify a system of variables, from the coordinate ring of $U\times C_\KS^*$, that may be used as local coordinates for the (completed) local ring at the chosen point in the singular locus.
The rest of the computation follows the method of Section~\ref{ssec:case-KS}.
We find that the rank of the Hessian at the chosen point is $\dim C_\psi -\codim C^*_\KS  = 40-(48-32)=24$.
This shows that the rank of the Hessian is precisely the codimension (inside $U\times C_\KS^*$) of the singular locus, and from this we conclude that the rank of $\mathcal{M}_{\psi}$  is $1$. 
Though this was automatic in the cases under consideration in \cite{CFMMX}*{Theorem 7.19}, it should not be seen as automatic for this case as there are groups for which this will not occur.
Next, we verify that the Hessian determinant is the perfect square of a function so that we may conclude the representation associated to $\mathcal{M}_\psi$ is the trivial representation.
We note that these determinants are often not perfect squares, this occurs in several of the examples considered in \cite{CFMMX}.
We explain how to use our code to reproduce these calculations in Appendix \ref{app:Hessian}.
Unfortunately, these calculations are not particularly amenable to being displayed here.

With reference to Equation~\eqref{eqn:M}, it follows from these calculations that
\[
\mathcal{M}_\psi = \1_{(\sing{\tilde f}_\psi^{-1}(0))\cap \mathcal{\tilde O}_\psi}, 
\]
and therefore 
\begin{equation}\label{eqn:21}
\Evs_{C_\KS} {\rho_\psi}_! \1_{\widetilde{C}_\psi}[\dim \widetilde{C}_\psi]
=
{\rho'''_\psi}_! \1_{(\sing{\tilde f}_\psi^{-1}(0))\cap \mathcal{\tilde O}_\psi}\\
= 
\1_{\Lambda^\text{gen}_{C_\KS}} \oplus \mathcal{L}_{\Lambda^\text{gen}_{C_\KS}},
\end{equation}
where $\rank \mathcal{L}_{\Lambda^\text{gen}_{C_\KS}} =1$ and $\mathcal{L}_{\Lambda^\text{gen}_{C_\KS}}$ corresponds to a quadratic character of $A^\ABV_{\phi_\KS}$.  
Comparing \eqref{eqn:21} with \eqref{eqn:Evs-psi} gives
\begin{equation}
\Evs_{C_\KS}\IC(\1_{C_\psi}) \oplus \1_{\Lambda^{gen}_{C_\KS}} = \1_{\Lambda^{gen}_{C_\KS}} \oplus \mathcal{L}_{\Lambda^{gen}_{C_\KS}}.
\end{equation}
It follows that $\Evs_{C_\KS}\IC(\1_{C_\psi}) =\mathcal{L}_{\Lambda^{gen}_{C_\KS}}$,
where this local system is defined by the non-trivial character of the automorphism group of the double cover 
\[ \rho_\psi''' : (\sing{\tilde f}_\psi^{-1}(0))\cap \mathcal{\tilde O}_\psi \to \Lambda^\text{gen}_{C_\KS}. \]

\subsection{A study of quadratic covers of the generic conormal bundle}\label{ssec:L}

Two $H$-equivariant double covers of $\Lambda^\text{gen}_{C_\KS}$ appear in this paper: the first is 
\[
z' : \widetilde{\Lambda}^\text{gen}_{C_\KS} \ceq  {\Lambda}^\text{gen}_{C_\KS}\times_{Z'_\text{reg}} T'_\text{reg} \to {\Lambda}^\text{gen}_{C_\KS}
\]
appearing in Proposition~\ref{prop:L} and the second is 
\[
\rho_\psi''' : (\sing{\tilde f}_\psi^{-1}(0))\cap \mathcal{\tilde O}_\psi \to \Lambda^\text{gen}_{C_\KS}
\]
appearing in Section~\ref{ssec:case-psi}.
In this section we show that these two quadratic covers are isomorphic.
It will then follow that the local system $\mathcal{L}_{\Lambda^\text{gen}_{C_\KS}}$ appearing in Section~\ref{ssec:case-psi} is the local system defined by the non-trivial character of the automorphism group of the double cover $z' : \widetilde{\Lambda}^\text{gen}_{C_\KS}  \to {\Lambda}^\text{gen}_{C_\KS}$, thus concluding the proof of Theorem~\ref{thm:geometric}

From Section~\ref{ssec:generic} recall the cover $q: \Lambda^\text{gen}_{C_\KS} \to  Z'_\text{reg}$ given by $q(x,y) = [(ac)^{-1}(bd)]$, where $a,b,c,d\in\GL_2(\CC)$ are defined by Equation~\eqref{abcd}.
Define 
\[
\begin{array}{rcl}
{\dot q}: \Lambda^\text{gen}_{C_\KS} &\to&  \GL(E_{\lambda_3}/\ker y_3)\\
(x,y) &\mapsto& (ac)^{-1}(bd).
\end{array}
\]
Using this, consider the varieties 
\[
\Lambda''_{C_\KS}
\ceq
\left\{ (x,y,E^3_{\lambda_3}) \in \Lambda^\text{gen}_{C_\KS}\times\operatorname{Gr}^3(E_{\lambda_3})\ \Big\vert
\begin{array}{c} 
\ker y_3\subset E^3_{\lambda_3}\\
{\dot q}(x,y)(E^3_{\lambda_3}/\ker y_3) \subset E^3_{\lambda_3}/\ker y_3
\end{array}
\right\}
\]
and
\[
\Lambda'''_{C_\KS}
\ceq\!
\left\{ (x,y,E^3_{\lambda_3},\lambda) \in \Lambda^\text{gen}_{C_\KS}\times\operatorname{Gr}^3(E_{\lambda_3}) \times \mathbb{G}_\text{m}\,\Big\vert\!
\begin{array}{c} 
\ker y_3\subset E^3_{\lambda_3}\\
({\dot q}(x,y)\!-\!\lambda)(E^3_{\lambda_3}/\ker y_3) = 0
\end{array}\!
\right\}.
\]
Note that in the definition of $\Lambda''_{C_\KS}$ the condition ${\dot q}(x,y)(E^3_{\lambda_3}/\ker y_3) \subset E^3_{\lambda_3}/\ker y_3$ encodes that $E^3_{\lambda_3}/\ker y_3$ is an eigenspace for ${\dot q}(x,y)$,
whereas in the definition of $\Lambda'''_{C_\KS}$  the condition $({\dot q}(x,y) - \lambda)(E^3_{\lambda_3}/\ker y_3) = 0$  encodes that $E^3_{\lambda_3}/\ker y_3$ is an eigenspace for ${\dot q}(x,y)$ with eigenvalue $\lambda$.
It is thus immediate that the natural map $\Lambda'''_{C_\KS} \rightarrow \Lambda''_{C_\KS}$ is an isomorphism.

We now claim that the natural map $(\sing{\tilde f}_\psi^{-1}(0))\cap \mathcal{\tilde O}_\psi \rightarrow \Lambda^\text{gen}_{C_\KS}\times\operatorname{Gr}^3(E_{\lambda_3})$
defined by projection to $\Lambda^\text{gen}_{C_\KS}$ and the three-dimensional vector space in the partial flag induces an isomorphism to $\Lambda''_{C_\KS}$.
To see this we first note that the map is equivariant.
We next note that by direct inspection the map induces an isomorphism over $(x_\KS,y_\KS(t_1,t_2))$.
Because $\Lambda^\text{gen}_{C_\KS}$ is the $H$ orbit of $(x_\KS,y_\KS(t_1,t_2))$ it follows that the induced map
\[ (\sing{\tilde f}_\psi^{-1}(0))\cap \mathcal{\tilde O}_\psi \rightarrow \Lambda''_{C_\KS} \]
is indeed an isomorphism. 
Now define 
\[
\begin{array}{rcl}
\Lambda'''_{C_\KS} &\to& \widetilde{\Lambda}^\text{gen}_{C_\KS}\\
(x,y,E^3_{\lambda_3},\lambda) &\mapsto& (x,y,\diag(\lambda,\det{\dot q}(x,y)/\lambda)).
\end{array}
\]
This is an $H$-equivariant isomorphism and clearly commutes with the $H$-equivariant projections $\Lambda''_{C_\KS} \to \Lambda^\text{gen}_{C_\KS}$ and $z' : \widetilde{\Lambda}^\text{gen}_{C_\KS}  \to \Lambda^\text{gen}_{C_\KS}$.

Combining the maps above, it follows that 
\[
z' : \widetilde{\Lambda}^\text{gen}_{C_\KS} \ceq  {\Lambda}^\text{gen}_{C_\KS}\times_{Z'_\text{reg}} T'_\text{reg} \to {\Lambda}^\text{gen}_{C_\KS}
\]
is isomorphic to 
\[
\rho_\psi''' : (\sing{\tilde f}_\psi^{-1}(0))\cap \mathcal{\tilde O}_\psi \to \Lambda^\text{gen}_{C_\KS}
\]
as covers $\Lambda^\text{gen}_{C_\KS}$. 
This concludes the proof of Theorem~\ref{thm:geometric}.

\subsection{Speculation on endoscopy for the Kashiwara-Saito representation}


Motivated by the theory of endoscopy, we note that the algebraic group
\begin{equation}\label{def:calS}
\mathcal{S}^\ABV_{\phi_\KS} \ceq Z_{H}(\{(x_\KS,y_\KS(t)) \tq t\in T'_\text{reg}\})
\end{equation}
determines non-trivial group homomorphisms 
\begin{equation}\label{StoA}
\mathcal{S}^\ABV_{\phi_\KS}\to \pi_0(\mathcal{S}^\ABV_{\phi_\KS}) \to A^\ABV_{\phi_\KS},
\end{equation}
defined as follows.
First note that $\mathcal{S}^\ABV_{\phi_\KS}$ has two connected components: the identity component is the group of $h= (h_0,h_1,h_2,h_3,h_4)\in H$ for which $h_0\in T$ and
\[
h_1 =
\begin{pmatrix}
h_0 & u\\
0 & h_0
\end{pmatrix}, \quad
h_2 = 
\begin{pmatrix}
h_0 & 0\\
0 & h_0
\end{pmatrix}, \quad
h_3 =
\begin{pmatrix}
h_0 & v\\
0 & h_0
\end{pmatrix}, \quad
h_4 = h_0;
\]
the other component is the coset represented by
$s_\KS \ceq (s_0,s_1,s_2,s_3,s_4)\in Z_{H}(x_\KS)$ for which $s_0 = \left(\begin{smallmatrix} 0 & 1 \\ -1 & 0 \end{smallmatrix}\right)$ and 
\[
s_1 =
\begin{pmatrix}
s_0 & 0\\
0 & s_0
\end{pmatrix}, \quad
s_2 = 
\begin{pmatrix}
s_0 & 0\\
0 & s_0
\end{pmatrix}, \quad
s_3 =
\begin{pmatrix}
s_0 & 0\\
0 & s_0
\end{pmatrix}, \quad
s_4 = s_0.
\]
As an element of $\GL_{16}(\CC)$, $s_\KS$ is the block matrix with $8$ copies of $s_0$ down the diagonal.
Now define $\pi_0(\mathcal{S}^\ABV_{\phi_\KS}) \to A^\ABV_{\phi_\KS}$ as follows.
The image of the component of $s_\KS$ is the non-trivial automorphism of the cover $z' : \widetilde{\Lambda}^\text{gen}_{C_\KS} \to \Lambda^\text{gen}_{C_\KS}$ defined by
\[
s_\KS : (x_\KS, y_\KS(t), t) \mapsto (x_\KS, y_\KS(t), s_0ts_0^{-1}).
\]

Now, building on \cite{CFMMX}*{Definition 3} and \cite{CFZ:unipotent}*{Definition 4}, consider the distribution
\begin{equation}\label{eqn:Thetas}
\Theta_{\phi_\KS,s} 
\ceq \sum_{\pi\in \Pi^\ABV_{\phi_\KS}(G(F))} (-1)^{\dim(\phi_\KS)-\dim(\phi_\pi)} \trace \langle s , \pi\rangle\ \Theta_\pi
\end{equation}
where we identify $s\in \mathcal{S}^\ABV_{\phi_\KS}$ with its image in $A^\ABV_{\phi_\KS}$ using Equation~\eqref{StoA}.
By Theorem~\ref{thm:geometric} and the work above, if $s_1\in \mathcal{S}^\ABV_{\phi_\KS}$ is in the component of $s_\KS$ then the value of the character $\langle - , \pi\rangle$ at $s_1$ is  $-1$, so
\begin{equation}
\Theta_{\phi_\KS,s_1} = \Theta_{\pi_\KS} - \Theta_{\pi_\psi},
\end{equation}
while if $s_0\in (\mathcal{S}^\ABV_{\phi_\KS})^\circ$ then
\begin{equation}
\Theta_{\phi_\KS,s_0} = \Theta_{\pi_\KS} + \Theta_{\pi_\psi}.
\end{equation}
We ask if $\phi_\KS$ factors as a Langlands parameter through $\dualgroup{G}' \ceq Z_{\GL_{16}(\CC)}(s)$, so $\phi_\KS = r\circ \phi'_{\KS}$ for some Langlands parameter $\phi'_{\KS}$ for $G'$, where $r: Z_{\GL_{16}(\CC)}(s) \hookrightarrow \GL_{16}(\CC)$ is inclusion. 
We further ask if for every $f$ in the Hecke algebra $C^\infty_c(\GL_{16}(F))$, whether it is true that
\[
\Theta_{\phi_\KS,s}(f) = \Theta_{\phi'_\KS}(f'),
\]
where $f'$ is the Langlands-Shelstad transfer to $C^\infty_c(G'(F))$ of $f$?
If the spectral endoscopic transfer is given by parabolic induction in this case, then this seems to be impossible for $s\in \mathcal{S}^\ABV_{\phi_\KS}$ in the component of $s_\KS$.  We thank the referee for this important observation.


\let\normalsize\small
   \appendix
 \small

\section{Macaulay2 instructions}\label{sec:computations}

The purpose of this appendix is to present the Macaulay2 code used in the proof of the geometric main result, Theorem~\ref{thm:geometric}. This code is available from the GitHub repository \href{https://github.com/CliftonCunningham/Voganish-Project}{Voganish-Project},
 which includes:  
\begin{enumerate}
\item[] \code{KSrepresentations.m2},
\item[] \code{ComputeDuals.m2}, 
\item[] \code{NetworkFlow.m2}, 
\item[] \code{ComputeCover.m2}, 
\item[] \code{VanishingCycles.m2}, 
\item[] \code{VoganV.m2}, 
\item[] \code{PSNF.m2} (\code{partialSmithNormalForm}). 
\end{enumerate}
Detailed descriptions of each function used in this appendix can be found in these files. 
Our code for \code{ComputeDuals.m2} uses \code{NetworkFlow.m2} which, as the title suggests, relies on ideas from the theory of network flows; in particular, it does not use the Moeglin-Waldspurger algorithm \cite{MW:involution} for the corresponding duality on multisegments.
We remark that $V$ is an example of a Vogan variety; see \cite{CFMMX}*{Section 4.2} for the definition. There are several functions in the files listed above, some of which we do not describe in this appendix, that can be applied to arbitrary Vogan varieties.


\subsection{General computations}\label{ssec:generalcomputations}

In this section, we introduce functions that perform the following tasks: list all rank triangles/orbits of a Vogan variety; display an orbit in terms of its rank conditions; compute the dual of an orbit; compute the dimension of an orbit and its closure; and determine if an orbit is contained in the closure of another. We also provide a proof of Lemma~\ref{lem:Kittsix}. The code in this section is particularly relevant to Sections~\ref{ssec:ranktriangles}, \ref{ssec:Conormalbundle}, and \ref{ssec:reduction}. Note that all functions used in this section can be applied to an arbitrary Vogan variety.

We obtain a complete list of rank triangles/$H$-orbits for $V$ using the eigenspace dimensions $(2,4,4,4,2)$ of $\lambda_{\KS}(\text{Fr})$:\\

\indent \code{i1 : L := getAllStrataR(\{2,4,4,4,2\});}\\

\noindent We store an $H$-orbit of $V$ as follows:\\

\indent \code{i2 : CKS := new RankConditions from (\{2,4,4,4,2\},\{\{2,2,2,2\},\{0,2,0\},\{0,0\},\{0\}\});}\\

\indent \code{i3 : Cpsi := new RankConditions from (\{2,4,4,4,2\},\{\{2,3,3,2\},\{1,2,1\},\{1,1\},\{0\}\});}\\

\indent \code{i4 : CL := new RankConditions from (\{2,4,4,4,2\},\{\{2,4,2,2\},\{2,2,0\},\{0,0\},\{0\}\});}\\

\noindent We compute the duals of our $H$-orbits:\\

\indent \code{i5 : computeDual CKS}\\
\

\indent \code{o5 = 2 \ 4 \ 4 \ 4 \ 2}\\
\hspace*{1.5cm} \code{2 \ 2 \ 2 \ 2}\\
\hspace*{1.9cm}\code{0 \ 2 \ 0}\\
\hspace*{2.2cm}\code{0 \ 0}\\
\hspace*{2.5cm}\code{0}\\

\indent \code{i6 : computeDual Cpsi}\\
\

\indent \code{o6 = 2 \ 4 \ 4 \ 4 \ 2}\\
\hspace*{1.5cm} \code{2 \ 3 \ 3 \ 2}\\
\hspace*{1.9cm}\code{1 \ 2 \ 1}\\
\hspace*{2.2cm}\code{1 \ 1}\\
\hspace*{2.5cm}\code{0}\\

\indent \code{i7 : CR := computeDual CL}\\
\

\indent \code{o7 = 2 \ 4 \ 4 \ 4 \ 2}\\
\hspace*{1.5cm} \code{2 \ 2 \ 4 \ 2}\\
\hspace*{1.9cm}\code{0 \ 2 \ 2}\\
\hspace*{2.2cm}\code{0 \ 0}\\
\hspace*{2.5cm}\code{0}\\

\noindent We form an ideal consisting of polynomial equations that describes the closure of an $H$-orbit:\\

\indent \code{i8: getEquations CKS;}\\

\noindent Recall that the dimension of an $H$-orbit coincides with the dimension of its closure. Thus, the dimension of an $H$-orbit can be obtained from the ideal that describes its closure:\\

\indent \code{i9: dim o8}\\

\indent \code{o9 = 32}\\

\indent \code{i10: dim getEquations Cpsi}\\

\indent \code{o10 = 40}\\

\indent \code{i11: dim getEquations CR}\\

\indent \code{o11 = 36}\\

\noindent Given any two orbits $C'$ and $C$, with $C'\neq C$, we can check if $C'$ is in the closure of $C$:\\

\indent \code{i12: CKS < Cpsi;}\\

\indent \code{o12 = true}\\

\noindent Here, we use the order defined in Section~\ref{ssec:ranktriangles} by $C' < C$ if and only if $C'\subsetneq \overline{C}$. We compute the list of $H$-orbits $C$ of $V$ that satisfy $C_{\KS}< C$ and $C_{\KS}< \widehat{C}$, thus providing a proof of Lemma~\ref{lem:Kittsix}:\\

\indent \code{i12: orbits6 := \{\};}\\

\indent \code{i13: for C in L do (if CKS < C and CKS < (computeDual C) then \\
\hspace*{4.1cm} orbits6 = append(orbits6, C))}\\

\indent \code{i14: orbits6}\\

\indent \code{o15 = \{2 \ 4 \ 4 \ 4 \ 2, 2 \ 4 \ 4 \ 4 \ 2, 2 \ 4 \ 4 \ 4 \ 2, 2 \ 4 \ 4 \ 4 \ 2,}\\
\hspace*{1.9cm} \code{2 \ 2 \ 3 \ 2 \ \ \ \ 2 \ 2 \ 4 \ 2 \ \ \ \ 2 \ 3 \ 2 \ 2 \ \ \ \ 2 \ 3 \ 3 \ 2}\\
\hspace*{2.3cm}\code{0 \ 2 \ 1 \ \ \ \ \ \ \ 0 \ 2 \ 2 \ \ \ \ \ \ \ 1 \ 2 \ 0 \ \ \ \ \ \ \ 1 \ 2 \ 1}\\
\hspace*{2.6cm}\code{0 \ 0 \ \ \ \ \ \ \ \ \ \ 0 \ 0 \ \ \ \ \ \ \ \ \ \ 0 \ 0 \ \ \ \ \ \ \ \ \ \ 0 \ 0}\\
\hspace*{2.9cm}\code{0 \ \ \ \ \ \ \ \ \ \ \ \ \ 0 \ \ \ \ \ \ \ \ \ \ \ \ \ 0 \ \ \ \ \ \ \ \ \ \ \ \ \ 0}\\
\indent\indent\indent \code{--------------------------------------------------------------------------------------------------------------------------------}\\
\indent\indent\indent\indent \code{2 \ 4 \ 4 \ 4 \ 2, 2 \ 4 \ 4 \ 4 \ 2\}}\\
\hspace*{1.9cm} \code{2 \ 3 \ 3 \ 2 \ \ \ \ 2 \ 4 \ 2 \ 2 }\\
\hspace*{2.3cm}\code{1 \ 2 \ 1 \ \ \ \ \ \ \ 2 \ 2 \ 0 }\\
\hspace*{2.6cm}\code{1 \ 1 \ \ \ \ \ \ \ \ \ \ 0 \ 0 }\\
\hspace*{2.9cm}\code{0 \ \ \ \ \ \ \ \ \ \ \ \ \ 0}\\

\noindent This computation, together with Proposition~\ref{prop:ordering}, allows us to reduce Theorem~\ref{thm:geometric} to Equations~\eqref{r}, \eqref{m}, \eqref{R} and \eqref{psi}.

%
%

\subsection{Microlocal vanishing cycle calculations}
\label{ssec: MVC appendix}

In this section, we illustrate how to reproduce the Macaulay2 computations used in the proof of Equations~\eqref{r}, \eqref{m}, \eqref{R} and \eqref{psi}. We introduce functions that perform the following tasks: verify that our covers are semi-small; generate the relevant Jacobians needed in the proofs;  substitute particular points into these Jacobians; and perform elementary row and column operations on matrices over polynomial rings. The code in this section is relevant to Sections~\ref{ssec:overview} - \ref{ssec:case-psi}.

All functions in this section, with the exception of \code{partialSmithNormalForm} (\code{PSNF.m2}), are specific to $V$. Throughout this appendix $t_1$ refers to \code{y\_\{2,0,2\}} in the code, and $t_2$ to \code{y\_\{2,1,3\}}.

\subsubsection{$\Evs_{C_{\KS}}\IC(\1_{C_r})$}
\label{Assec: Cr}

The computations in this section are specific to Section~\ref{ssec:case-r}. 
The equations that describe the affine chart $U\subseteq \widetilde{C}_r$ are the generators of the ideal \code{coverCr()}. One can optionally verify that the cover is semi-small using \code{smallCr()}. The output lists the relevant strata for the cover. We compute the Jacobian for the equations that describe $U\times C_{\KS}^\ast$ and $\tilde{f}_r^{-1}(0)$:\\

\indent \code{i1 : JacCr();}\\

\noindent Next, we evaluate the Jacobian \code{o1} at the generic point $(x_\KS,y_\KS(t_1,t_2))$, with $\diag(t_1,t_2)\in T'_\text{reg}$, and further add the conditions that describe $U\cap\rho_r^{-1}(x_{\KS})$: \\

\indent \code{i2: subJacCr(o1);}\\

\noindent The resulting matrix \code{o2} is over $\mathbb{Q}[t_1,t_2,c_3]$.
We perform elementary row and column operations on \code{o2}: \\

\indent \code{i3: partialSmithNormalForm(o2);}\\

\noindent The output \code{o3} is a $45\times 45$ block diagonal matrix whose blocks consist of a $43\times 43$ identity matrix and the matrix
\[
\begin{bmatrix}
t_1c_3+c_3^2 & -t_1-c_3\\
t_2c_3^2 & -t_2c_3
\end{bmatrix}.
\]

Recall that the fibre of $\rho_r$ above $C_{\KS}$ is $\mathbb{P}^1\cong \{[a:b]\}$. The above calculations are for verifying smoothness in the affine chart $U$ that contains the points $\{[1:b]\}$. We must also check smoothness in an affine chart for $\widetilde{C}_r$ that contains the points $\{[a:1]\}$. To do this, we use the affine chart $U'$ for $\widetilde{C}_r$, which is different from the one used in \code{JacCr()}, corresponding to:
\[
E_{\lambda_1}^2=\begin{bmatrix}
1 & 0\\
0 & 1\\
a_1 & a_2\\
a_3 & a_4
\end{bmatrix},\quad
E_{\lambda_2}^1=\begin{bmatrix}
b_1\\
b_2\\
b_3\\
1
\end{bmatrix}, \quad
E_{\lambda_2}^3=\begin{bmatrix}
1 & 0 & 0\\
0 & 1 & 0\\
c_1 & c_2 & c_3\\
0 & 0 & 1
\end{bmatrix},\quad
E_{\lambda_3}^2=\begin{bmatrix}
1 & 0\\
0 & 1\\
d_1 & d_2\\
d_3 & d_4
\end{bmatrix}.
\]

\noindent We then continue with the same approach illustrated above. Compute the Jacobian for the equations that describe $U'\times C_{\KS}^\ast$ and $\tilde{f}_r^{-1}(0)$:\\ 

\indent \code{i4: JacCr2();}\\

\noindent We evaluate the Jacobian \code{o4} at the generic point $(x_\KS,y_\KS(t_1,t_2))$, with $\diag(t_1,t_2)\in T'_\text{reg}$, and further add the conditions that describe $U'\cap\rho_r^{-1}(x_{\KS})$: \\

\indent \code{i5: subJacCr(o4);}\\

\noindent The resulting matrix \code{o5} is over $\mathbb{Q}[t_1,t_2]$.
We perform  elementary row and column operations on \code{o5}:\\

\indent \code{i6: partialSmithNormalForm(o5);}\\

\noindent The output \code{o6} is a $45\times 45$ block diagonal matrix whose blocks consist of a $43\times 43$ identity matrix and the matrix
\[
\begin{bmatrix}
-t_1c_3-1 & t_1c_3^2+c_3\\
-t_2 & t_2c_3
\end{bmatrix}.
\]
The rank of $\code{o6}$ is $43$ if and only if $-t_1c_3-1=0$ and $-t_2=0$. But, $t_2\neq 0$ since $\diag(t_1,t_2)\in T'_{\text{reg}}$. And so the rank of \code{o6} is $44$. Since the rank of \code{o6} coincides with the generic rank of \code{o4}, we are smooth.

\subsubsection{$\Evs_{C_{\KS}}\IC(\1_{C_m})$}
\label{Assec: Cm}

The computations in this section are specific to Section~\ref{ssec:case-m}.
The equations that describe the affine chart $U\subseteq \widetilde{C}_m$ are the generators of the ideal \code{coverCm()}. One can optionally verify that the cover is semi-small using \code{smallCm()}.  The output lists the relevant strata for the cover. We compute the Jacobian for the equations that describe $U\times C_{\KS}^\ast$ and $\tilde{f}_m^{-1}(0)$:\\

\indent \code{i1 : JacCm();}\\

\noindent Next, we evaluate the Jacobian \code{o1} at the generic point $(x_\KS,y_\KS(t_1,t_2))$, with $\diag(t_1,t_2)\in T'_\text{reg}$, and further add the conditions that describe $U\cap\rho_m^{-1}(x_{\KS})$: \\

\indent \code{i2: subJacCm(o1);}\\

\noindent The resulting matrix \code{o2} is over $\mathbb{Q}[t_1,t_2,c_3,g]$. 
 We perform elementary row and column operations on \code{o2}:\\

\indent \code{i3: partialSmithNormalForm(o2);}\\

\noindent The output \code{o3} is a $48\times 47$ block diagonal matrix whose blocks consist of a $45\times 45$ identity matrix and the matrix

\[
\begin{bmatrix}
c_3-g & -c_3+g\\
-t_1-c_3 & t_1+c_3\\
-t_2c_3 & t_2c_3
\end{bmatrix}.
\]\\
Let $p : \mathbb{P}^1\times\mathbb{P}^1 \to \rho_m^{-1}(x_\KS)$ be the isomorphism $p([a:b],[c:d])\! =\! (E_{\lambda_1}^2, E_{\lambda_2}^1, E_{\lambda_2}^3, E_{\lambda_3}^2, E_{\lambda_4}^1)$ where
\[
E_{\lambda_1}^2 = \begin{bmatrix}
1 & 0\\
0 & 1\\
0 & 0\\
0 & 0
\end{bmatrix},\quad
E_{\lambda_2}^1=\begin{bmatrix}
0\\
0\\
a\\
b
\end{bmatrix} ,\quad
E_{\lambda_2}^3=\begin{bmatrix}
1 & 0 & 0\\
0 & 1 & 0\\
0 & 0 & a\\
0 & 0 & b
\end{bmatrix},
\]
\[
E_{\lambda_3}^2=\begin{bmatrix}
1 & 0\\
0 & 1\\
0 & 0\\
0 & 0
\end{bmatrix}, \quad E_{\lambda_3}^3=\begin{bmatrix}
1 & 0 & 0\\
0 & 1 & 0\\
0 & 0 & c\\
0 & 0 & d
\end{bmatrix},\quad
E_{\lambda_4}^1=\begin{bmatrix}
c\\
d
\end{bmatrix}.
\]
The above calculations are for verifying smoothness in the affine chart $U$ that contains the points $p([1:b], [1:d])$. We use the same method as above to verify smoothness in a collection of affine charts for $\widetilde{C}_m$ that, together with $U$, cover $\widetilde{C}_m\cap\rho_m^{-1}(x_{\KS})$. For each of these charts, we list below the relevant points of the fibre that it contains, together with the corresponding Jacobian and analogous function  to \code{subJacCm()} needed for the computation:

\begin{enumerate}

\item[(2)] $p([1:b],[c:1])$ use \code{JacCm2()} and \code{subJaCm};

\item[(3)] $p([a:1],[1:d])$ use \code{JacCm3()} and \code{subJaCm};

\item[(4)] $p([a:1],[c:1])$ use \code{JacCm4()} and \code{subJaCm}.

\end{enumerate}

\subsubsection{$\Evs_{C_{\KS}}\IC(\1_{C_R})$}
\label{Assec: CR}
The computations in this section are specific to Section~\ref{ssec:case-R}. 
The equations that describe the affine chart $U\subseteq \widetilde{C}_R$ are the generators of the ideal \code{coverCR()}. One can optionally verify that the cover is small using \code{smallCR()}.  The output lists the relevant strata for the cover. We compute the Jacobian for the equations that describe $U\times C_{\KS}^\ast$ and $\tilde{f}_R^{-1}(0)$:\\

\indent \code{i1 : JacCR();}\\

\noindent Next, we evaluate the Jacobian \code{o1} at the generic point $(x_{\KS},y_{\KS}(t_1,t_2))$, with $\diag(t_1,t_2)\in T'_{\text{reg}}$, and further add the conditions that correspond to $\rho_R(x_{\KS},E)=x_{\KS}$:\\

\indent \code{i2: subJacCR(o1);}\\

\noindent The resulting matrix \code{o2} is over $\mathbb{Q}[t_1, t_2]$. We perform elementary row and column operations on \code{o2}:\\

\indent \code{i3 : partialSmithNormalForm(o2);}\\

\noindent The output \code{o3} is a $41\times 41$ identity matrix.

\subsubsection{$\Evs_{C_{\KS}}\IC(\1_{C_\psi})$}
\label{Assec: Cpsi}

The computations in this section are specific to Section~\ref{ssec:case-psi}.
The equations that describe the affine chart $U\subseteq \widetilde{C}_\psi$ are the generators of the ideal returned by \code{coverCpsi()}. In this case we must verify that our cover $\rho_\psi:\widetilde{C}_\psi\rightarrow\overline{C}_\psi$ is semi-small:\\

\indent \code{i1 : smallCpsi()}\\

\indent \code{o1 = (\{2 \ 4 \ 4 \ 4 \ 2, 2 \ 4 \ 4 \ 4 \ 2\}, \{\}, \{\})}\\
\hspace*{1.9cm} \code{2 \ 2 \ 2 \ 2 \ \ \ \ 2 \ 3 \ 3 \ 2}\\
\hspace*{2.3cm}\code{0 \ 2 \ 0 \ \ \ \ \ \ \ 1 \ 2 \ 1 }\\
\hspace*{2.6cm}\code{0 \ 0 \ \ \ \ \ \ \ \ \ \ 1 \ 1 }\\
\hspace*{2.9cm}\code{0 \ \ \ \ \ \ \ \ \ \ \ \ \ 0}\\

\noindent The second list in \code{o1} tells us that
\[
2\dim(\rho_\psi^{-1}(x)) + \dim C_i \leq \dim\widetilde{C}_\psi, \quad x\in C_i
\]
for all orbits $C_i\leq C_\psi$. Equality is obtained if $C_i$ is one of the orbits in the first list of \code{o1}. Together, these two statements tell us that $\rho_\psi:\widetilde{C}_\psi\rightarrow\overline{C}_\psi$ is semi-small. Next, we compute the Jacobian for the equations that describe $U\times C_{\KS}^\ast$ and $\tilde{f}_\psi^{-1}(0)$:\\

\indent \code{i2 : JacCpsi();}\\

\noindent We evaluate the Jacobian \code{o2} at the generic point $(x_\KS,y_\KS(t_1,t_2))$, with $\diag(t_1,t_2)\in T'_\text{reg}$, and further add the conditions that describe $U\cap\rho_\psi^{-1}(x_{\KS})$: \\

\indent \code{i3: subJacCpsi(o2);}\\

\noindent The resulting matrix \code{o3} is over $\mathbb{Q}[t_1,t_2,b_1,b_2,c_3,f_3]$. 
We perform elementary row and column operations on \code{o3}:\\

\indent \code{i4: partialSmithNormalForm(o3);}\\

\noindent To make \code{o4} easier to read, simply compute \code{matrix o4}. The output \code{o4} is a $52\times 51$ block diagonal matrix whose blocks consist of a $45\times 45$ identity matrix and the matrix \\
\[
B_\psi = \begin{bmatrix}
    \vdots & \vdots & \vdots & \vdots & \vdots &\vdots\\
    v   & b_1v & f_3v &  b_1f_3v & v & f_3v\\
    \vdots & \vdots & \vdots & \vdots & \vdots & \vdots
\end{bmatrix},\quad
\text{with}\quad
v=\begin{bmatrix}
-b_1+f_3\\
-b_1+f_3\\
-c_3 + f_3\\
t_1b_2 + b_2c_3 +1\\
t_2b_2c_3 + b_1\\
t_1b_2+b_2c_3 +1\\
t_2b_2c_3 + b_1
\end{bmatrix}.
\]

Recall that the fibre of $\rho_\psi$ above $x_{\KS}$ is the product of $\mathbb{P}^1\cong \{ [e:h] \} $ with the variety
\[
F\ceq \{([a:b], [c:d], [u:v:w:z]) \in \mathbb{P}^1\times \mathbb{P}^1\times \mathbb{P}^3 \tq av-bu=0,\ cz-dw=0 \},
\]
which is the blow up of  $\mathbb{P}^3$ along the two copies of $\mathbb{P}^1$ given by $(u,v)=(0,0)$ and $(w,z)=(0,0)$.
For the calculations below, we find it convenient to use the map $p : (\mathbb{P}^1)^{\times 4} \to  \rho_\psi^{-1}(x_\KS)$ defined by 
\[
p([a:b], [c:d], [e:h], [\beta:\alpha]) = (E_{\lambda_1}^2,E_{\lambda_2}^1,E_{\lambda_2}^3,E_{\lambda_3}^1,E_{\lambda_3}^2,E_{\lambda_3}^3,E_{\lambda_4}^1)
\]
where
\[
E_{\lambda_1}^2=\begin{bmatrix}
1 & 0\\
0 & 1\\
0 & 0\\
0 & 0
\end{bmatrix},\quad
E_{\lambda_2}^1=\begin{bmatrix}
\beta a\\
\beta b\\
\alpha c\\
\alpha d
\end{bmatrix}, \quad
E_{\lambda_2}^3=\begin{bmatrix}
1 & 0 & 0\\
0 & 1 & 0\\
0 & 0 & c\\
0 & 0 & d
\end{bmatrix},
\]
\[
E_{\lambda_3}^1=\begin{bmatrix}
a\\
b\\
0\\
0
\end{bmatrix}, \quad
E_{\lambda_3}^2=\begin{bmatrix}
1 & 0\\
0 & 1\\
0 & 0\\
0 & 0 
\end{bmatrix}, \quad
E_{\lambda_3}^3=\begin{bmatrix}
1 & 0 & 0\\
0 & 1 & 0\\
0 & 0 & e\\
0 & 0 & h
\end{bmatrix},\quad
E_{\lambda_4}^1=\begin{bmatrix}
e\\
h
\end{bmatrix}.
\]
The calculations above are for finding singularities in the affine chart $U$ that contains $p([1:b],[1:d],[1:h],[1:\alpha])$; these calculations show that the intersection of $\sing(\tilde{f}_\psi^{-1}(0))$ with $(U\times C_{\KS}^\ast)\cap(\rho_\psi')^{-1}(x_\KS,y_\KS(t_1,t_2))$ is
\[
\{ 
p([1:0],[1:0],[1:0],[1:-1/t_1]), p([1:t_2-t_1],[1:t_2-t_1],[1:t_2-t_1],[1:-1/t_2])\}.
\]
Notice that $[1:0]$ corresponds to the eigenspace of $\left(\begin{smallmatrix} t_1 & 1 \\ 0 & t_2 \end{smallmatrix}\right)$ for the eigenvalue $t_1$ and $[1:t_2-t_1]$ corresponds to the eigenspace for the eigenvalue $t_2$. 

We use the same method as above to check for singularities in a collection of affine charts that, together with $U$, cover $\widetilde{C}_\psi\cap\rho_\psi^{-1}(x_{\KS})$. In all cases, we find no new singularities. For each of these charts, we list below the relevant points of the fibre that it contains, together with the corresponding Jacobian and analogous function  to \code{subJacCpsi()} needed for the computation:

\begin{enumerate}

\item[(2)] 
$p([1:b],[1:d],[1:h],[\beta:1])$ use \code{JacCpsi2()} and \code{subJacCpsi2()};

\item[(3)] 
$p([1:b],[1:d],[e:1],[1:\alpha])$ use \code{JacCpsi3()} and \code{subJacCpsi3()};

\item[(4)] 
$p([1:b],[1:d],[e:1],[\beta:1])$ use \code{JacCpsi4()} and \code{subJacCpsi4()};

\item[(5)] 
$p([1:b],[c:1],[1:h],[1:\alpha])$ use \code{JacCpsi5()} and \code{subJacCpsi5()};

\item[(6)] 
$p([1:b],[c:1],[1:h],[\beta:1])$ use \code{JacCpsi6()} and \code{subJacCpsi6()};

\item[(7)] 
$p([1:b],[c:1],[e:1],[1:\alpha])$ use \code{JacCpsi7()} and \code{subJacCpsi7()};

\item[(8)] 
$p([1:b],[c:1],[e:1],[\beta:1])$ use \code{JacCpsi8()} and \code{subJacCpsi8()};

\item[(9)] 
$p([a:1],[1:d],[1:h],[1:\alpha])$ use \code{JacCpsi9()} and \code{subJacCpsi9()};

\item[(10)] 
$p([a:1],[1:d],[1:h],[\beta:1])$ use \code{JacCpsi10()} and \code{subJacCpsi10()};

\item[(11)] 
$p([a:1],[1:d],[e:1],[1:\alpha])$ use \code{JacCpsi11()} and \code{subJacCpsi11()};

\item[(12)] 
$p([a:1],[1:d],[e:1],[\beta:1])$ use \code{JacCpsi12()} and \code{subJacCpsi12()};

\item[(13)] 
$p([a:1],[c:1],[1:h],[1:\alpha])$ use \code{JacCpsi13()} and \code{subJacCpsi13()};

\item[(14)] 
$p([a:1],[c:1],[1:h],[\beta:1])$ use \code{JacCpsi14()} and \code{subJacCpsi14()};

\item[(15)] 
$p([a:1],[c:1],[h:1],[1:\alpha])$ use \code{JacCpsi15()} and \code{subJacCpsi15()};

\item[(16)] 
$p([a:1],[c:1],[e:1],[\beta:1])$ use \code{JacCpsi16()} and \code{subJacCpsi16()}.

\end{enumerate}

\subsection{Local Hessian calculations}\label{app:Hessian}

In this section we document how to reproduce our calculation to find the Hessians and Hessian determinants.
The code which performs our Hessian calculations can all be found in \code{KSHess.m2}.
We first focus on the case relevant for Section \ref{ssec:case-psi}.
First use the function\\

\indent\code{i1 : eqcpsi = EqCpsi();}\\

\noindent to compute the defining ideals and functions $\tilde{f}_\psi$ for the cover $\tilde{\mathcal{O}}_\psi$ for $C_\psi$ on which we wish to compute the Hessian.
Next, the function\\

\indent\code{i2 : locvar = getLocalCoordList( eqcpsi\#1, subXCpsi );}\\

\noindent computes a system of local coordinates which generate the completed local ring in an open neighbourhood of the point described by \code{subXCpsi}.
Then the function\\

\indent\code{i3 : impvarpart = getImpVarPartials( eqcpsi\#1 , locvar, subXCpsi );}\\

\noindent computes the partial derivatives of the implicit functions in terms of the chosen local coordinates. The resulting formulas are valid in an open neighbourhood of the point described by \code{subXCpsi}.
Now the function \\

\indent\code{i4 : hess = getHessian(eqcpsi\#0, locvar, impvarpart);}\\

\noindent computes the hessian of $\tilde{f}_\psi$ in terms of the chosen local coordinates using the precomputed formulas for the partial derivatives for the implicit functions.
At this point we may use\\

\indent\code{i5 : rnk = rank( subXCpsi( hess ) ) }\\

\indent\code{o5 = 24}\\

\noindent to obtain the rank of the Hessian at the point described by \code{subXCpsi}. This rank will be locally constant on an open neighbourhood of \code{subXCpsi} on the singular locus of $\tilde{f}_\psi$.
We obtain a rank of $24$.
Now the function\\

\indent\code{i6 : subhess = getSubHessian( hess, subXCpsi );}\\

\noindent returns a minor of the Hessian with non-zero determinant which has maximal rank on an open neighbourhood of \code{subXCpsi} on the singular locus of $\tilde{f}_\psi$.
Finally the function\\

\indent\code{i7 : iso=findIso(subhess\#0, subXCpsi);}\\

\noindent optimistically searches for a large isotropic subspace. The  subspace it finds is isotropic on an open neighbourhood of \code{subXCpsi} on the singular locus of $\tilde{f}_\psi$. 
 If \code{ numcolumns(iso\#0) } is half of \code{ rnk } then the Hessian determinant is a square on an open neighbourhood of \code{subXCpsi} on the singular locus of $\tilde{f}_\psi$. 
We find a $12$ dimensional isotropic subspace

To perform the calculations for Section \ref{ssec:case-KS} relevant to verifying that the character of the local system $\Evs_{C_\KS}\IC(\1_{C_\KS})$ is trivial simply use
\code{EqCk();} and \code{subXCks} in the above. 
In this case the Hessian has rank $16$ and the isotropic subspace has dimension $8$.

\normalsize
 

\end{document}